\documentclass[a4paper,11pt]{amsart}
\usepackage{amsfonts,amssymb,amsthm,cite,amsmath,amstext}
\usepackage{color}
\usepackage[colorlinks,linkcolor=black,citecolor=black]{hyperref}
\usepackage{comment}
\usepackage{framed}
\usepackage{calligra} 
\usepackage{tikz-cd}
\usepackage{enumitem}
\usepackage{makecell}

\definecolor{shadecolor}{gray}{0.875}

\newtheorem{thrm}{Theorem}[section]

\newtheorem{lem}[thrm]{Lemma}
\newtheorem{cor}[thrm]{Corollary}
\newtheorem{prop}[thrm]{Proposition}

\theoremstyle{definition}
\newtheorem{defn}[thrm]{Definition}

\newtheorem{example}[thrm]{Example}
\newtheorem{rmk}[thrm]{Remark}

\DeclareMathOperator{\Gr}{Gr}

\DeclareMathOperator{\Image}{Image}

\DeclareMathOperator{\nd}{nd}

\DeclareMathOperator{\End}{End}
\DeclareMathOperator{\Span}{span}

\DeclareMathOperator{\codim}{codim}

\DeclareMathOperator{\prim}{Prime}
\DeclareMathOperator{\Irred}{Irred}

\DeclareMathOperator{\IC}{IC}
\DeclareMathOperator{\codimension}{codim}

\title{On the kernels of Lefschetz operators}
\author{Jiajun Hu}
\date{}

\begin{document}


\begin{abstract}
We study the kernels of Lefschetz operators arising from complete intersections of nef classes. 
When the classes are semiample, we provide a complete characterization of mixed
Lefschetz properties below a fixed degree -- all obstructions are detected by irreducible subvarieties of relevant codimensions.
We further show that the kernels in the next degree range vanish off the diagonal, 
while on the diagonal they admit bases consisting of fundamental classes of irreducible subvarieties annihilated by the Lefschetz operator.
We also establish a kernel decoupling theorem for nef classes
on complex Lefschetz modules.
As an application, we obtain a kernel characterization for supercritical collections of nef classes, thereby proving a conjecture of Shenfeld-van Handel and Hu-Xiao.

\end{abstract}

\maketitle

\tableofcontents

\section{Introduction}

Let $X$ be an $n$-dimensional smooth projective variety over $\mathbb{C}$, and let
$L\in H^{1,1}(X,\mathbb{R})$ be a real $(1,1)$-class. The class $L$ naturally acts on the cohomology group $H^*(X,\mathbb{C})$ by cup product. We will not distinguish between $L$ and the corresponding endomorphism.
If $L\in\mathcal{K}_X$ is a K\"ahler class, classical Hodge theory provides two fundamental properties: the hard Lefschetz property and the Hodge-Riemann bilinear relations. 

\noindent \textbf{(HL)} 
For all integers $p,q\geq 0$ with $p+q\leq n$, the linear map
\begin{equation*}
    L^{n-p-q}:H^{p,q}(X)\to H^{n-q,n-p}(X)
\end{equation*}
is an isomorphism. 

\noindent \textbf{(HR)} The Hermitian form
\begin{equation*}
    H^{p,q}(X)\times H^{p,q}(X)\to \mathbb{C}\ \ \ (x,y)\mapsto (\sqrt{-1})^{p-q}(-1)^{\frac{(p+q)(p+q-1)}{2}} \int_X x\cdot \overline{y}\cdot L^{n-p-q}
\end{equation*}
is positive definite on the primitive space
\begin{equation*}
    P^{p,q}(X)=\ker(L^{n-p-q+1}:H^{p,q}(X)\to H^{n-q+1,n-p+1}(X)).
\end{equation*}

When $L$ degenerates to the boundary of the K\"ahler cone, these properties need no longer hold. For instance, if $L\in\overline{\mathcal{K}_X}$ is nef, then the map
$L^{n}:H^{0,0}(X)\to H^{n,n}(X)$
is an isomorphism if and only if $L$ is big. 
However, a complete characterization appears to be considerably more difficult in other bidegrees.
A remarkable result in this direction is due to de Cataldo-Migliorini \cite{decataldoLefsemismall}, who completely characterized when the first Chern class of a semiample line bundle satisfies the hard Lefschetz property in all degrees. Their result was subsequently extended in \cite{xiaomixedHRR,hushangxiao2023hardlefArxiv}.

Throughout this manuscript, by a slight abuse of notation, we use the same symbol $L$ for a line bundle and its first Chern class $c_1(L)$.
We also use additive notation for line bundles; thus $L+M$ denotes $L\otimes M$.
We recall the above characterization in a form suited to our purposes.

\begin{thrm}[de Cataldo-Migliorini, Xiao, Hu-Shang-Xiao]
    Let $0\leq d\leq n$ be an integer, and 
    let $L$ be a semiample line bundle. Then the following statements are equivalent.
    \begin{enumerate}
        \item For every subvariety $V\subset X$ with $\codimension_X V\le           \frac{d}{2}$, we have
                \begin{equation*}
                    L^{d-2\codimension_X V}\cdot [V]\neq 0.
                \end{equation*}
        \item For every integer $k$ with $0\leq k\leq \frac{d}{2}$, the map
        \begin{equation*}
            L^{d-2k}:H^{k,k}(X)\to H^{d-k,d-k}(X)
        \end{equation*}
        is injective.
        \item For all integers $p,q\geq 0$ with $ p+q\leq d$, the map 
        \begin{equation*}
            L^{d-p-q}:H^{p,q}(X)\to H^{d-q,d-p}(X)
        \end{equation*}
        is injective.
    \end{enumerate}
\end{thrm}

Roughly speaking, this theorem says that the obstructions to the Lefschetz property of a semiample class are detected by algebraic cycles. When $d=n$, the result is due to de Cataldo-Migliorini \cite{decataldoLefsemismall}, while the general case is due to Xiao \cite{xiaomixedHRR} and Hu-Shang-Xiao \cite{hushangxiao2023hardlefArxiv}. 
A semiample line bundle $L$ is called $d$-lef if it satisfies the first of the equivalent conditions above. The terminology reflects the fact that $L$ satisfies the Lefschetz property in the corresponding range of degrees.

The hard Lefschetz property and the Hodge-Riemann relations
also admit mixed versions. 
Establishing these properties for products of K\"ahler classes is already nontrivial.
The case of bidegree $(1,1)$  goes back to
Gromov \cite{gromov1990convex}, while the general case was
established by Dinh-Nguy\^en \cite{DN06}.
Cattani later gave another proof based on the descent
lemma \cite{cattanimixedHRR}. 
Lefschetz properties have also been established for
characteristic classes of ample vector bundles, including
top Chern classes and Schur classes in suitable bidegrees.
These classes need not be complete intersections.
For these results and related developments, we refer
to \cite{blochgieseker,ross2019hodge,rossHRkahler,zhengHRpoly}
and the references therein.

In these results, the Lefschetz property is obtained from
K\"ahler classes or ample vector bundles.
For complete intersections, a complementary problem is to characterize
the Lefschetz property when the factors are only nef.
On compact complex tori, such a characterization was obtained in \cite{hxhardlef}
in terms of the numerical dimensions.
This numerical criterion motivated subsequent
work
\cite{hushangxiao2023hardlefArxiv,HuXiaoLefChar,HuXiao25LefCharFree}.

Another motivation comes from the equality characterization of the Khovanskii-Teissier inequality in algebraic geometry. Related problems are studied in 
For example, let $A,B,L_1,\ldots,L_{n-2}$ be nef classes on $X$.
Write $\mathcal{L}=L_1\cdots L_{n-2}$.
The Khovanskii-Teissier inequality states that
\begin{equation*}
        (\int_X A\cdot B\cdot\mathcal{L})^2 \geq  \int_X A^2\cdot\mathcal{L} \int_X B^2\cdot\mathcal{L}.
\end{equation*}
If $\int_X A\cdot B\cdot\mathcal{L}>0$, then equality holds if and only if there exists
a real constant $c>0$ such that
\begin{equation*}
    A-cB\in\ker\mathcal{L}\cap H^{1,1}(X,\mathbb{R}).
\end{equation*}
Thus a geometric description of this kernel yields a
characterization of the equality cases.

A related question is Teissier's proportionality problem, which asserts that 
two big and nef classes are proportional if and only if equality holds in the Brunn-Minkowski inequality or equivalently simultaneously in all the log-concavity inequalities. This was proved by Boucksom-Favre-Jonsson in the projective setting and Fu-Xiao in the K\"ahler setting, and further generalized to strictly Lorentzian classes by Hu-Xiao \cite{BFJ09,fx19,huxiao25hodgeindex}.

On smooth projective toric varieties, intersection numbers of
nef line bundles are identified with mixed
volumes of lattice polytopes up to normalization. In this way,
the Khovanskii-Teissier inequality corresponds to the Alexandrov-Fenchel inequality.
Shenfeld-van Handel \cite{SvH23Acta} characterized the equality
cases of the Alexandrov-Fenchel inequality for convex polytopes.
Their algebraic formulation of the result led them to ask for
analogous descriptions in algebraic geometry; see \cite[Section~16.2]{SvH23Acta}.
These questions were subsequently studied
in \cite{HuXiaoLefChar,HuXiao25LefCharFree}.

In this paper, we extend the study of these kernels to higher degrees. The positivity conditions introduced below can be viewed as higher-degree analogues of supercriticality.

\subsection{Pure Lefschetz properties}
Our first result refines the preceding characterization by restricting the range of degrees under consideration. More precisely, we show that the obstructions to the Lefschetz property below a fixed degree are detected precisely by subvarieties whose codimensions lie in the corresponding range.
For each integer $k\geq 0$, denote by $\Irred^k(X)$ the set of all irreducible subvarieties of $X$ of codimension $k$.

\begin{thrm}\label{purelefcharintro}
    Let $d,l$ be integers satisfying $0\le 2l-1\leq d\leq n$ and let $L$ be a semiample line bundle.
    Then the following statements are equivalent.
    \begin{enumerate}
        \item For every integer $k$ with $0\leq k\leq l-1$ and every $V\in \Irred^k(X)$, we have
        \begin{equation*}
            L^{d-2k}\cdot [V]\neq 0.
        \end{equation*}
        \item For every integer $k$ with $0\leq k\leq l-1$, the linear map
        \begin{equation*}
            L^{d-2k}:H^{k,k}(X)\to H^{d-k,d-k}(X)
        \end{equation*}
        is injective.
        \item For all integers $p,q\geq 0$ with $ p+q\leq 2l-2$, the linear map 
            \begin{equation*}
                L^{d-p-q}: H^{p,q}(X)\to H^{d-q,d-p}(X)
            \end{equation*}
            is injective.
    \end{enumerate}
\end{thrm}

When the first equivalent condition holds, we say that $L$ is $d$-lef below degree $l$. Since $\Irred^0(X)=\{X\}$, we have
\begin{equation*}
    L \text{ is } d\text{-lef below degree } 1 \iff \nd_X(L)\geq d.
\end{equation*}
In fact, one can say more. Under the same lefness assumption, the Lefschetz maps remain injective in the next degree range except possibly in bidegree $(l,l)$, and in this exceptional bidegree their kernel admits a precise geometric description.

\begin{thrm}\label{purekerintro}
    Let $d,l$ be integers satisfying $l\ge 1$ and  $2l+1\leq d\leq n$, and
    let $L$ be a semiample line bundle that is $d$-lef below degree $l$. Then for all integers $p,q\geq 0$ with $2l-1\leq p+q\leq 2l$, the following statements hold.
    \begin{enumerate}
        \item If $(p,q)\neq (l,l)$, then the linear map 
            \begin{equation*}
                L^{d-p-q}: H^{p,q}(X)\to H^{d-q,d-p}(X)
            \end{equation*}
            is injective.
        \item If $(p,q)=(l,l)$, then
        \begin{equation*}
            \ker (L^{d-2l}:H^{l,l}(X)\to H^{d-l,d-l}(X))=\Span_{\mathbb{C}}\{[V]: V\in \Irred^l(X), L^{d-2l}\cdot [V]=0\}.
        \end{equation*}
    \end{enumerate}
     Moreover, the classes $\{[V]: V\in \Irred^l(X), L^{d-2l}\cdot [V]=0\}$
     are linearly independent.
\end{thrm}

The constraints on the integers $d,l$ in Theorems \ref{purelefcharintro} and \ref{purekerintro} are imposed to ensure that the relevant degree ranges are nonempty and that the powers of $L$ appearing in the two
theorems are positive.
Since both the operator $L^{d-2l}$ and the cycle classes in
Theorem \ref{purekerintro}(2) are defined over $\mathbb{R}$,
the same kernel characterization holds with real coefficients.
We use complex coefficients here to treat all bidegrees
uniformly.

The proofs of Theorems \ref{purelefcharintro} and \ref{purekerintro} are inspired by the work of de Cataldo-Migliorini \cite{decataldoLefsemismall,deCMdecomp}. The main problem is to determine the kernel
$\ker L^{d-p-q}\cap H^{p,q}(X).$
Since this space is unchanged after replacing $L$ by a positive tensor power, we may assume that $L$ is basepoint-free. The lefness condition can then be reformulated in terms of the Kodaira map $\phi_{|L|}$ associated to $L$. Using the geometric description of the decomposition theorem for this morphism developed by de Cataldo and Migliorini, we identify the relevant part of the perverse filtration on $H^{p,q}(X)$. In particular, we determine the subspace $H^{p,q}_{\leq d-n-1}(X),$
which leads to the kernel characterization via the inclusion
$\ker L^{d-p-q}\cap H^{p,q}(X)\subset H^{p,q}_{\leq d-n-1}(X).$

\subsection{Mixed Lefschetz properties}
We next extend these results to mixed products. 
Let $\mathcal{L}=(L_1,...,L_m)$ be a collection of semiample line bundles and let $\alpha \in \mathbb{Z}_{>0}^m$. We write $\mathcal{L}^{\alpha}=\prod_{i=1}^m L_i^{\alpha_i}$ for the complete intersection class associated with $\alpha$. By a slight abuse of notation, we also write $\mathcal{L}=L_1\cdots L_m$ for the complete intersection class.

The single-class characterization above admits a natural mixed generalization: powers of a single line bundle are replaced by all complete intersection classes arising from the collection $\mathcal{L}$, while the obstructions continue to be detected by algebraic cycles.

\begin{thrm}\label{mixedlefcharintro}
    Let $d,l,m$ be integers satisfying $0\leq 2l-1\leq d\leq n$ and $1\leq m\leq d-2l+2$, 
    and let $\mathcal{L}=(L_1,...,L_m)$ be a collection of semiample line bundles. Then the following statements are equivalent.
    \begin{enumerate}
        \item For every integer $k$ with $0\leq k\leq l-1$, every $V\in \Irred^k(X)$, and every $\alpha\in \mathbb{Z}_{>0}^{m}$ with $|\alpha|=d-2k$, we have 
        \begin{equation*}
            \mathcal{L}^{\alpha} \cdot [V]\neq 0.
        \end{equation*}
        
        \item For every integer $k$ with $0\leq k\leq l-1$ and every $\alpha\in \mathbb{Z}_{>0}^{m}$ with $|\alpha|=d-2k$, the linear map 
        \begin{equation*}
            \mathcal{L}^{\alpha}:H^{k,k}(X)\to H^{d-k,d-k}(X)
        \end{equation*}
        is injective.
        
        \item For all integers $p,q\geq 0$ with $p+q\leq 2l-2$ and every $\alpha\in \mathbb{Z}_{>0}^{m}$ with $|\alpha|=d-p-q$, the linear map 
        \begin{equation*}
            \mathcal{L}^{\alpha}:H^{p,q}(X)\to H^{d-q,d-p}(X)
        \end{equation*}
        is injective.
        
    \end{enumerate}
\end{thrm}

As before, a collection $\mathcal{L}=(L_1,...,L_m)$ of semiample line bundles is called $d$-lef below degree $l$ if $1\leq m\leq d-2l+2$ and
the first equivalent condition in Theorem \ref{mixedlefcharintro} holds. 
By the positivity criterion for complete intersections of semiample divisors (see \cite{KK16,hxhardlef}, or Proposition \ref{hallrado}), 
for any $0\leq k\leq l-1$ and any $V\in\Irred^k(X)$, the following conditions are equivalent:
\begin{itemize}
    \item For every $\alpha \in \mathbb{Z}_{>0}^m$ with $|\alpha|=d-2k$, we have  $\mathcal{L}^{\alpha}\cdot [V]\neq 0.$
    \item For every nonempty subset $I\subset [m]$, we have  
    $$\nd_V(L_I)\geq |I|+d-m-2k$$
    where $L_I=\sum_{i\in I}L_i$ and $\nd_V(L_I)=\max\{r\in \mathbb{N}:L_I^r \cdot [V]\neq 0\}$.
\end{itemize}
It follows that 
\begin{equation*}
    \mathcal{L} \text{ is } d\text{-lef below degree } l \iff L_I \text{ is } (d-m+|I|)\text{-lef below degree } l, \forall \emptyset \neq I\subset [m].
\end{equation*}
In particular, when $l=1$, we obtain
\begin{equation*}
    \mathcal{L} \text{ is } d\text{-lef below degree } 1 \iff \nd_X(L_I)\geq |I|+d-m, \forall \emptyset \neq I\subset [m].
\end{equation*}

As in the single-class case, the lefness assumption yields a stronger statement in the next degree range.

\begin{thrm}\label{mixedkerintro}
    Let $d,l$ be integers satisfying $l\geq 1$ and $ 2l+1\leq d\leq n$, 
    and let $\mathcal{L}=(L_1,...,L_m)$ be a collection of semiample line bundles. Suppose that $\mathcal{L}$ is $d$-lef below degree $l$. 
    Then for all integers $p,q\geq 0$ with $2l-1\leq p+q\leq 2l$ and  every $\alpha\in \mathbb{Z}^m_{>0}$ with $|\alpha|=d-p-q$, the following statements hold.
    \begin{enumerate}
        \item If $(p,q)\neq (l,l)$, then the map
        \begin{equation*}
            \mathcal{L}^{\alpha}:H^{p,q}(X)\to H^{d-q,d-p}(X)
        \end{equation*}
        is injective.
        \item If $(p,q)= (l,l)$, then
        \begin{equation*}
            \ker(\mathcal{L}^{\alpha}:H^{l,l}(X)\to H^{d-l,d-l}(X))=\Span_{\mathbb{C}}\{[V]:V\in \Irred^l(X), \mathcal{L}^{\alpha }\cdot [V]=0\}.
        \end{equation*}
    \end{enumerate}
    Moreover, the classes $\{[V]:V\in \Irred^l(X), \mathcal{L}^{\alpha }\cdot [V]=0\}$ are linearly independent in this case.
\end{thrm}

The case $(p,q)=(1,1)$ is worth mentioning directly since it is related to the equality characterization for the Khovanskii-Teissier inequality.
For further discussions and applications, we refer to \cite{HuXiaoLefChar, HuXiao25LefCharFree} and references therein. We write $\prim(X)=\Irred^1(X)$ for the set of all prime divisors.

\begin{cor}
    Let $\mathcal{L}=(L_1,...,L_{n-2})$ be a supercritical collection of semiample line bundles, that is, $\nd_X(L_I)\geq |I|+2$ for every nonempty subset $I\subset [n-2]$. Then
    \begin{equation*}
        \ker\mathcal{L}\cap H^{1,1}(X)=\Span_{\mathbb{C}}\{[D]:D\in \prim(X), \mathcal{L}\cdot [D]=0\}.
    \end{equation*}
\end{cor}
This result also explains the notion of supercriticality for a collection $\mathcal{L}=(L_1,...,L_{n-2})$. 
It is exactly equivalent to $\mathcal{L}$ being $n$-lef below degree $1$.
This equivalence also serves as one of the main motivations for our lefness notion.
We will later generalize this result to nef classes using the kernel characterization theorem for Lefschetz modules.



We provide two approaches to Theorems \ref{mixedlefcharintro} and \ref{mixedkerintro}.
The first approach is geometric. Related ideas already appear in \cite{hushangxiao2023hardlefArxiv,HuXiaoLefChar,HuXiao25LefCharFree}. The argument proceeds by induction.
We first choose a general hypersurface on which the required
lefness conditions are preserved, and apply the induction
hypothesis there.
The single-class kernel characterization then allow us to lift the resulting information back to $X$.
The second approach is inspired by the decomposition theorem for Lefschetz modules recently established by Amini-Huh-Larson \cite{aminihuhlarson25lefschetzmodule}, together with the descent and purity theorems of Cattani-Kaplan-Schmid \cite{CKS87}. This approach is more abstract and relies on the theory of degenerations of Hodge structures. 
It applies to arbitrary nef classes and more generally to Lefschetz modules.

\subsection{The Lefschetz-module approach}
Lefschetz modules provide an abstract framework for intersection cohomology, involving a variety of nontrivial examples, including the intersection cohomology of projective varieties \cite{saito88,saito90}, the combinatorial intersection cohomology of projective fans \cite{karu04}, and the intersection cohomology of matroids \cite{BHMPW26singularHodge}.
For further details and many beautiful applications, we refer to \cite{aminihuhlarson25lefschetzmodule} and the references therein.

Let us briefly recall the setup.
 Let $A$ be a finite-dimensional graded commutative $\mathbb{R}$-algebra and let
$\mathcal{K}_A\subset A^1$ be a fixed open convex cone, called the K\"ahler cone. Its closure $\overline{\mathcal{K}_A}$ is called the nef cone.
Let $M$ be a finite-dimensional graded $A_{\mathbb{C}}=A\otimes_{\mathbb{R}}\mathbb{C}$-module and let $Q_M: M\times M\to \mathbb{C}$
be an $A_{\mathbb{C}}$-invariant Hermitian pairing on $M$. 
Roughly speaking, the pair $(M,Q_M)$ is a complex Lefschetz module over $(A,\mathcal{K}_A)$ if it satisfies Poincar\'e duality, the hard Lefschetz property and the Hodge-Riemann relations.
We refer to Section \ref{decoupLefMod} for the precise definition.

We now fix a complex Lefschetz module $(M,Q_M)$ of degree $n$ over $(A,\mathcal{K}_A)$. The reader may keep the following example in mind.
Let $X$ be a compact K\"ahler manifold of dimension $n$.
Take $A=\bigoplus_{k=0}^n H^{k,k}(X,\mathbb{R})$ and let $\mathcal{K}_A=\mathcal{K}_X$ be the K\"ahler cone.
Set $M=\bigoplus_{k=0}^n H^{k,k}(X)$ with the Hermitian pairing  $Q_M$ defined by
$$Q_M(x,y)= \int_X x\cdot \overline{y}$$
for $x\in M^k,y\in M^{n-k}$ and $Q_{M}(M^i,M^j)=0$ whenever $i+j\neq n$. 
Then the classical Hodge theory implies that $(M,Q_M)$ is a complex Lefschetz module of degree $n$ over $(A,\mathcal{K_A})$.

Motivated by the equivalent characterizations in Theorem \ref{purelefcharintro}, we define lefness on $M$.

\begin{defn}
    Let $d,l$ be integers satisfying $2l-1\leq d\leq n$. A nef class $L\in \overline{\mathcal{K}_A}$ is called $d$-lef below degree $l$ on $M$ if the map
    \begin{equation*}
        L^{d-2k}:M^k\to M^{d-k}
    \end{equation*}
    is injective for every integer $k\leq l-1$.
\end{defn}

The following theorem gives a criterion for mixed Lefschetz properties in this abstract setting.
A result essentially equivalent to the following criterion was first observed by Enhan Li and Weizhe Zheng \cite{lizhengHRpurity}.
In our framework, it follows naturally from the kernel characterization below.

\begin{thrm}\label{alglefcharintro}
    Let $d,l,m$ be integers satisfying $2l-1\leq d\leq n$ and $1\leq m\leq d-2l+2$, and let $(L_1,...,L_{m})$ be a collection of nef classes. Then the following statements are equivalent.
    \begin{enumerate}
        \item For every $k\leq l-1$ and any $\alpha\in \mathbb{Z}_{>0}^m$ with $|\alpha|=d-2k$, the map
        \begin{equation*}
            \mathcal{L}^{\alpha} :M^k\to M^{d-k}
        \end{equation*}
        is injective.
        \item For every nonempty subset $I\subset [m]$, the class $L_I$ is $(|I|+d-m)$-lef below degree $l$.
    \end{enumerate}
\end{thrm}

As before, a collection $(L_1,...,L_{m})$ of nef classes is called $d$-lef below degree $l$ on $M$ if $1\leq m\leq d-2l+2$ and one of the above equivalent conditions holds.
As in the geometric setting, the kernel admits a description in the next degree.
For $\alpha\in \mathbb{Z}^m$ and every nonempty subset $I\subset [m]$, we write $\alpha_I=\sum_{i\in I}\alpha_i.$

\begin{thrm}\label{algkerintro}
    Let $d,l$ be integers satisfying $2l+1\leq d\leq n$, and let $(L_1,...,L_{m})$ be a collection of nef classes that is $d$-lef below degree $l$. Then for every $\alpha\in \mathbb{Z}_{>0}^m$ with $|\alpha|=d-2l$, we have
    \begin{equation*}
        \ker \mathcal{L}^{\alpha}\cap M^l=\sum_{\emptyset\neq I\subset [m]}\ker L_{I}^{\alpha_I}\cap M^l.
    \end{equation*}
\end{thrm}

Applying these two results to the cohomology of smooth projective varieties, we recover the mixed Lefschetz theorems \ref{mixedlefcharintro} and \ref{mixedkerintro} from the pure Lefschetz theorems \ref{purelefcharintro} and \ref{purekerintro}.
We refer to Section \ref{geomapp} for details.
As another application, we characterize the kernel associated with a supercritical collection of nef classes, thereby solving a conjecture of Shenfeld-van Handel and Hu-Xiao \cite{SvH23Acta,HuXiaoLefChar}.

\begin{thrm}\label{nefker(1,1)intro}
    Let $X$ be a smooth projective variety over $\mathbb{C}$ of dimension $n$, and let $\mathcal{L}=(L_1,...,L_{n-2})$ be a supercritical collection of nef classes. Then 
    \begin{equation*}
        \ker  \mathcal{L}\cap H^{1,1}(X,\mathbb{R})=\Span_{\mathbb{R}}\{[D]:D\in \prim (X), \mathcal{L}\cdot [D]=0\}.
    \end{equation*}
\end{thrm}

As a consequence, one obtains the following characterization of the equality cases in the Khovanskii-Teissier inequality.

\begin{cor}\label{extremalKT}
        Let $X$ be a smooth projective variety over $\mathbb{C}$ of dimension $n$, let $A,B$ be nef classes and let $1\leq k\le n-1$ be an integer.
        Assume that $\int_X A^{k}\cdot B^{n-k}>0$.
        Then the equality 
        \begin{equation*}
            (\int_X A^{k}\cdot B^{n-k})^2=\int_X A^{k-1}\cdot B^{n-k+1}\int_X A^{k+1}\cdot B^{n-k-1}
        \end{equation*}
        holds if and only if there exists a real constant $c>0$ such that
        \begin{equation*}
            A-cB=\sum a_D [D]
        \end{equation*}
        where $a_D\in \mathbb{R}$ and the sum runs over prime divisors $D$ satisfying $A^{k-1}\cdot B^{n-k-1}\cdot[D]=0.$
\end{cor}

This article is organized as follows. In Section \ref{singlelefchar}, 
we characterize the kernels of Lefschetz operators given by powers of a single semiample class, proving Theorems \ref{purelefcharintro} and \ref{purekerintro}.
To extend the single-class case to the mixed setting, we present two approaches. The first one is geometric and is developed in Section \ref{geometricdecoupling}.
The second is based on the theory of Lefschetz modules and is presented in Section \ref{decoupLefMod}.

\section{Kernel characterizations for a single class}\label{singlelefchar}
In this section, we prove the kernel characterizations for a single semiample class $L$. The proof is inspired by \cite{decataldoLefsemismall,deCMdecomp}. We first study the geometry of the Kodaira map $f$ associated with $L$. The lefness condition on $L$ will be explained as a generalized semismallness condition on $f$. We then analyze the decomposition theorem for $f$ and determine the relevant terms of the perverse filtration which characterize the kernels.

Throughout this section, we use the following setup and notation.
\begin{itemize}
    \item Let $X$ be a smooth projective variety over $\mathbb{C}$ of dimension $n$.
    \item Let $L$ be a line bundle on $X$. By an abuse of notation, we use the same symbol $L$ to denote its first Chern class $c_1(L)\in H^{1,1}(X,\mathbb{R})$. We use additive notation for line bundles; thus $L+M$ denotes $L\otimes M$.
    \item Let $H^*(X)$ denote the cohomology ring of $X$ with real coefficients. 
    For a complex $M$ of vector spaces or sheaves, we use the degree shift convention $(M[k])^i=M^{i+k}.$
    \item We denote by $\Irred^l(X)$ the set of all irreducible subvarieties of $X$ of codimension $l$.
    \item If $L$ is nef, we define its numerical dimension by
    $$\nd(L)=\max\{k\in \mathbb{N}:L^k\neq 0\}.$$
    More generally, 
    for every irreducible subvariety $V\subset X$, we define the numerical dimension of $L$ with respect to $V$ by
    $$\nd_V(L)=\max\{k\in \mathbb{N}:L^k\cdot [V]\neq 0\}.$$
\end{itemize}

\subsection{The decomposition theorem}

We briefly recall the notation and results from the theory of perverse sheaves that will be used in this section. Unless otherwise specified, all sheaves are taken with real coefficients. We refer to \cite{BBDG82DecompsitionThrm,deCMdecomp,deCataldoMiglioriniPerverse} for further details.

Let $Y$ be a complex irreducible algebraic variety. We denote by $\mathbb{R}_Y$ and $\omega_Y$ the constant sheaf and the dualizing complex on $Y$, respectively. If $Y$ is smooth, then there is a canonical isomorphism  $\omega_Y\simeq \mathbb{R}_Y[2\dim Y].$
We write $D_c^b(Y)$ for the bounded derived category of constructible sheaves of real vector spaces on $Y$, and denote by 
$\mathbb{D}_Y:D_c^b(Y)\to D_c^b(Y) $
the Verdier duality functor.

The category $D_c^b(Y)$ carries the standard $t$-structure, whose truncation functors are denoted by $\tau_{\leq k}$ and $\tau_{\geq k}$. 
Thus, for every $K\in D_c^b(Y)$ and every integer $k$, there is a canonical distinguished triangle 
\begin{equation*} 
    \tau_{\leq k}K \rightarrow K \rightarrow \tau_{\geq k+1}K \overset{+1}{\rightarrow}. 
\end{equation*} 
We also consider the perverse $t$-structure on $D_c^b(Y)$. Its truncation functors and cohomology functors are denoted by 
${}^p\tau_{\leq k},  {}^p\tau_{\geq k} $ and $  {}^p\mathcal{H}^k,$
respectively. In particular, there is a canonical distinguished triangle 
\begin{equation*}
    {}^p\tau_{\leq k}K \rightarrow K \rightarrow {}^p\tau_{\geq k+1}K \overset{+1}{\rightarrow}. 
\end{equation*} 
The objects in the heart ${}^pD_c^{\leq 0}(Y)\cap{}^pD_c^{\geq 0}(Y)$ are called perverse sheaves.

Let $Y=\bigsqcup_{\alpha \in \Lambda} S_{\alpha}$ be a Whitney stratification of $Y$ and let $\mathcal{L}$ be a local system on a stratum $S_{\alpha}$. 
\begin{defn}
    The intersection complex $\IC(\overline{S_{\alpha}},\mathcal{L})\in D_c^b(Y)$ is the unique complex supported on $\overline{S_{\alpha}}$ satisfying the following conditions.
\begin{itemize}
    \item (Normalization) $\IC(\overline{S_{\alpha}},\mathcal{L})_{|S_{\alpha}}\cong \mathcal{L}[\dim S_{\alpha}]$.
    \item (Support condition) $\mathcal{H}^i(\IC(\overline{S_{\alpha}},\mathcal{L}))_{|S_{\beta}}=0$ 
    for every $\beta <\alpha$ and every $i \geq -\dim S_{\beta}.$
    \item (Cosupport condition) $\mathbb{D}_Y\IC(\overline{S_{\alpha}},\mathcal{L})$ satisfies the support condition.
\end{itemize}
With this normalization, the intersection cohomology groups with coefficients $\mathcal{L}$ are defined by 
\begin{equation*} 
    IH^k(\overline{S_{\alpha}},\mathcal{L}) := \mathbb{H}^{k-\dim S_{\alpha}} \bigl( \overline{S_{\alpha}},\IC(\overline{S_{\alpha}},\mathcal{L}) \bigr)
\end{equation*}
where $\mathbb{H}$ denotes hypercohomology.
\end{defn}

We now turn to the relative setting.
Let $f:X\to Y$ be a projective morphism where $X$ is smooth of dimension $n$. 
\begin{defn}
    A stratification $Y=\bigsqcup_{\alpha\in \Lambda}S_{\alpha}$ of $Y$ is said to be adapted to $f$ if it is a Whitney stratification and for every $\alpha\in \Lambda$, the restriction of $f$ over $S_{\alpha}$ is a stratified submersion. More precisely, there exists a Whitney stratification $f^{-1}(S_\alpha)=\bigsqcup_\beta T_{\alpha,\beta}$  such that 
    \begin{equation*} 
        f_{|T_{\alpha,\beta}}: T_{\alpha,\beta}\rightarrow S_\alpha 
    \end{equation*}
    is a submersion for every $\beta$.
\end{defn}

Such an adapted stratification exists by \cite[Part I, \S1.7]{GoreskyMacphersonStraMorse}. Since $f$ is projective, Thom's first isotopy lemma \cite[Part I, \S1.5]{GoreskyMacphersonStraMorse} implies that the map $f$ is locally topologically trivial over each stratum $S_\alpha$. In particular, $Rf_* \mathbb{R}[n]$ is constructible with respect to an adapted stratification.

We recall the two parts of the decomposition package that will be used later.

\begin{itemize}
\item \textbf{Decomposition theorem.} There is a (noncanonical) isomorphism in $D^b_c(Y)$
\begin{equation*}
    Rf_*\mathbb{R}_X[n] \cong \bigoplus_{i\in \mathbb{Z}} {}^p\mathcal H^i(Rf_*\mathbb{R}[n])[-i].
\end{equation*}

\item \textbf{Semisimplicity theorem.} Each perverse cohomology sheaf
${}^p\mathcal H^i(Rf_*\mathbb{R}[n])$
is semisimple. 
After refining the adapted stratification if necessary, it decomposes as
\begin{equation*}
{}^p\mathcal H^i(Rf_*\mathbb{R}[n])\cong \bigoplus_{\alpha\in \Lambda} \operatorname{IC}(\overline{S_{\alpha}},\mathcal{L}_{\alpha,i})
\end{equation*}
where $\mathcal{L}_{\beta,i}$ is a semisimple local system on $S_{\alpha}$. 
Restricting to $S_{\alpha}$ and taking the $(-\dim S_{\alpha})$-th cohomology sheaf, we obtain 
$$\mathcal{L}_{\alpha,i}=\mathcal{H}^{-\dim S_{\alpha }}({}^p\mathcal H^i(Rf_*\mathbb{R}[n])_{|S_{\alpha}}).$$
\end{itemize}

The perverse truncation functors on $Rf_*\mathbb{R}_X[n]$ induce an increasing filtration on the cohomology of $X$. For every pair of integers $i,j$, define
\begin{equation*}
H^i_{\leq j}(X)
= 
\Image (
\mathbb{H}^{i-n}(Y,{}^p\tau_{\leq j}Rf_*\mathbb{R}_X[n])
\rightarrow
\mathbb{H}^{i-n}(Y,Rf_*\mathbb{R}_X[n])).
\end{equation*}
Using the canonical identification $\mathbb{H}^{i-n}\!\left(Y,Rf_*\mathbb{R}_X[n]\right)\cong H^{i}(X),$
we obtain the increasing perverse filtration
\begin{equation*}
\cdots \subset H^i_{\leq j-1}(X) \subset H^i_{\leq j}(X) \subset H^i_{\leq j+1}(X) \subseteq \cdots,
\end{equation*}

After choosing a decomposition in the decomposition theorem, this filtration admits a splitting
\begin{equation*}
    H^i_{\leq j}(X)=\bigoplus_{k\leq j} H^{i}_k(X)
\end{equation*}
where
\begin{equation*}
    H^i_k(X)\cong H^{i}_{\leq k}(X)/H^{i}_{\leq k-1}(X)\cong \mathbb{H}^{i-n-k}(Y,{}^p\mathcal{H}^k(Rf_*\mathbb{R}_X[n])).
\end{equation*}
Although such a splitting is noncanonical, the perverse filtration itself is canonical.

The following observation, which follows directly from the identification of the perverse filtration with the weight filtration associated with $L$ (see \cite[Proposition 5.2.4]{deCMdecomp}), provides the link between the perverse filtration and the kernels of the Lefschetz operators considered in this manuscript.

\begin{lem}\label{kerandperv}
    Let $A$ be an ample $\mathbb{Q}$-line bundle on $Y$ and let $L=f^*A$.
    Then for all integers $d,k$ with $0\leq k\leq d\leq n$, we have
    $$\ker(L^{d-k}:H^k(X)\rightarrow H^{2d-k}(X))\subset H^{k}_{\leq d-n-1}(X).$$
\end{lem}
\begin{proof}
    By the defining properties of the weight filtration associated with $L$
    (see, for example, Section \ref{decoupLefMod}), we have
    \begin{equation*}
        \ker L^{d-k}\cap H^k(X) \subset W^L_{d-k-1}\cap H^{k}(X).
    \end{equation*}
    On the other hand, \cite[Proposition~5.2.4]{deCMdecomp} gives
    \begin{equation*}
        W^L_{d-k-1}\cap H^{k}(X)= H^{k}_{\leq d-n-1}(X).
    \end{equation*}
    The desired inclusion follows.
\end{proof}

\subsection{Lefness for semiample line bundles}
We now introduce the positivity condition that will govern the Lefschetz properties studied in this section.

\begin{defn}
    Let $d,l$ be integers satisfying $2l-1\leq d\leq n$ and let $L$ be a semiample line bundle.
    $L$ is called $d$-lef below degree $l$ if 
    $$L^{d-2\codimension V}\cdot [V]\neq 0$$ 
    for every subvariety $V\subset X$ with $\codimension_X V\leq l-1$. 
\end{defn}

When $l=1$, this condition is simply 
\begin{equation*} 
    L\text{ is }d\text{-lef below degree }1 \iff  \nd(L)\geq d. 
\end{equation*} 
More generally, when passing from $d$-lefness below degree $l-1$ to $d$-lefness below degree $l$, the additional condition is
\begin{equation*}
    L^{d-2(l-1)}\cdot [V]\neq 0, \ \ \ \forall V\in \Irred^{l-1}(X).
\end{equation*} 
This condition is nontrivial precisely when 
$2l-1\leq d. $
Thus the relevant range is 
\begin{equation*} 
    1\leq d\leq n,  \ \ \  1\leq l\leq \left\lceil\frac{d}{2}\right\rceil. 
\end{equation*} 
When $l=\left\lceil\frac{d}{2}\right\rceil, $
we simply say that $L$ is $d$-lef. 
As we will see in Proposition \ref{equiv}, $n$-lefness is equivalent to notion of lefness introduced in \cite{decataldoLefsemismall}. And $d$-lefness agrees with
the notion of the same name in \cite{hushangxiao2023hardlefArxiv}.
Here we additionally keep track of the degree truncation parameter $l$.

In the remainder of this subsection, we fix integers $d,l$ satisfying $0\leq 2l-1\leq d\leq n$ and a semiample line bundle $L$. We also fix a morphism $f:X\to Y$ such that $L=f^* A$ for some ample $\mathbb{Q}$-line bundle on $Y$. Such a morphism exists, for example, by taking the Kodaira map $\phi_{|mL|}$ associated with a sufficiently divisible multiple $mL$ of $L$. The key property that will be used repeatedly is $\nd_V(L)= \dim f(V)$ for every subvariety $V\subset X$.
We further fix a stratification $Y=\bigsqcup_{\alpha \in \Lambda}S_{\alpha}$ of $Y$ adapted to $f$.

We first show that lefness will imply the rigidity of certain "exceptional" subvarieties.

\begin{lem}\label{rigidity}
    If $L$ is $d$-lef below degree $l$,
    then for every $V\in \Irred^l(X)$ satisfying $\dim f(V)\leq d-2l$, 
    there exists a stratum $S_{\alpha}$ such that $f(V)=\overline{S_{\alpha}}$, and $V$ is a maximal-dimensional irreducible component of $\overline{f^{-1}(S_\alpha)}$.
\end{lem}
\begin{proof}
    Since $f(V)$ is irreducible, there is a unique stratum $S_\alpha$ such that $f(V)\cap S_\alpha$ is dense in $f(V)$.
    We first show that $S_{\alpha} \subset f(V)$. 
    Suppose otherwise. Since $f(V)\cap S_\alpha$ is dense in $f(V)$, we have
    \begin{equation*} \dim(f(V)\cap S_\alpha)=\dim f(V). \end{equation*}
    If $\dim S_\alpha=\dim f(V)$, then $f(V)\cap S_\alpha$ is also dense in $S_\alpha$. Hence  $\overline{S_\alpha} = \overline{f(V)\cap S_\alpha} = f(V),$ contradicting $S_\alpha\not\subset f(V)$. Therefore $r:=\dim S_\alpha-\dim f(V)\geq 1.$
    Let $W$ be an irreducible component of $\overline{f^{-1}(S_{\alpha})}$ of maximal dimension.
    Since $f$ is a stratified submersion over $S_{\alpha}$, the fibers over $S_{\alpha}$ all have the same dimension. Therefore,
    \begin{equation*}
        \dim W \geq \dim S_{\alpha}+\dim V-\dim f(V)=\dim V+r.
    \end{equation*}
    Equivalently, $\codimension W \leq l-r.$
    Since $L$ is $d$-lef below degree $l$, we obtain $$L^{d-2l+2r}\cdot [W]\neq 0.$$ 
    On the other hand,
    $$\dim f(W)=\dim S_{\alpha}=\dim f(V)+r\leq d-2l+r,$$
    and hence
    $$L^{d-2l+r+1}\cdot [W]=0,$$
    a contradiction.

    It remains to show that $V$ is a maximal-dimensional irreducible component of $\overline{f^{-1}(S_\alpha)}$.
    Suppose otherwise. Then there exists an irreducible component $W$ of $\overline{f^{-1}S_{\alpha}}$ satisfying $\dim W>\dim V.$ 
    Then $\codimension W \leq l-1.$ Since $L$ is $d$-lef below degree $l$, we have
    \begin{equation*}
        L^{d-2\codimension W}\cdot [W]\neq 0.
    \end{equation*}
    On the other hand, $$\dim f(W)=\dim S_{\alpha}=\dim f(V)\leq d-2l,$$
    and therefore
    $$L^{d-2l+1}\cdot [W]=0,$$
    a contradiction.
\end{proof}

\begin{cor}\label{finiteness}
    If $L$ is $d$-lef below degree $l$, then the set 
    \begin{equation*}
        \{V \in \Irred^l(X):L^{d-2l+1}\cdot [V]=0\}
    \end{equation*}
    is finite
\end{cor}
\begin{proof}
    By Lemma \ref{rigidity}, every irreducible subvariety $V$ of codimension $l$ annihilated by $L^{d-2l+1}$ is an irreducible component of $\overline{f^{-1}(S_{\alpha})}$ for some stratum $S_{\alpha}$. Since there are only finitely many strata and each
    $\overline{f^{-1}(S_{\alpha})}$ has only finitely many irreducible components, the desired finiteness follows immediately.
\end{proof}


As in \cite{decataldoLefsemismall,hushangxiao2023hardlefArxiv}, we characterize the lefness of $L$ in terms of the geometry of $f$.
Recall that the defect of semismallness of $f$ can be expressed as 
\begin{equation*}
    r(f)= \max_{\alpha\in \Lambda}\{\dim S_{\alpha}+2\dim f^{-1}(y_{\alpha})-n\}
\end{equation*}
where $y_{\alpha}\in S_{\alpha}$ is an arbitrary point. 
When $\dim S_{\alpha}\geq d$, we have
\begin{equation*}
    \dim S_{\alpha}+2\dim f^{-1}(y_{\alpha})-n\leq n-\dim S_{\alpha} \leq n-d.
\end{equation*}
Hence $r(f)\leq n-d$ if and only if
\begin{equation*}
    \dim f^{-1}y_{\alpha}\leq n-\frac{d+\dim S_{\alpha}}{2}
\end{equation*}
for every stratum $S_{\alpha}$ with $\dim S_{\alpha}\leq d$. 
Therefore the following equivalence $(i)\iff (iv)$ implies that $L$ is $d$-lef if and only if $r(f)\leq n-d$. This agrees with the definition of $d$-lefness in \cite{hushangxiao2023hardlefArxiv}.

\begin{prop}\label{equiv}
    The following statements are equivalent.
    \begin{enumerate}[label=(\roman*)]
        \item $L$ is $d$-lef below degree $l$.

        \item For every integer $n-d\leq k\leq n-d+l$, we have
        \begin{equation*}
            \dim\{x\in X:\dim (f^{-1}f(x),x) \geq k\}\leq 2n-d-k
        \end{equation*}
        where $(f^{-1}f(x),x)$ denotes the germ of the fiber $f^{-1}f(x)$ at $x$.
        
        \item For every integer $n-d\leq k\leq n$, we have 
        \begin{equation*}
            \dim\{y\in Y:\dim f^{-1}(y) = k\}\leq
            \begin{cases}
            2n-d-2k & \text{if } n-d \leq k \leq n-d+l \\
            n-l-k & \text{if }  n-d+l\leq k \leq n
            \end{cases}.
        \end{equation*}
    
        \item For every stratum $S_{\alpha}$ and every point $y_{\alpha}\in S_{\alpha}$, we have 
        \begin{equation*}
        \dim f^{-1} y_{\alpha} \leq
        \begin{cases}
        n - l - \dim S_{\alpha} & \text{if } 0 \leq \dim S_{\alpha} \leq d - 2l \\
        n - \frac{d}{2} - \frac{\dim S_{\alpha}}{2} & \text{if } d - 2l \leq \dim S_{\alpha} \leq d \\
        n - \dim S_{\alpha} & \text{if } d \leq \dim S_{\alpha} \leq n
        \end{cases}.
        \end{equation*}
        Here the bound $\dim f^{-1} y_{\alpha} \leq n - \dim S_{\alpha}$ is automatic for every stratum $S_{\alpha}$. We include the third case for completeness.
    \end{enumerate}
\end{prop}
\begin{proof}
    \noindent
    $(i)\implies (ii)$
    For every integer $k$, write
    $$X^{\geq k}=\{x\in X:\dim (f^{-1}f(x),x) \geq k\}.$$
    This set is Zariski closed by the upper semicontinuity of the local dimension of the fibers.
    We prove by contradiction. Suppose that there exists  $n-d\leq k\leq n-d+l$ such that
    \begin{equation*}
         \dim X^{\geq k} \geq  2n-d-k+1.
    \end{equation*}
    Let $V$ be an irreducible component of $X^{\geq k}$ of maximal dimension. 
    Then $$\codimension_X V \leq n-(2n-d-k+1)=d-n-1+k \leq l-1.$$ Since $L$ is $d$-lef below degree $l$, we have
    $L^{d-2\codimension_X V}\cdot [V]\neq 0.$ Moreover
    $$d-2\codimension_X V=d-\codimension_X V -n+\dim V\geq \dim V-k+1,$$
    and hence
    $$L^{\dim V-k+1}\cdot [V]\neq 0.$$
    On the other hand, since $V$ is a maximal-dimensional irreducible component of $X^{\geq k}$, we have $\dim V=\dim X^{\geq k}.$ 
    Therefore 
    \begin{equation*} \dim f(V) \leq \dim f(X^{\geq k}) \leq \dim X^{\geq k}-k = \dim V-k. \end{equation*}
    It follows that
    \begin{equation*}
        L^{\dim V-k+1}\cdot [V]=0,
    \end{equation*}
    a contradiction. This proves $(ii)$.

    \medskip
    \noindent   
    $(ii)\implies (iii)$ 
    For every integer $k$, set 
    $$Y^k=\{y\in Y:\dim f^{-1}(y) = k\}.$$
    The top-dimensional part of every fiber over $Y^k$ is contained in $X^{\geq k}$. Therefore
    \begin{equation*} 
        \dim Y^k+k \leq \dim X^{\geq k}.
    \end{equation*}
    For $n-d\leq k\leq n-d+l$, by $(ii)$, we have
    \begin{equation*}
        \dim Y^k \leq \dim X^{\geq k}-k\leq 2n-d-2k.
    \end{equation*}
    For $n-d+l\leq k\leq n$, since $X^{\geq k}\subset X^{\geq n-d+l}$, we have
    \begin{equation*}
        \dim Y^k \leq \dim X^{\geq n-d+l}-k\leq n-l-k. 
    \end{equation*}
    This proves $(iii)$.
        
    \medskip
    \noindent    
    $(iii)\implies (iv)$
    If $d < \dim S_{\alpha} \leq n$, the statement is clear since 
    \begin{equation*}
        \dim S_{\alpha}+\dim f^{-1}(y_{\alpha})=\dim f^{-1}(S_{\alpha})\leq n.
    \end{equation*}
    Suppose next that $0\leq  \dim S_{\alpha}\leq d-2l$. We argue by contradiction. Suppose that 
    $$\dim f^{-1}(y_{\alpha})\geq n-l-\dim S_{\alpha}+1.$$ 
    Then $S_{\alpha}\subset Y^{\dim f^{-1}(y_{\alpha})}$ and the second estimate in $(iii)$ applies, which gives
    $$\dim S_{\alpha}\leq \dim Y^{\dim f^{-1}(y_{\alpha})} \leq n-l-\dim f^{-1}(y_{\alpha})\leq \dim S_{\alpha}-1.$$
    This is a contradiction.
    
    Finally, suppose that $d-2l< \dim S_{\alpha}\leq d$. 
    Assume, for contradiction, that
    \begin{equation}\label{dn19i}
        \dim f^{-1}(y_{\alpha})> n-\frac{d}{2}-\frac{\dim S_{\alpha}}{2}.
    \end{equation}
    Since $\dim f^{-1}(y_{\alpha})$ is an integer, this implies 
    $$\dim f^{-1}(y_{\alpha})\geq n-\lceil \frac{d+\dim S_{\alpha}}{2}\rceil+1\geq n-d.$$
    If $\dim f^{-1}(y_{\alpha}) >n-d+l$, then the second estimate in $(iii)$ gives
    \begin{equation*}
        \dim S_{\alpha}\leq \dim Y^{\dim f^{-1}(y_{\alpha})}\leq n-l-\dim f^{-1}(y_{\alpha})< d-2l,
    \end{equation*}
    contradicting $\dim S_{\alpha}>d-2l$.
    We may therefore assume
    $$n-d\leq \dim f^{-1}(y_{\alpha}) \leq n-d+l.$$
    The first estimate in $(iii)$ combined with \eqref{dn19i} then gives
    \begin{equation*}
        \dim S_{\alpha} \leq \dim Y^{\dim f^{-1}(y_{\alpha})} \leq 2n-d-2 \dim f^{-1}(y_{\alpha})<\dim S_{\alpha},
    \end{equation*}
    a contradiction. This proves $(iv)$.

    $(iv)\implies (i)$
    We prove by induction on $k$ that 
    \begin{equation*} L^{d-2k}\cdot[V]\neq0 \end{equation*} 
    for every $V\in\Irred^k(X)$ and every $0\leq k\leq l-1.$
    
    For the base case, we need to show that $L^{d}\cdot [X]\neq 0$ which is equivalent to $\dim Y \geq d$. Suppose otherwise. Then every stratum $S_\alpha$ satisfies $\dim S_\alpha<d$.
    Hence, by $(iv)$, 
    $$\dim f^{-1}(S_{\alpha})=\dim f^{-1}(y_{\alpha})+\dim S_{\alpha}<n.$$
    Thus every $f^{-1}(S_\alpha)$ has dimension strictly smaller than $n$. Since there are only finitely many strata, this contradicts 
    $X=\bigsqcup_{\alpha\in\Lambda}f^{-1}(S_\alpha).$
    Therefore $\dim Y\geq d$.
    
    Now let $1\leq k\leq l-1$, and assume inductively that $L$ is $d$-lef below degree $k$. We show that 
    $$L^{d-2k}\cdot [V]\neq 0$$
    for every $V\in \Irred^{k}(X)$. 
    Suppose otherwise. 
    Then there exists $V\in \Irred^k(X)$ such that $L^{d-2k}\cdot [V]=0.$ 
    By Lemma \ref{rigidity}, there exists a stratum $S_{\alpha}$ such that 
    $f(V)=\overline{S_{\alpha}}$ and 
    $V$ is an irreducible component of $\overline{f^{-1}(S_{\alpha})}.$
    Since $L^{d-2k}\cdot [V]=0,$ we have $$\dim S_{\alpha}\leq d-2k-1.$$
    We distinguish two cases.
    If $\dim S_{\alpha}\le d-2l$, then $(iv)$ gives
    \begin{equation*}
        \dim f^{-1}(y_{\alpha})\leq n-l-\dim S_{\alpha}.
    \end{equation*}
    hence
    \begin{equation*}
        \dim V\leq \dim \overline{f^{-1}(S_{\alpha})}=\dim S_{\alpha}+\dim f^{-1}(y_{\alpha})\leq n-l<n-k,
    \end{equation*}
    contradicting $\dim V=n-k$.
    
    If  $d-2l<\dim S_{\alpha}\le d-2k-1$, then $(iv)$ gives
    $$\dim f^{-1}(y_{\alpha})\leq n-\frac{d}{2}-\frac{\dim S_{\alpha}}{2}.$$
    It follows that
    \begin{equation*}
        \dim V\leq \dim \overline{f^{-1}(S_{\alpha})}=\dim S_{\alpha}+\dim f^{-1}(y_{\alpha})\leq n-k-\frac{1}{2}, 
    \end{equation*}
   again contradicting $\dim V=n-k.$
    
   Thus no such $V$ exists, and $L$ is $d$-lef below degree $k+1$. The induction proves $(i)$.
\end{proof}

\subsection{Perverse filtration and lefness}
Throughout this subsection, we fix integers $d,l$ satisfying $0\le 2l-1\le d\leq n$ and a basepoint-free line bundle $L$.
Unless otherwise stated, we assume that $L$ is $d$-lef below degree $l$.
We write $$f:X\to Y\subset \mathbb{P}(H^0(X,L))$$ for the Kodaira map associated with $L$ and
fix a Whitney stratification $Y=\bigsqcup_{\alpha \in \Lambda}S_{\alpha}$ of $Y$ adapted to $f$.
We write $\alpha\leq \beta$ if $S_{\alpha}\subset \overline{S_{\beta}}$.
Set $K=Rf_*\mathbb{R}_X[n]$, $P_t={}^p\mathcal{H}^t(K)$ and $\mathcal{L}_{\alpha,t}=\mathcal{H}^{-\dim S_{\alpha}}(P_{t|S_{\alpha}})$. Then the decomposition and semisimplicity theorems imply
\begin{equation*}
    K\cong \bigoplus_{t\in \mathbb{Z}} P_t[-t]  \cong  \bigoplus_{t\in \mathbb{Z}}\bigoplus_{\alpha \in \Lambda} \IC(\overline{S_{\alpha}},\mathcal{L}_{\alpha,t})[-t].
\end{equation*}

The morphism $f$ determines an increasing perverse filtration on $H^*(X)$. 
Our goal is to determine the subspace
$H^{k}_{\leq d-n-1}(X)$ for all $0\le k\leq 2l.$

\begin{prop}\label{deterPerv}
    For every $0\le k\le 2l-1$, we have $H^k_{\leq d-n-1}(X)=0,$ while
    \begin{equation*}
        H^{2l}_{\leq d-n-1}(X) =\Span_{\mathbb{R}}\{[V]:V\in \Irred^l(X),L^{d-2l}\cdot [V]=0\}.
    \end{equation*}
    Moreover, the classes $\{[V]:V\in \Irred^l(X),L^{d-2l}\cdot [V]=0\}$
    are linearly independent.
\end{prop}

The proof of Proposition \ref{deterPerv} consists of two steps.
First, we determine the dimensions
of $H^k_{\leq d-n-1}(X)$
for every $0\leq k\leq 2l$ in Lemmas \ref{dual}-\ref{dimcount}.
In particular, this proves the statement in degree range $0\leq k\leq 2l-1$ and 
\begin{equation*}
    \dim H^{2l}_{\leq d-n-1}(X) =\# \{V\in \Irred^l(X):L^{d-2l}\cdot [V]=0\}
\end{equation*}
It remains to identify $H^{2l}_{\leq d-n-1}(X)$. By
Lemma \ref{kerandperv},
\begin{equation*}
    \Span_{\mathbb{R}}\{ [V]: V\in\Irred^l(X), L^{d-2l}\cdot[V]=0\} \subset
    \ker L^{d-2l} \subset H^{2l}_{\leq d-n-1}(X).
\end{equation*}
Thus it suffices to show that the classes
$\{[V]: V\in\Irred^l(X), L^{d-2l}\cdot[V]=0\}$
are linearly independent. This is proved in
Lemmas \ref{linindeppartial}-\ref{linindepfull}.

The following result is well-known to experts. We include a short proof.
\begin{lem}\label{dual}
    For every $\alpha \in\Lambda$ and every $k\in \mathbb{Z}$, we have the duality $\mathcal{L}_{\alpha,k}^*\cong \mathcal{L}_{\alpha,-k}$.
\end{lem}
\begin{proof} 
    By the semisimplicity theorem, we have
    \begin{equation*}
        P_k\cong \bigoplus_{\beta \in \Lambda} \IC(\overline{S_{\beta}},\mathcal{L}_{\beta,k}).
    \end{equation*}
    By the duality 
    $\mathbb{D}_Y(P_k) \cong P_{-k}$,
    taking the Verdier dual of the above decomposition gives
    \begin{equation*}
        P_{-k}\cong \bigoplus_{\beta \in \Lambda} \IC(\overline{S_{\beta}},\mathcal{L}_{\beta,k}^*).
    \end{equation*}
    It follows that
    \begin{equation*}
        \bigoplus_{\beta \in \Lambda} \IC(\overline{S_{\beta}},\mathcal{L}_{\beta,k}^*) \cong \bigoplus_{\beta \in \Lambda} \IC(\overline{S_{\beta}},\mathcal{L}_{\beta,-k}).
    \end{equation*}
    Restricting to $S_{\alpha}$ and taking the $(-\dim S_{\alpha})$-th cohomology sheaf then gives $\mathcal{L}_{\alpha,k}^*\cong \mathcal{L}_{\alpha,-k}.$
\end{proof}

The fiber dimension estimate of $f$ imposes strong constraints on these local systems $\mathcal{L}_{\alpha,k}$. To simplify the notation, we write 
    \begin{equation*}
        b(s)=
        \begin{cases}
            n-2l-s  & \text{if } 0\leq s\leq d-2l\\
            n-d & \text{if } d-2l\leq s\leq d\\
            n-s & \text{if } d\leq s\leq n
        \end{cases}.
    \end{equation*}

\begin{lem}\label{pervrange}
    The local system $\mathcal{L}_{\alpha,k}$ vanishes unless $|k|\leq b(\dim S_{\alpha})$.
    Moreover, we have
    $$\mathcal{L}_{\alpha, b(\dim S_{\alpha})}\cong (R^{b(\dim S_{\alpha})+n-\dim S_{\alpha}}f_*\mathbb{R}_X)_{|S_{\alpha}}.$$
    
\end{lem}
\begin{proof}
    By the decomposition and semisimplicity theorems, we have
    \begin{equation}\label{fnqi2}
        K\cong \bigoplus_{j\in \mathbb{Z}}\bigoplus_{\beta\in \Lambda} \IC(\overline{S_{\beta}},\mathcal{L}_{\beta,j})[-j].
    \end{equation}
    Fix a stratum $S_{\alpha}$. Restricting (\ref{fnqi2}) to $S_{\alpha}$ and taking the $(k-\dim S_{\alpha})$-th cohomology sheaf yields
    \begin{equation*}          
    (R^{k+n-\dim S_{\alpha}}f_{*}\mathbb{R}_X)_{|S_{\alpha}} =\mathcal{H}^{k-\dim S_{\alpha}}(K)_{|S_{\alpha}}\cong
    \bigoplus_{j\in \mathbb{Z}}\bigoplus_{\beta\in \Lambda} 
    \mathcal{H}^{k-j-\dim S_{\alpha}} (\IC(\overline{S_{\beta}},\mathcal{L}_{\beta,j}))_{|S_{\alpha}}.
    \end{equation*}
    By the support condition, we have $\mathcal{H}^k (\IC(\overline{S_{\beta}},\mathcal{L}_{\beta,j}))_{|S_{\alpha}}=0$ unless
    \begin{itemize}
        \item $\beta=\alpha$ and $k=-\dim S_{\alpha}$, or
        \item $\beta>\alpha$ and $k\leq -\dim S_{\alpha}-1$.
    \end{itemize}
    Moreover, $\mathcal{H}^{-\dim S_{\alpha}} (\IC(\overline{S_{\alpha}},\mathcal{L}_{\alpha,j}))_{|S_{\alpha}}=\mathcal{L}_{\alpha,j}$.
    Consequently,
    \begin{equation}\label{fn2iu}
        (R^{k+n-\dim S_{\alpha}}f_{*}\mathbb{R}_X)_{|S_{\alpha}} \cong 
    \mathcal{L}_{\alpha,k} \oplus  \bigoplus_{\beta>\alpha} \bigoplus_{j\geq k+1} \mathcal{H}^{k-j-\dim S_{\alpha}}(\IC(\overline{S_{\beta}},\mathcal{L}_{\beta,j}))_{|S_{\alpha}}.
    \end{equation}
    In particular, $\mathcal{L}_{\alpha,k}$ appears as a direct summand of $(R^{k+n-\dim S_{\alpha}}f_{*}\mathbb{R}_X)_{|S_{\alpha}}$. By proper base change, the latter sheaf vanishes unless 
    $$k\leq \dim S_{\alpha}-n+2\dim f^{-1}(y_{\alpha})$$
    where $y_{\alpha}\in S_{\alpha}$ is an arbitrary point. 
    By the lefness assumption of $L$ and Proposition \ref{equiv}, we have
    \begin{equation*}
        \dim f^{-1} y_{\alpha} \leq
        \begin{cases}
        n - l - \dim S_{\alpha} & \text{if } 0 \leq \dim S_{\alpha} \leq d - 2l \\
        n - \frac{d}{2} - \frac{\dim S_{\alpha}}{2} & \text{if } d - 2l \leq \dim S_{\alpha} \leq d \\
        n - \dim S_{\alpha} & \text{if } d \leq \dim S_{\alpha} \leq n
        \end{cases}.
    \end{equation*}
    Hence $\mathcal{L}_{\alpha,k}$ can be nonzero only when 
    \begin{equation*}
        k\leq b(\dim S_{\alpha})=
        \begin{cases}
        n - 2l - \dim S_{\alpha} & \text{if } 0 \leq \dim S_{\alpha} \leq d - 2l \\
        n - d & \text{if } d - 2l \leq \dim S_{\alpha} \leq d \\
        n - \dim S_{\alpha} & \text{if } d \leq \dim S_{\alpha} \leq n
        \end{cases}.
    \end{equation*}
    By Lemma \ref{dual}, $\mathcal{L}_{\alpha,-k}\cong \mathcal{L}_{\alpha,k}^*.$
    The same bound therefore holds for $-k$, proving the first assertion.

    Next we show
    $$\mathcal{L}_{\alpha, b(\dim S_{\alpha})}\cong (R^{b(\dim S_{\alpha})+n-\dim S_{\alpha}}f_*\mathbb{R}_X)_{|S_{\alpha}}.$$
    By (\ref{fn2iu}), it is enough to prove that
    \begin{equation*}
        \mathcal{H}^{b(\dim S_{\alpha})-j-\dim S_{\alpha}}( \IC(\overline{S_{\beta}}, \mathcal{L}_{\beta,j}) )_{|S_{\alpha}}=0
    \end{equation*}
     whenever $\beta>\alpha$ and $j\geq b(\dim S_{\alpha})+1$.
     Since $b(-)$ is nonincreasing and $\dim S_{\beta}> \dim S_{\alpha}$, we have
     $$j\ge b(\dim S_{\alpha})+1\geq b(\dim S_{\beta})+1.$$
     Therefore $\mathcal{L}_{\beta,j}=0$
    and the desired vanishing follows.
\end{proof}

\begin{lem}\label{IH0}
    Let $S_{\alpha}$ be a stratum. Suppose that $\dim f^{-1}(y_{\alpha})\leq k$ for every $y_{\alpha}\in S_{\alpha}$. Then 
    $$IH^0(\overline{S_{\alpha}},(R^{2k}f_*\mathbb{R}_X)_{|S_{\alpha}})^*)
    \cong H_{2k+2\dim S_{\alpha}}^{\text{BM}}(\overline{f^{-1}(S_{\alpha})}).$$
    Moreover, 
    $H_{2k+2\dim S_{\alpha}}^{\mathrm{BM}} ( \overline{f^{-1}(S_{\alpha})} )$
    has a basis consisting of the fundamental classes of the $(k+\dim S_{\alpha})$-dimensional irreducible components of $\overline{f^{-1}(S_{\alpha})}$.
\end{lem}
\begin{proof}
    Write $\mathcal{L}= ((R^{2k }f_*\mathbb{R}_X)_{|S_{\alpha}})^*$ and denote by $j:S_{\alpha}\to \overline{S_{\alpha}}$ denotes the open inclusion.
    By the hypercohomology spectral sequence
    \begin{equation*}
        E_{2}^{p,q}=H^{p}(\overline{S_{\alpha}},\mathcal{H}^{q}(\IC(\overline{S_{\alpha}}, \mathcal{L}))  \implies  IH^{p+q-\dim S_{\alpha}}(\overline{S_{\alpha}},\mathcal{L})
    \end{equation*}
    and Deligne's construction of the intersection complex, we have
    \begin{equation*}
        IH^0(\overline{S_{\alpha}},\mathcal{L})=H^{0}(\overline{S_{\alpha}},\mathcal{H}^{-\dim S_{\alpha}}(\IC(\overline{S_{\alpha}}, \mathcal{L})) =H^0(\overline{S_{\alpha}},j_*\mathcal{L})=H^0(S_{\alpha},\mathcal{L}).
    \end{equation*}
    Poincar\'e duality then gives
    \begin{equation*}
        IH^0(\overline{S_{\alpha}},\mathcal{L})\cong  H^{2\dim S_{\alpha}}_{c}(S_{\alpha},\mathcal{L})^*.
    \end{equation*}
    We claim that 
    \begin{equation*}
        H^{2\dim S_{\alpha}}_{c}(S_{\alpha},(R^{2k}f_*\mathbb{R}_X)_{|S_{\alpha}}) 
        \cong H_c^{2k+2\dim S_{\alpha}}(f^{-1}(S_{\alpha})).
    \end{equation*}
    To prove the claim, consider the Leray spectral sequence
    \begin{equation*}
        E^{p,q}_2=H^{p}_c(S_{\alpha},(R^{q}f_*\mathbb{R}_X)_{|S_{\alpha}})\implies H^{p+q}_c(f^{-1}S_{\alpha}).
    \end{equation*}
    Since $\dim f^{-1}(y_{\alpha})\leq k$ for every $y_{\alpha}\in S_{\alpha}$, 
    proper base change gives $(R^{q}f_*\mathbb{R}_X)_{|S_{\alpha}}=0$ unless $0\leq q\leq 2k.$
    Moreover, we have $H_c^p(S_{\alpha},-)=0$ unless $0\leq p\leq 2\dim S_{\alpha}. $
    Hence all differentials entering or leaving $E_r^{2\dim S_{\alpha},2k}$ vanish, and therefore 
    \begin{equation*} E_2^{2\dim S_{\alpha},2k} = E_{\infty}^{2\dim S_{\alpha},2k}. \end{equation*}
    Moreover, this is the only possible nonzero term on the anti-diagonal $p+q=2k+2\dim S_{\alpha}$. Consequently,
    \begin{equation*}
        E_{\infty}^{2\dim S_{\alpha},2k}=H_c^{2k+2\dim S_{\alpha}}(f^{-1}(S_{\alpha})).
    \end{equation*}
    Therefore
    \begin{equation*}
        H^{2\dim S_{\alpha}}_{c}(S_{\alpha},(R^{2k}f_*\mathbb{R}_X)_{|S_{\alpha}}) 
        \cong H_c^{2k+2\dim S_{\alpha}}(f^{-1}(S_{\alpha})).
    \end{equation*}
    This proves the claim.
    
    Combining the preceding isomorphisms and using Verdier duality, we obtain
    \begin{equation*}
        IH^0(\overline{S_{\alpha}},L)\cong H_c^{2k+2\dim S_{\alpha }}(f^{-1}(S_{\alpha}))^* \cong H^{\text{BM}}_{2k+2\dim S_{\alpha }}(f^{-1}(S_{\alpha})) \cong H_{2k+2\dim S_{\alpha}}^{\mathrm{BM}} ( \overline{f^{-1}(S_{\alpha})} )
    \end{equation*}
    Finally, since $\dim \overline{f^{-1}(S_{\alpha})} \leq k+\dim S_{\alpha},$
    the group 
    $H_{2k+2\dim S_{\alpha}}^{\mathrm{BM}} ( \overline{f^{-1}(S_{\alpha})} )$ 
    is freely generated by the fundamental classes of the $(k+\dim S_{\alpha})$-dimensional irreducible components of $\overline{f^{-1}(S_{\alpha})}$. This completes the proof.
\end{proof}

We are now ready to determine $\dim H^k_{\leq d-n-1}(X)$.

\begin{lem}\label{dimcount}
    For every $0\le k\leq 2l-1$, we have  $H^k_{\leq d-n-1}(X)=0,$
    while
    \begin{equation*}
        \dim H^{2l}_{\leq d-n-1}(X)= \#\{V\in \Irred^l(X):\nd_V(L)=\dim f(V )\leq d-2l-1\}.
    \end{equation*}
    Moreover, for every $0\leq j\leq d-2l-1$,
    \begin{equation*}
        \dim  H^{2l}_{2l-n+j}(X)= \#\{V\in \Irred^l(X):\nd_V(L)=\dim f(V)=j\}.
    \end{equation*}
\end{lem}
\begin{proof}
    By the definition of the perverse filtration and the semisimplicity theorem, we have
    \begin{align*}
        H^{k}_{\leq d-n-1}(X)
        &\cong 
        \bigoplus_{t\leq d-n-1} \mathbb{H}^{k-n-t}(Y,P_t)\\
        &\cong \bigoplus_{t\leq d-n-1} \bigoplus_{\alpha \in \Lambda} \mathbb{H}^{k-n-t} (\overline{S_{\alpha}},\IC({\overline{S_{\alpha}}},\mathcal{L}_{\alpha,t}))\\
        &= \bigoplus_{t\leq d-n-1} \bigoplus_{\alpha \in \Lambda}  
        IH^{k-n-t+\dim S_{\alpha}}(\overline{S_{\alpha}},\mathcal{L}_{\alpha,t}).
    \end{align*}
    Since $L$ is $d$-lef below degree $l$, Lemma \ref{pervrange} implies that $\mathcal{L}_{\alpha,t}=0$ unless $|t|\leq b(\dim S_{\alpha})$.
    Since $t\leq d-n-1$ and $b(s)\geq n-d$ whenever $s\geq d-2l$, only the
    strata satisfying $\dim S_{\alpha}\leq d-2l-1$ 
    can contribute. 
    For such a stratum, the possible values of $t$ satisfy
    $\dim S_{\alpha}+2l-n \leq t\leq d-n-1.$
    Therefore,
    \begin{equation*}
        H^{k}_{\leq d-n-1}(X) \cong 
         \bigoplus_{\makecell{\scriptstyle \alpha \in \Lambda\\\ \scriptstyle \dim S_{\alpha}\leq d-2l-1}} 
         \bigoplus_{\dim S_{\alpha}+2l-n \leq t\leq {d-n-1}}
         IH^{k-n-t+\dim S_{\alpha}}(\overline{S_{\alpha}},\mathcal{L}_{\alpha,t}).
    \end{equation*}
    On the other hand, we have 
    $ IH^{j}(\overline{S_{\alpha}},\mathcal{L}_{\alpha,t})=0$ for all $j<0$.
    Hence a summand can be nonzero only if $k-n-t  +\dim S_{\alpha}\geq 0$,
    or equivalently  $t\leq\dim S_{\alpha}+k-n.$
    Combining this with 
    $\dim S_{\alpha}+2l-n \leq t\leq {d-n-1},$
    we see that the only terms that can contribute satisfy
    $$\dim S_{\alpha}+2l-n \leq t\leq \dim S_{\alpha}+k-n.$$
    Consequently, the direct sum is zero if $k<2l$, while for $k=2l$
    there is only one term that can contribute for each $\alpha$, namely
    $$t=\dim S_{\alpha}+2l-n=-b(\dim S_{\alpha}).$$
    It follows that $H^k_{\leq d-n-1}=0$ for  $k\leq 2l-1$
    and
    \begin{align*}
        H^{2l}_{\leq d-n-1}(X)
        &\cong  \bigoplus_{\makecell{\scriptstyle \alpha \in \Lambda\\ \scriptstyle \dim S_{\alpha}\leq d-2l-1}} IH^0(\overline{S_{\alpha}},\mathcal{L}_{\alpha,-b(\dim S_{\alpha})})\\
        &\cong \bigoplus_{\makecell{\scriptstyle \alpha \in \Lambda\\ \scriptstyle \dim S_{\alpha}\leq d-2l-1}}  IH^0(\overline{S_{\alpha}},\mathcal{L}_{\alpha,b(\dim S_{\alpha})}^*)\\
         &\cong \bigoplus_{\makecell{\scriptstyle \alpha \in \Lambda\\ \scriptstyle \dim S_{\alpha}\leq d-2l-1}} IH^0(\overline{S_{\alpha}},(R^{2n-2l-2\dim S_{\alpha}}f_*\mathbb{R}_X)^*)
    \end{align*}
    where the second isomorphism follows from Lemma \ref{dual}, the third from Lemma \ref{pervrange}. 
    For every stratum occurring in the above direct sum, Lemma \ref{equiv}
    gives
    \begin{equation*}
        \dim f^{-1}(y_{\alpha})
        \leq n-l-\dim S_{\alpha}.
    \end{equation*}
    Hence Lemma \ref{IH0} applied with $k=n-l-\dim S_{\alpha}$ gives
    \begin{equation*}
        IH^0(\overline{S_{\alpha}},(R^{2n-2l-2\dim S_{\alpha}}f_*\mathbb{R}_X)^*) \cong H_{2n-2l}^{\text{BM}}(\overline{f^{-1}(S_{\alpha})}).
    \end{equation*}
    Moreover, $H_{2n-2l}^{\text{BM}}(\overline{f^{-1}(S_{\alpha})})$ has a basis consisting of the
    fundamental classes of the $(n-l)$-dimensional irreducible components
    of $\overline{f^{-1}(S_{\alpha})}$. 
    Therefore
    \begin{equation*}
        \dim \bigoplus_{\makecell{\scriptstyle \alpha \in \Lambda\\\ \scriptstyle \dim S_{\alpha}\leq d-2l-1}}
        H_{2n-2l}^{\text{BM}}(\overline{f^{-1}(S_{\alpha})})
        \leq 
        \#\{V\in \Irred^l(X): \dim f(V)\leq d-2l-1\}.
    \end{equation*}
    Conversely, by Lemma \ref{rigidity}, every $V\in \Irred^l(X)$ satisfying $\dim f(V)\leq d-2l-1$ appears as an $(n-l)$-dimensional irreducible component
    of $\overline{f^{-1}(S_{\alpha})}$ for some stratum $S_{\alpha}$ with $\dim S_{\alpha}\leq d-2l-1$. 
    Hence the converse inequality also holds. Thus
    \begin{equation*}
        \dim \bigoplus_{\makecell{\scriptstyle \alpha \in \Lambda\\\ \scriptstyle \dim S_{\alpha}\leq d-2l-1}}
        H_{2n-2l}^{\text{BM}}(\overline{f^{-1}(S_{\alpha})})
        =
        \#\{V\in \Irred^l(X): \dim f(V)\leq d-2l-1\}.
    \end{equation*}
    Combining the above isomorphisms gives
    \begin{equation*}
        \dim H^{2l}_{\leq d-n-1}(X)= \#\{V\in \Irred^l(X):\dim f(V )\leq d-2l-1\}.
    \end{equation*}
    
    Finally, for every $0\le j\le d-2l-1$,
    restricting the above calculation to the strata of dimension $j$ gives
    \begin{align*}
        H^{2l}_{2l-n+j}(X) &\cong
        \bigoplus_{\makecell{\scriptstyle \alpha \in \Lambda\\\ \scriptstyle \dim S_{\alpha}=j}} 
        H_{2n-2l}^{\text{BM}}(\overline{f^{-1}(S_{\alpha})}).
    \end{align*}
    Therefore
    \begin{equation*}
        \dim H^{2l}_{2l-n+j}(X)=\#\{V\in \Irred^l(X):\nd_V(L)=\dim f(V)=j\}.
    \end{equation*}
    This completes the proof.
\end{proof}

Next, we prove that the fundamental classes of these subvarieties are linearly independent. We first treat a special case and then prove the general case by induction.

\begin{lem}\label{linindeppartial}
If $L$ is $(2l+1)$-lef below degree $l$, then the classes
\begin{equation*}
    \{[V]\in H^{2l}(X):V\in \Irred^l (X), L\cdot[V]=0\}
\end{equation*}
are linearly independent.
\end{lem}

\begin{proof}
    %
    This also follows from \cite[Proposition 6.2.2]{deCMdecomp}. 
    For the reader's convenience, we give an alternative proof using our previous calculations.
    
    Let $f:X\to Y$ be the Kodaira map associated with $L$. Then 
    \begin{equation*}
        L\cdot[V]=0\iff \dim f(V)=0.
    \end{equation*}
    To simplify notation, write
    \begin{equation*}
        V_0=\Span_{\mathbb{R}}\{[V]\in H^{2l}(X):V\in \Irred^l (X), \dim f(V)=0\}.
    \end{equation*}
    To prove the linear independence of the generating set
    \begin{equation*}
        \{[V]\in H^{2l}(X):V\in \Irred^l (X), \dim f(V)=0\},
    \end{equation*}
    it suffices to show that
    $$\dim V_0=\# \{V\in \Irred^l (X): \dim f(V)=0\}.$$
    We claim that $V_0=H^{2l}_{\leq 2l-n}.$
    The result then follows from Lemma \ref{dimcount}, which gives 
    \begin{equation*} 
        \dim H^{2l}_{\leq 2l-n}(X) = \#\left\{ V\in\Irred^l(X): \dim f(V)=0 \right\}. 
    \end{equation*}
   
    We first reformulate $V_0$ in a form convenient for the proof.
    By Lemma \ref{rigidity}, every $V\in\Irred^l(X)$ with $\dim f(V)=0$ is an $(n-l)$-dimensional irreducible component of $f^{-1}(S_{\alpha})$ for some zero-dimensional stratum $S_{\alpha}$.
    Consequently,
    \begin{equation*}
        V_0= \Image( \bigoplus_{\dim S_{\alpha}=0} H^{BM}_{2n-2l}(f^{-1}S_{\alpha})
        \to H^{BM}_{2n-2l}(X)\cong H^{2l}(X)).
    \end{equation*}
    By proper base change and Verdier duality, the above cycle map is induced on hypercohomology by the natural morphism
    \begin{equation*} 
        \bigoplus_{\dim S_{\alpha}=0} ( (R^{2n-2l}f_*\mathbb{R}_X)|_{S_{\alpha}} )^* [n-2l] \longrightarrow \mathbb{D}_Y(Rf_*\mathbb{R}_X[n]), 
    \end{equation*} 
    where the summands on the left are viewed as skyscraper sheaves on $Y$.
    More precisely, the cycle map above is identified with the map
    \begin{equation*}
        \mathbb{H}^{2l-n}(Y,\bigoplus_{\dim S_{\alpha}=0}
    ((R^{2n-2l}f_*\mathbb{R}_X)_{|S_{\alpha}})^*[n-2l])\to \mathbb{H}^{2l-n}(Y,\mathbb{D}_Y(Rf_*\mathbb{R}_X[n]))
    \end{equation*}
    We next describe the morphism giving rise to this map.
    Since $L$ is $(2l+1)$-lef below degree $l$, Lemma \ref{equiv} implies
    \begin{equation*} 
    \dim f^{-1}(y_{\alpha}) \leq 
    \begin{cases} n-l & \text{if } \dim S_{\alpha}=0\\ 
    n-l-\dfrac{\dim S_{\alpha}+1}{2} & \text{if } 1\leq \dim S_{\alpha}\leq 2l+1\\
    n-\dim S_{\alpha} & \text{if } 2l+1\leq \dim S_{\alpha}\leq n 
    \end{cases} .
    \end{equation*} 
    In particular, $R^qf_*\mathbb{R}_X=0$ for $q>2n-2l$ and $(R^{2n-2l}f_*\mathbb{R}_X)_{|S_{\alpha}}=0$ for $\dim S_{\alpha}>0$.
    Hence the standard truncation satisfies 
    \begin{equation*} 
    \tau_{\geq n-2l}(Rf_*\mathbb{R}_X[n]) \cong \bigoplus_{\dim S_{\alpha}=0} (R^{2n-2l}f_*\mathbb{R}_X)|_{S_{\alpha}} [2l-n]. 
    \end{equation*}
    Thus the canonical morphism 
    \begin{equation*} 
        Rf_*\mathbb{R}_X[n] \to \tau_{\geq n-2l}(Rf_*\mathbb{R}_X[n]) 
    \end{equation*} 
    induces the morphism 
    \begin{equation*} 
        \mathbb{D}_Y( \tau_{\geq n-2l}(Rf_*\mathbb{R}_X[n]) ) \to
        \mathbb{D}_Y(Rf_*\mathbb{R}_X[n]) 
    \end{equation*} 
    considered above after applying Verdier duality.

    We now identify this morphism with the canonical morphism 
    \begin{equation*}  
        {}^p\tau_{\leq 2l-n}(Rf_*\mathbb{R}_X[n]) \to Rf_*\mathbb{R}_X[n]. 
    \end{equation*}
    Consider the distinguished triangle for the standard $t$-structure
    \begin{equation*}
        \tau_{\leq n-2l-1}(Rf_*\mathbb{R}_X[n]) 
        \to 
        Rf_* \mathbb{R}_X[n]
        \to 
        \tau_{\geq n-2l}(Rf_*\mathbb{R}_X[n])
        \overset{+1}{\to}.
    \end{equation*}
    We claim that $\tau_{\leq n-2l-1}(Rf_*\mathbb{R}_X[n]) \in {}^p\mathcal{D}^{\leq n-2l-1}$ and $\tau_{\geq n-2l}(Rf_*\mathbb{R}_X[n]) \in {}^p\mathcal{D}^{\geq n-2l}.$

    For the first assertion, it is enough to verify that
    \begin{equation*}
        \mathcal{H}^{j}(\tau_{\leq n-2l-1}(Rf_*\mathbb{R}_X[n]))_{|S_{\alpha}}=0
    \end{equation*}
    for every stratum $S_{\alpha}$ and every $j>n-2l-1-\dim S_{\alpha}$.
    If $\dim S_{\alpha}=0$, this follows immediately from the definition of the standard truncation. 
    Suppose that $\dim S_{\alpha}\geq 1$. 
    It is then enough to prove that $(R^{q}f_*\mathbb{R}_X)_{|S_{\alpha}}=0$ for all $q\geq 2n-2l-\dim S_{\alpha}.$
    From the fiber-dimension estimate above, we have 
    \begin{equation*} 
        \dim f^{-1}(y_{\alpha}) \leq n-l-\frac{\dim S_{\alpha}+1}{2},
    \end{equation*}
    whenever $\dim S_{\alpha}\geq 1. $
    By proper base change $(R^{j}f_*\mathbb{R}_X)_{|S_{\alpha}}=0$ for all $j\geq 2n-2l-\dim S_{\alpha}.$
    Therefore $$\tau_{\leq n-2l-1}(Rf_*\mathbb{R}_X[n])\in {}^p\mathcal{D}^{\leq n-2l-1}.$$

    For the second assertion, we have already shown that
    \begin{equation*}
        \tau_{\geq n-2l}(Rf_*\mathbb{R}_X[n])=\bigoplus_{\dim S_{\alpha}=0}(R^{2n-2l}f_*\mathbb{R}_X)_{|S_{\alpha}}[2l-n] .
    \end{equation*}
    Each summand is a skyscraper perverse sheaf shifted by $[2l-n]$.
    Hence 
    $$\tau_{\geq n-2l}(Rf_*\mathbb{R}_X[n]) \in {}^p\mathcal{D}^{\geq n-2l}. $$

    Applying Verdier duality to the standard truncation triangle gives the distinguished triangle
    \begin{equation*}
        \mathbb{D}_Y(\tau_{\geq n-2l}(Rf_*\mathbb{R}_X[n]))
        \to
        \mathbb{D}_Y(Rf_* \mathbb{R}_X[n])
        \to
        \mathbb{D}_Y(\tau_{\leq n-2l-1}(Rf_*\mathbb{R}_X[n]))
        \overset{+1}{\to}.
    \end{equation*}
    Since Verdier duality exchanges the two halves of the perverse $t$-structure,
    the first and the third terms belong to ${}^p\mathcal{D}^{\leq 2l-n}$ and ${}^p\mathcal{D}^{\geq 2l-n+1}$ respectively.
    Using the canonical self-duality 
    $$\mathbb{D}_Y(Rf_* \mathbb{R}_X[n])\cong Rf_* \mathbb{R}_X[n],$$ 
    the uniqueness of the perverse truncation triangle identifies
    \begin{equation*}
        \mathbb{D}_Y(\tau_{\geq n-2l}(Rf_*\mathbb{R}_X[n]))
        \to
        \mathbb{D}_Y(Rf_* \mathbb{R}_X[n])\cong Rf_* \mathbb{R}_X[n]
    \end{equation*}
    with the canonical perverse truncation morphism
    \begin{equation*}
        ^{p}\tau_{\leq 2l-n}(Rf_*\mathbb{R}_X[n])\to Rf_* \mathbb{R}_X[n].
    \end{equation*}
    Taking $\mathbb{H}^{2l-n}(Y,-)$ and using the definition of the perverse filtration, the image of this morphism is $H^{2l}_{\leq 2l-n}(X).$
    On the other hand, by the construction above, the same image is precisely $V_0$. Therefore, 
    \begin{equation*}
        V_0=H^{2l}_{\leq 2l-n}(X).
    \end{equation*}
    This completes the proof.
\end{proof}

We will use the following version of Bertini's theorem; see, for example, \cite{lazarsfeldPosI}.
\begin{lem}\label{Bertini}
    Let $L$ be a basepoint-free line bundle on $X$, and let $V_1,...,V_m\subset X$ be subvarieties. 
    There exists a Zariski dense open subset $U\subset |L|$ such that every hypersurface $H\in U$ satisfies the following properties.
    \begin{enumerate}
    \item For every $i$, the intersection $V_i\cap H$ is proper and
    pure-dimensional. If $$V_i\cap H=W_1\cup\cdots\cup W_N$$
    is the decomposition into irreducible components, then
    $[V_i]_{X|H}=\sum_{j=1}^N[W_j]_H.$
    \item If $V_i$ is smooth, then $V_i\cap H$ is smooth.
    \item If $\nd_{V_i}(L)\ge2$, then $V_i\cap H$ is irreducible.
    \end{enumerate}
\end{lem}

\begin{lem}\label{Lpreserving}
    Let $d,l$ be integers satisfying $l\geq 1$ and $ 2l+1\leq d\leq n$, 
    and let $L$ be a basepoint-free line bundle that is $d$-lef below degree l. Then for every $k\leq d-2l+1$ and general hypersurfaces $H_1,...,H_k\in |L|$, the subvariety
    $X_k=H_1\cap \cdots\cap H_k$
    is smooth and irreducible, and the restriction $L_{|X_k}$ is $(d-k)$-lef below degree $l$.
\end{lem}
\begin{proof}
    As before, let $f:X\to Y\subset \mathbb{P}(H^{0}(X,L))$ be
    the Kodaira map associated with $L$. 
    Then the system $|L|$ is identified with the pullback of the hyperplane linear system: $|L|=f^*|\mathcal{O}_{\mathbb{P}(H^{0}(X,L))}(1)|.$
    Let $Y=\bigsqcup_{\alpha\in \Lambda}S_{\alpha}$ be a Whitney stratification of $Y$ adapted to $f$.
    By Lemma \ref{equiv}, for every stratum $S_\alpha$ and every point
    $y_\alpha\in S_\alpha$, we have
    \begin{equation*}
    \dim f^{-1}(y_\alpha)\leq
    \begin{cases}
     n-l-\dim S_\alpha &  \text{if }0\leq \dim S_\alpha\leq d-2l\\
     n-\frac{d}{2}-\frac{\dim S_\alpha}{2} & \text{if }d-2l\leq \dim S_\alpha\leq d\\ 
     n-\dim S_\alpha &  \text{if }d\leq \dim S_\alpha\leq n
    \end{cases}.
    \end{equation*}
    Choose general hyperplanes $A_1,...,A_k \in |\mathcal{O}_{\mathbb{P}(H^{0}(X,L))}(1)|$ such that every nonempty intersection 
    $S_{\alpha}\cap A_1\cap\cdots \cap A_k$
    is transversal for every stratum $S_{\alpha}$. 
    Let $H_i=f^{-1}A_i$ be the corresponding members of $|L|$, and set
    $X_k=H_1\cap\cdots \cap H_k.$ 
    Since $L$ is $d$-lef below degree $l$ and $l\geq 1$, we have $\nd(L)=\dim Y\geq d.$
    We prove inductively that $X_s$ is smooth and irreducible for every $1\leq s\leq k$. Suppose that $X_s$ is smooth and irreducible. 
    Since the hyperplanes are general, we have $f(X_s)=Y\cap A_1\cap\cdots\cap A_s$ and hence 
    \begin{equation*} 
        \nd_{X_s}(L) \geq \dim f(X_s) = \dim Y-s \geq d-s. 
    \end{equation*} 
    For $s\leq k-1$, we have 
    \begin{equation*} 
        d-s \geq d-k+1 \geq 2l \geq 2. 
    \end{equation*} 
    Therefore, Lemma \ref{Bertini} shows that the general intersection $X_{s+1}=X_s\cap H_{s+1}$ is smooth and irreducible. 
    Thus $X_k$ is smooth and irreducible by induction.
    
    Let $Y_k=Y\cap A_1\cap \cdots\cap A_k.$ 
    By the transversality above, the connected components of
    $S_\alpha\cap A_1\cap\cdots\cap A_k$
    form a Whitney stratification of $Y_k$ adapted to the restricted
    morphism $f_{|X_k}:X_k\longrightarrow Y_k .$
    Moreover, if $y_\alpha\in S_\alpha\cap A_1\cap\cdots\cap A_k,$
    then $(f|_{X_k})^{-1}(y_\alpha)=f^{-1}(y_\alpha).$
    And  we have 
    $$\dim(S_\alpha\cap A_1\cap\cdots\cap A_k)
        =\dim S_\alpha-k$$
        by transversality.
    Therefore, the fiber-dimension estimates above become
    \begin{equation*}
        \dim f^{-1} y_{\alpha} \leq
        \begin{cases}
        n-k - l - \dim S_{\alpha}\cap A_1\cap \cdots \cap A_k & \text{if } 0 \leq \dim S_{\alpha}\cap A_1\cap \cdots \cap A_k \leq d-k - 2l \\
        n-k - \frac{d-k}{2} - \frac{\dim S_{\alpha}\cap A_1\cap \cdots \cap A_k}{2} & \text{if } d -k- 2l \leq \dim S_{\alpha}\cap A_1\cap \cdots \cap A_k \leq d -k\\
        n-k - \dim S_{\alpha}\cap A_1\cap \cdots \cap A_k & \text{if } d -k\leq \dim S_{\alpha}\cap A_1\cap \cdots \cap A_k \leq n-k
        \end{cases}.
    \end{equation*}
    Applying Lemma \ref{equiv} once again, we conclude that the restricted
    line bundle $L_{|X_k}$
    is $(d-k)$-lef below degree $l$. This completes the proof.

\end{proof}

\begin{lem}\label{linindepfull}
    Let $d,l$ be integers satisfying $l\geq 1$ and $2l+1\leq d\leq n$ and let $L$ be a basepoint-free line bundle. If $L$ is $d$-lef below degree l, then the classes
    \begin{equation*}
        \{[V]\in H^{2l}(X):V\in \Irred^l (X), L^{d-2l}\cdot [V]=0\}
    \end{equation*}
    are linearly independent.
\end{lem}
\begin{proof}
    Let $f:X\to Y$ be the Kodaira map associated with $L$. Then 
    \begin{equation*}
        L^{d-2l}\cdot [V]=0\iff \dim f(V)\leq d-2l-1.
    \end{equation*}
    Reorder these subvarieties as 
    $$V_{0,1},...,V_{0,M_0},V_{1,1},...,V_{1,M_1},...,V_{d-2l-1,1},...,V_{d-2l-1,M_{d-2l-1}}$$ 
    such that $\dim f(V_{i,j})=i$ for every $i,j$.
    Suppose that there is a linear relation 
    \begin{equation*}
        \sum_{i=0}^{d-2l-1}\sum_{j=1}^{M_i}a_{i,j}[V_{i,j}]=0.
    \end{equation*}
    Choose general hypersurfaces $H_1,...,H_{d-2l-1}\in |L|$, and set $X_{s}= H_1\cap \cdots \cap H_{s}$ for every $1\leq s\le d-2l-1.$
    We choose the hypersurfaces sufficiently generally so that Lemma \ref{Bertini} applies simultaneously to all the subvarieties $V_{i,j}$ and their pairwise intersections.
    
    By Lemma \ref{Lpreserving},  the variety $X_{d-2l-1}$  is smooth and irreducible, and the restriction $L_{|X_{d-2l-1}}$ is $(2l+1)$-lef below degree $l$. 
    Since $X_{d-2l-1}$ is the inverse image under $f$ of a general linear subspace of codimension $d-2l-1$, we have $V_{i,j}\cap X_{d-2l-1}=\emptyset$ unless $i=d-2l-1$.
    For every $1\leq j\leq M_{d-2l-1}$, write 
    $$V_{d-2l-1,j}\cap X_{d-2l-1}=W_{j,1}\cup \cdots \cup W_{j,N_j}$$
    as the decomposition into irreducible components.
    By Lemma \ref{Bertini}, the intersection is proper and pure-dimensional, and hence every $W_{j,k}$ has codimension $l$ in $X_{d-2l-1}$.
    Moreover, we have $$\dim f_{|X_{d-2l-1}}(W_{j,k})=0.$$
    
    Restricting the above linear relation to $X_{d-2l-1}$, we obtain
    \begin{equation*}
        \sum_{j=1}^{M_{d-2l-1}}a_{d-2l-1,j}[V_{d-2l-1,j}\cap X_{d-2l-1}]_{X_{d-2l-1}}=\sum_{j=1}^{M_{d-2l-1}}\sum _{k=1}^{N_j}a_{d-2l-1,j}[W_{j,k}]_{X_{d-2l-1}}=0.
    \end{equation*}
    Since $L_{|X_{d-2l-1}}$ is $(2l+1)$-lef below degree $l$ and $L_{|X_{d-2l-1}}\cdot [W_{j,k}]_{X_{d-2l-1}}$=0 for all $j,k$,  Lemma \ref{linindeppartial} implies that the classes
    \begin{equation*}
        \{[W_{j,k}]_{X_{d-2l-1}}: 1\leq j \leq M_{d-2l-1},1\leq k\leq N_{j}\}
    \end{equation*}
    are linearly independent.
    Therefore $a_{d-2l-1,j}=0$ for every $1\leq j \leq M_{d-2l-1}$. Applying the same argument successively to the varieties
    $X_{d-2l-2},...,X_1,X_0$, we conclude that $a_{i,j}=0$ for all $i,j$. Hence the classes
    \begin{equation*}
        \{[V]\in H^{2l}(X):V\in \Irred^l (X), \dim f(V)\leq d-2l-1\}
    \end{equation*}
    are linearly independent, which completes the proof.
\end{proof}

\begin{proof}[Proof of Proposition \ref{deterPerv}]
    By Lemma \ref{dimcount}, we have $H^k_{\leq d-n-1}(X)=0,$
    for every $0\leq k\leq 2l-1$,
    and
    \begin{equation*}
        \dim H^{2l}_{\leq d-n-1}(X) = \#\{V\in\Irred^l(X): L^{d-2l}\cdot[V]=0 \}.
    \end{equation*}
    On the other hand, Lemma \ref{kerandperv} gives
    \begin{equation*}
        \Span_{\mathbb{R}} \{ [V]: V\in\Irred^l(X), L^{d-2l}\cdot[V]=0\}
        \subset H^{2l}_{\leq d-n-1}(X).
    \end{equation*}
    By Lemma \ref{linindepfull}, the classes $\{[V]:V\in\Irred^l(X),L^{d-2l}\cdot[V]=0\}$
    are linearly independent. Hence the subspace on the left has the same
    dimension as $H^{2l}_{\leq d-n-1}(X)$, and therefore
    \begin{equation*}
        H^{2l}_{\leq d-n-1}(X)=\Span_{\mathbb{R}} \{[V]:V\in\Irred^l(X),L^{d-2l}\cdot[V]=0\}.
    \end{equation*}
\end{proof}


\subsection{Kernel characterizations}
In this subsection, we prove Theorem \ref{purelefcharintro} and \ref{purekerintro}.

\begin{thrm}\label{hardlefsupercrit}
    Let $d,l$ be integers satisfying $0\leq 2l-1\leq d\leq n$ and let $L$ be a semiample line bundle.
    If $L$ is $d$-lef below degree $l$, then
    $$L^{d-k}:H^{k}(X)\rightarrow H^{2d-k}(X)$$ 
    is injective for every $0\leq k\leq 2l-1$.
    If moreover $2l+1\leq d$, then we have
    \begin{equation*}
        \ker (L^{d-2l}:H^{2l}(X)\rightarrow H^{2d-2l}(X))=\Span_{\mathbb{R}}\{[V]: V\in \Irred^l(X), L^{d-2l}\cdot [V]=0\}.
    \end{equation*}
    Moreover, the classes $\{[V]: V\in \Irred^l(X), L^{d-2l}\cdot [V]=0\}$
    are linearly independent.
\end{thrm}
\begin{proof}
    Since $L$ is semiample, after replacing $L$ by a sufficiently divisible positive multiple, we may assume that $L$ is basepoint-free. This replacement does not affect the lefness assumption, the kernels of the Lefschetz operators, or the condition 
    $L^{d-2l}\cdot[V]=0. $

    By Lemma \ref{kerandperv},
    \begin{equation*}
        \ker (L^{d-k}:H^{k}(X)\rightarrow H^{2d-k}(X)) \subset H^k_{\leq d-n-1}(X).
    \end{equation*}
    Proposition \ref{deterPerv} gives $H^k_{\leq d-n-1}(X)=0$ for every $0\leq k\leq 2l-1$. Hence
    $$L^{d-k}:H^{k}(X)\rightarrow H^{2d-k}(X)$$ 
    is injective for every $0\leq k\leq 2l-1$.

    Suppose now that $2l+1\leq d$. We have
    \begin{equation*}
        \Span_{\mathbb{R}}\{[V]: V\in \Irred^l(X), L^{d-2l}\cdot [V]=0\} \subset \ker L^{d-2l}\cap H^{2l}(X)\subset H^{2l}_{\leq d-n-1}(X).
    \end{equation*}
    By Proposition \ref{deterPerv}, the first and the last terms are equal.
    Therefore both inclusions are equalities, and hence
    \begin{equation*}
        \ker (L^{d-2l}:H^{2l}(X)\rightarrow H^{2d-2l}(X))=\Span_{\mathbb{R}}\{[V]: V\in \Irred^l(X), L^{d-2l}\cdot [V]=0\}.
    \end{equation*}
   The linear independence of these classes also follows from Proposition \ref{deterPerv}. This completes the proof.
\end{proof}

Since the Lefschetz operators are compatible with the Hodge decomposition and the fundamental classes of subvarieties are of diagonal Hodge type, we immediately obtain the following corollaries.

\begin{cor}\label{purelefchar}(=Theorem \ref{purelefcharintro})
    Let $d,l$ be integers satisfying $0\leq 2l-1\leq d\leq n$ and let $L$ be a semiample line bundle.
    Then the following statements are equivalent.
    \begin{enumerate}
        \item $L$ is $d$-lef below degree $l$.
        \item For every integer $k$ with $0\leq k\leq l-1$, the linear map
                \begin{equation*}
                    L^{d-2k}:H^{k,k}(X)\to H^{d-k,d-k}(X)
                \end{equation*}
                is injective.
        \item For all integers $p,q\geq 0$ with $ p+q\leq 2l-2$, the linear map 
            \begin{equation*}
                L^{d-p-q}: H^{p,q}(X)\to H^{d-q,d-p}(X)
            \end{equation*}
            is injective.
    \end{enumerate}
\end{cor}
\begin{proof}
    The implications $(3)\implies(2) \implies (1)$ are clear. 
    It remains to prove $(1)\implies(3)$. Since $p+q\leq 2l-2<2l-1,$ Theorem \ref{hardlefsupercrit} tensored with $\mathbb{C}$ implies that 
    \begin{equation*} 
        L^{d-p-q}:H^{p+q}(X,\mathbb{C}) \to H^{2d-p-q}(X,\mathbb{C}) 
    \end{equation*} 
    is injective. Its restriction to the summand $H^{p,q}(X)$ is therefore injective. This proves $(3)$.
\end{proof}

\begin{cor}\label{pureker}(=Theorem \ref{purekerintro})
    Let $d,l$ be integers satisfying $l\ge 1$ and  $2l+1\leq d\leq n$, and
    let $L$ be a semiample line bundle that is $d$-lef below degree $l$. Then for all integers $p,q\geq 0$ with $p+q\leq 2l$, the following statements hold.
    \begin{enumerate}
        \item If $(p,q)\neq (l,l)$, then the linear map 
            \begin{equation*}
                L^{d-p-q}: H^{p,q}(X)\to H^{d-q,d-p}(X)
            \end{equation*}
            is injective;
        \item If $(p,q)=(l,l)$, then
        \begin{equation*}
            \ker (L^{d-2l}:H^{l,l}(X)\to H^{d-l,d-l}(X))=\Span_{\mathbb{C}}\{[V]: V\in \Irred^l(X), L^{d-2l}\cdot [V]=0\}.
        \end{equation*}
    \end{enumerate} 
        Moreover, the classes $\{[V]: V\in \Irred^l(X), L^{d-2l}\cdot [V]=0\}$
        are linearly independent.
\end{cor}
\begin{proof}
    If $p+q\leq 2l-1$, the first assertion follows from Theorem \ref{hardlefsupercrit} after tensoring with $\mathbb{C}$. 
    
    It remains to consider $p+q=2l$. Tensoring the kernel characterization in Theorem \ref{hardlefsupercrit} with $\mathbb{C}$ gives 
    \begin{equation*} 
    \ker ( L^{d-2l}:H^{2l}(X,\mathbb{C}) \to H^{2d-2l}(X,\mathbb{C}) ) = \Span_{\mathbb{C}} \{ [V]: V\in\Irred^l(X),\ L^{d-2l}\cdot[V]=0 \}. 
    \end{equation*} 
    Since every fundamental class on the right-hand side has Hodge type $(l,l)$, the kernel is contained in $H^{l,l}(X)$. 
    Hence the restriction of $L^{d-2l}$ to $H^{p,q}(X)$ is injective whenever $p+q=2l$ and $(p,q)\neq(l,l)$, while for $(p,q)=(l,l)$ its kernel is exactly 
    \begin{equation*} \Span_{\mathbb{C}} \left\{ [V]: V\in\Irred^l(X),\ L^{d-2l}\cdot[V]=0 \right\}. \end{equation*} 
    The linear independence also follows from Theorem \ref{hardlefsupercrit}.
    This completes the proof.
\end{proof}

We can also state these results in a uniform way.

\begin{cor}\label{pureker(p,q)}
    Let $d,p,q\geq 0$ be integers satisfying $p+q < d\leq n$. Assume that $L$ is $d$-lef below degree $\lceil \frac{p+q}{2} \rceil$. 
    Then 
    \begin{equation*}
        \ker L^{d-p-q}\cap H^{p,q}(X)=
        \begin{cases}
            0 &\text{if } p\ne q\\
            \Span_{\mathbb{C}}\{[V]: V\in \Irred^p(X), L^{d-2p}\cdot [V]=0\}
            & \text{if }p=q
        \end{cases}.
    \end{equation*}
\end{cor}
\begin{proof}
    When $p=q=0$, the statements are trivial. We may therefore assume $p+q\geq 1$, and set $l=\left\lceil \frac{p+q}{2}\right\rceil.$

    If $p+q$ is odd, then $p+q=2l-1.$ By tensoring Theorem \ref{hardlefsupercrit} with $\mathbb{C}$, the linear map
    \begin{equation*}
        L^{d-p-q}:H^{p+q}(X,\mathbb{C}) \to H^{2d-p-q}(X,\mathbb{C})
    \end{equation*}
    is injective. Consequently,
    \begin{equation*}
        L^{d-p-q}:H^{p,q}(X) \to H^{d-q,d-p}(X)
    \end{equation*}
    is also injective.

    If $p+q$ is even, then $p+q=2l.$ Since $p+q<d$, we have $2l+1\leq d$. Hence tensoring the second statement of Theorem \ref{hardlefsupercrit} with $\mathbb{C}$ gives
    \begin{equation*}
        \ker( L^{d-2l}:H^{2l}(X,\mathbb{C}) \to H^{2d-2l}(X,\mathbb{C}) )
        = \Span_{\mathbb{C}} \{ [V]: V\in\Irred^l(X), L^{d-2l}\cdot[V]=0\}.
    \end{equation*}
    Every fundamental class appearing on the right-hand side has Hodge
    type $(l,l)$. Therefore the kernel is contained in
    $H^{l,l}(X)$. 
    It follows that
    \begin{equation*}
        \ker L^{d-2l}\cap H^{p,q}(X)=
        \begin{cases}
            0  & \text{if } p\neq q\\
            \Span_{\mathbb{C}}\{[V]:V\in\Irred^l(X),L^{d-2l}\cdot[V]=0\}
            &\text{if } p= q=l
        \end{cases}.
    \end{equation*}
    
    Thus the equality holds in all cases, completing the proof.
\end{proof}



\section{Geometric decoupling}\label{geometricdecoupling}

In the previous section, we characterized the kernels of Lefschetz operators associated with a single semiample line bundle $L$ under suitable lefness assumptions. In this section, we extend the discussion to the mixed setting and prove Theorems \ref{mixedlefcharintro} and \ref{mixedkerintro}.
The passage from the pure to the mixed setting holds more generally for Lefschetz modules. Before turning to this abstract formulation in the next section, we give a geometric proof for semiample classes. The main ideas already appear in \cite{hushangxiao2023hardlefArxiv,HuXiaoLefChar ,HuXiao25LefCharFree}. 

The proof proceeds by induction in two steps. First, we construct suitable hypersurfaces on which the restricted line bundles satisfy the positivity assumptions needed for induction; see Section~\ref{hypersurf}. 
Second, we deduce the desired statements on the original variety from those on the hypersurfaces. This step uses the Lefschetz theorems for a single class and is carried out in Section~\ref{mixedlefsec}.

\subsection{Lefness for collections of semiample classes}
In this subsection, we introduce the notion of lefness for collections of semiample classes and establish some of its basic properties.
Let $X$ be a smooth projective variety over $\mathbb{C}$ of dimension $n$ and
let $\mathcal{L}=(L_1,...,L_m)$ be a collection of semiample line bundles.
For every nonempty subset $I\subset [m]$, 
we write $L_I=\sum_{i\in I}L_i.$
For every $r\in [m]$, we denote by 
$$\mathcal{L}_{\backslash r}=(L_1,...,L_{r-1},L_{r+1},...,L_m)$$ 
the collection obtained from $\mathcal{L}$ by removing $L_r$.
For every $\alpha\in \mathbb{Z}_{>0}^m$, we write 
$\mathcal{L}^{\alpha}=\prod_{i=1}^m L_i^{\alpha_i}$
for the corresponding complete intersection class. 
By an abuse of notation, we also write
$\mathcal{L}=L_1\cdots L_m.$

Recall that a semiample line bundle $L$ is $d$-lef below degree $l$ if 
\begin{equation*} 
    L^{d-2\codim_X V}\cdot[V]\neq 0 
\end{equation*} 
for every irreducible subvariety $V\subset X$ with $\codim_XV\leq l-1$. 
This definition naturally extends to collections of semiample line bundles.

\begin{defn}   
    A collection of semiample line bundles $\mathcal{L}=(L_1,...,L_{m})$ is called $d$-lef below degree $l$ if $m\leq d-2l+2$ and we have
    \begin{equation*}
        \mathcal{L}^{\alpha}\cdot [V]\neq 0
    \end{equation*}
    for every subvariety $V\subset X$ with $\codim_XV\leq l-1$ and every $\alpha\in \mathbb{Z}_{>0}^m$ satisfying $|\alpha|=d-2\codimension_X V.$
\end{defn}

The condition $m\leq d-2l+2$ ensures that the above definition is non-vacuous in every codimension under consideration. 
Indeed, if $m> d-2l+2$, then there is no $\alpha\in \mathbb{Z}_{>0}^m$ with $|\alpha|=d-2l+2$, and hence the condition for  $V\in \Irred^{l-1}(X)$ would be vacuous.

We will see in Lemma \ref{equivlefcolfree} that 
\begin{equation*}
    \mathcal{L} \text{ is } d\text{-lef below degree } l \iff L_I \text{ is } (d-m+|I|)\text{-lef below degree } l, \forall \emptyset \neq I\subset [m].
\end{equation*}
To prove this, we use a nonvanishing criterion for such complete intersection classes,  established in \cite{KK16} for semiample classes and extended in \cite{hxhardlef} to arbitrary nef classes. See also Corollary \ref{hallradonef} for a Lefschetz-module proof of this result.
For every $\alpha\in\mathbb{Z}_{>0}^m$ and every nonempty subset $I\subset[m]$, we write $\alpha_I=\sum_{i\in I}\alpha_i.$

\begin{prop}\label{hallrado}
    Let $V$ be an irreducible subvariety and let $(L_1,...,L_m)$ be a collection of nef classes. For any $\alpha\in \mathbb{Z}_{>0}^{m}$,
    the following statements are equivalent.
    \begin{enumerate}
        \item $\mathcal{L}^{\alpha}\cdot [V]\neq 0$.
        \item $\nd_V(L_I)\geq \alpha_I$ for every nonempty subset $I\subset [m]$.
    \end{enumerate}
Consequently, for every integer $k\geq m$, the following statements are equivalent.
\begin{enumerate}
    \item  $\mathcal{L}^{\alpha}\cdot [V]\neq 0$ for every $\alpha\in \mathbb{Z}_{>0}^{m}$ with $|\alpha|=k$;
    \item $\nd_V(L_I)\geq |I|+k-m$ for every nonempty subset $I\subset [m]$.
\end{enumerate} 
\end{prop}
\begin{proof}
    The first equivalence follows from \cite[Corollary 4.7]{hxhardlef}; see also Corollary \ref{hallradonef}. In the semiample case, this criterion also follows from \cite[Theorems 2.14 and 2.19]{KK16}. We then explain the consequence. By the first part, $\mathcal{L}^{\alpha}\cdot [V]\neq 0$ for every $\alpha\in \mathbb{Z}_{>0}^{m}$ with $|\alpha|=k$ is equivalent to $\nd(L_I)\geq \alpha_I$ for every such $\alpha$ and every nonempty subset $I\subset [m]$. 
    Since $$\max_{|\alpha|=k}\alpha_I =|I|+k-m,$$ the result follows.
\end{proof}

This motivates the following notion of numerical dimension for a collection of nef classes.

\begin{defn}
    Let $\mathcal{L}=(L_1,\ldots,L_m)$ be a collection of nef classes. For every irreducible subvariety $V\subset X$, we define the numerical dimension of $\mathcal{L}$ with respect to $V$ by
    \begin{equation*}
        \nd_V(\mathcal{L})=\min_{\emptyset \neq I\subset [m]}\{\nd_V(L_I)+m-|I|\}.
    \end{equation*}
\end{defn}

When $\mathcal{L}=(L,\ldots,L)$ consists of copies of a single nef class $L$, we recover the usual numerical dimension: $\nd_V(\mathcal{L})=\nd_V(L).$
By definition,
$$\nd_V(\mathcal{L})\geq k\iff \nd_V(L_I)\geq |I|+k-m, \forall \emptyset \neq I\subset [m] .$$
Combining this observation with Proposition \ref{hallrado}, we obtain the following equivalent characterizations of lefness for a collection.

\begin{lem}\label{equivlefcolfree}
    Let $d,l$ be integers satisfying $0\leq 2l-1\leq d\leq n$ and 
    let $\mathcal{L}=(L_1,...,L_m)$ be a collection of semiample line bundles with $1\leq m\leq d-2l+2$. Then the following statements are equivalent.
    \begin{enumerate}
        \item $\mathcal{L}$ is $d$-lef below degree $l$;
        \item $L_I$ is $(|I|+d-m)$-lef below degree $l$ for every nonempty subset $I\subset [m]$;
        \item For every $V\subset X$ with $\codimension _XV\leq l-1$, we have $\nd_V(\mathcal{L})\geq d-2\codimension_XV.$
    \end{enumerate}
\end{lem}
\begin{proof}
    $(1)\iff(2)$ 
    By definition, $\mathcal{L}$ is $d$-lef below degree $l$ if and only if
    \begin{equation*}
        \mathcal{L}^{\alpha}\cdot [V]\neq 0
    \end{equation*}
    for every subvariety $V\subset X$ with $\codimension_X V\leq l-1$ and every $\alpha \in \mathbb{Z}_{>0}^m$ with $|\alpha|=d-2\codimension_XV$.
    By Proposition \ref{hallrado}, this is equivalent to 
    $$\nd_V(L_I)\geq |I|+d-m-2\codimension_XV$$ 
    for every such $V$ and every nonempty subset $I\subset[m]$.
    By the definition of lefness for a single class, this is equivalent to saying that $L_I$ is $(|I|+d-m)$-lef below degree $l$ for every nonempty subset $I\subset[m]$.

    $(2)\iff(3)$ This follows directly from
    \begin{equation*}
        \nd_V(\mathcal{L})\geq d-2\codimension_X V \iff \nd_V(L_I)\geq |I|+d-m-2\codimension_X V, \forall \emptyset \neq I\subset [m].
    \end{equation*}
\end{proof}

We next study how lefness behaves when classes are added to the collection.
Denote by $$(L_1[\alpha_1],...,L_{m}[\alpha_m])$$ the collection in which $L_i$ is repeated $\alpha_i$ times for every $i$.

\begin{lem}\label{addgeom}
    Let $d,l$ be integers satisfying $0\leq 2l-1\leq d\leq n$ and 
    let $\mathcal{L}=(L_1,...,L_m)$ be a collection of semiample line bundles. 
    Suppose that $\mathcal{L}$ is $d$-lef below degree $l.$ Then the following statements hold.
    \begin{enumerate}
        \item If $m<d-2l+2$, then for every semiample line bundle $L$ that is $(d-m)$-lef below degree $l$, the collection $(L_1,...,L_m,L)$ is $d$-lef below degree $l$.
        \item For every $\alpha=(\alpha_1,\ldots,\alpha_m)\in\mathbb{Z}_{>0}^m$ with $|\alpha|\leq d-2l+2$, the collection $(L_1[\alpha_1],\ldots,L_m[\alpha_m])$ is $d$-lef below degree $l$.
        \item If $d\leq n-1$, then for every ample line bundle $A$, the
        collection $(L_1,\ldots,L_m,A)$ is $(d+1)$-lef below degree $l$.
    \end{enumerate}
\end{lem}
\begin{proof}
    $(1)$ 
    By Lemma \ref{equivlefcolfree}, for every nonempty subset $I\subset[m]$, the line bundle $L_I$ is $(d-m+|I|)$-lef below degree $l$.
    In particular, it is also $(d-m+|I|-1)$-lef below degree $l$. Moreover, $L$ is $(d-m)$-lef below degree $l$ by assumption.
    It remains to consider sums of the form $L_I+L$, where $\emptyset\neq I\subset[m]$. For every irreducible subvariety $V\subset X$ with $\codim_X V\leq l-1$, we have 
    \begin{equation*} 
        \nd_V(L_I+L) \geq \nd_V(L_I) \geq d-m+|I|-2\codim_X V. 
    \end{equation*} 
    Hence $L_I+L$ is $(d-m+|I|)$-lef below degree $l$. Applying Lemma \ref{equivlefcolfree} to the collection $(L_1,\ldots,L_m,L)$ proves $(1)$.

    $(2)$ The collection
    $(L_1[\alpha_1],\ldots,L_m[\alpha_m])$ 
    is obtained from $\mathcal{L}$ by successively adding copies of the line bundles $L_i$.
    By Lemma \ref{equivlefcolfree}, each $L_i$ is $(d-m+1)$-lef below degree $l$. 
    We then apply $(1)$ successively, and the repeated collection remains $d$-lef below degree $l$. 

    $(3)$ By Lemma \ref{equivlefcolfree}, it suffices to show that 
    $A$ is $(1+d-m)$-lef below degree $l$,
    $L_I$ is $(|I|+d-m)$-lef below degree $l$ and $L_I+A$ is $(|I|+d-m+1)$-lef below degree $l$ for every nonempty subset $I\subset[m]$.
    The first and third assertions follow from the ampleness of $A$ and $L_I+A$ respectively, while
    the second assertion follows immediately from the assumption that
    $\mathcal{L}$ is $d$-lef below degree $l$.
\end{proof}

As in Corollary \ref{finiteness}, we have the following finiteness result for a collection.
\begin{lem}\label{fincol}
    Let $d,l$ be integers satisfying $0\le 2l-1\leq d\leq n$ and let
    $\mathcal{L}=(L_1,...,L_m)$ be a collection of semiample line bundles that is $d$-lef below degree $l$. Then for every $\alpha\in \mathbb{Z}_{>0}^{m}$ with $|\alpha|=d-2l+1$,
    the set $\{V\in \Irred^l(X): \mathcal{L}^{\alpha}\cdot [V]=0\}$
    is finite.
\end{lem}
\begin{proof}
    By Proposition \ref{hallrado}, $\mathcal{L}^{\alpha}\cdot [V]=0$ if and only if 
    $L_I^{\alpha_I}\cdot [V]=0$ for some $I\subset [m]$. 
    Therefore
    \begin{equation*}
        \{V\in \Irred^l(X): \mathcal{L}^{\alpha}\cdot [V]=0\}=\bigcup_{\emptyset \neq I\subset [m]} \{V\in \Irred^l(X): L_I^{\alpha_I}\cdot [V]=0\}.
    \end{equation*}
    It remains to prove the finiteness of each set on the right-hand side.
    By Lemma \ref{equivlefcolfree}, $L_I$ is $(|I|+d-m)$-lef below degree $l$. Since 
    $$\alpha_I\leq d-2l+1-(m-|I|)\leq d-m+|I|-2l+1,$$
    we have  $\alpha_I+2l-1  \leq |I|+d-m . $
    Hence $L_I$ is $(\alpha_I+2l-1)$-lef below degree $l$. Corollary \ref{finiteness} then implies that the set $\{V\in \Irred^l(X): L_I^{\alpha_I}\cdot [V]=0\}$
    is finite. This completes the proof.
\end{proof}

\subsection{Restriction to general hypersurfaces}\label{hypersurf}
In this subsection, we study the preservation of lefness under restriction to a general hypersurface. 
Let $V\subset X$ be a subvariety. We denote by $[V]_{X,hom}$ the fundamental class of $V$ in the homology group of $X$.
We will use the following consequence of \cite[Lemma 2.6]{HHMWW}; see also \cite[Lemma 2.13]{HuXiao25LefCharFree}.

\begin{lem}\label{homfiber}
    Let $L$ be a basepoint-free line bundle on $X$, and let $W\subset X$ be an irreducible subvariety. For a general member $Y\in|L|$ and every irreducible component $V$ of $Y\cap W$, there exists a constant $c>0$ such that $[V]_{W,hom}=cL_{|W}\cap [W]_{W,hom}$ in $H_*(W)$
\end{lem}

The following proposition generalizes \cite[Proposition 3.3]{hushangxiao2023hardlefArxiv} to the setting of lefness below a fixed degree. In their proof, the authors use Hironaka's principle of counting constants to bound the dimension of the loci where the fibers have large dimension. The same method can be adapted to the present setting. Here we give a different proof.

\begin{prop}\label{lefpres}
    Let $d,l$ be integers satisfying $0\leq 2l-1\leq  d\leq n-1$, and let $L,M$ be basepoint-free line bundles. 
    Suppose that $\nd(M)\geq 2$, 
    $L$ is $d$-lef below degree $l$ and 
    $L+M$ is $(d+1)$-lef below degree $l$.
    Then for a general hypersurface $Y\in |M|$, the restriction
    $L_{|Y}$ is $d$-lef below degree $l$.
\end{prop}
\begin{proof}
    Let $f$ be the Kodaira map associated with $L$.
    By Bertini's theorem and Lemma \ref{homfiber}, a general hypersurface $Y\in|M|$ can be chosen such that
    \begin{itemize}
        \item $Y$ is a smooth irreducible hypersurface;
        \item For every integer $n-d\leq k\leq n-d+l$,
         the hypersurface $Y$ intersects every irreducible component $V$ of 
         $\{x\in X: \dim(f^{-1}f(x),x)\geq k\}$ 
         properly. 
         Moreover, for every irreducible component $W$ of $V\cap Y$, there is a constant $c_W>0$ such that 
        $[W]_{V,hom}=c_W M_{|V}\cap [V]_{V,hom}$ in $H_*(V).$
    \end{itemize}

    By contradiction, suppose that $L_{|Y}$ is not $d$-lef below degree $l$.
    Then there exists an irreducible subvariety $W\subset Y$ with $\codimension_Y W\leq l-1$ such that 
    $$L_{|Y}^{d-2\codimension_Y W}\cdot [W]_Y=0.$$
    Since $\codimension_Y W= \codimension_X W-1$, the projection formula
    gives 
    \begin{equation}\label{a1}
        L^{d+2-2\codimension_X W}\cdot [W]_X=0.
    \end{equation}
    Since $f$ is the Kodaira map associated with $L$, the equality \eqref{a1} implies 
    \begin{equation*}  
        \dim f(W) \leq d+1-2\codim_X W. 
    \end{equation*}
    Hence $W$ is contained in
    \begin{equation*}
        Z:=\{x\in X:\dim (f^{-1}f(x),x)\geq 2n-d-\dim W-1\}.
    \end{equation*}
    Since $2n-d-\dim W-1  = n-d+\codim_Y W$ and $\codim_Y W\le l-1$,
    we have 
    \begin{equation*} 
        n-d \leq 2n-d-\dim W-1 \leq n-d+l-1. 
    \end{equation*} 
    Therefore Proposition \ref{equiv}$(ii)$ gives 
    $$\dim Z\leq\dim W+1.$$
    Let $V$ be an irreducible component of $Z$ containing $W$. 
    Since $Y$ intersects $V$ properly, 
    \begin{equation*} 
        \dim V\cap Y=\dim V-1. 
    \end{equation*} 
    On the other hand, $W\subset V\cap Y$ implies
    \begin{equation*} 
        \dim V\cap Y \geq \dim W \geq \dim Z-1. 
    \end{equation*} 
    Hence $V$ is a maximal-dimensional component of $Z$ and $W$ is an irreducible component of $V\cap Y$.
    
    By the choice of $Y$ and Lemma \ref{homfiber}, we have
    \begin{equation*} 
        [W]_{V,\mathrm{hom}} = c_W\,M_{|V}\cap[V]_{V,\mathrm{hom}} 
    \end{equation*} 
    for some $c_W>0$. Pushing forward to $X$ and using Poincar\'e duality gives
    \begin{equation*}
        [W]_X=c_W M\cdot [V]_X.
    \end{equation*}
    Together with (\ref{a1}), this yields
    \begin{equation}\label{a2}
        L^{d+2-2\codimension_X W}\cdot M\cdot [V]_X=0
    \end{equation}
    On the other hand, since $\codim_X V=\codim_X W-1=\codim_Y W\leq l-1,$
    the lefness assumptions on $L$ and $L+M$ give
    \begin{equation*}
        (L+M)^{d+1-2\codimension_XV}\cdot [V]_X \neq 0 \text{ and } L^{d-2\codimension_X V}\cdot [V]_X\neq0.
    \end{equation*}
    Equivalently,
    \begin{equation*}
        (L+M)^{d+3-2\codimension_X W}\cdot [V]_X\neq 0 \text{ and }L^{d+2-2\codimension_X W}\cdot [V]_X \neq 0.
    \end{equation*}
    Moreover, $$M\cdot[V]_X = \frac{1}{c_W}[W]_X \neq 0.$$ 
    Then Proposition \ref{hallrado} yields 
    $$L^{d+2-2\codimension_X W}\cdot M\cdot [V]_X\neq 0,$$ contradicting  (\ref{a2}). This completes the proof.
\end{proof}

The preservation of lefness for a collection then follows immediately by Lemma \ref{equivlefcolfree}.

\begin{cor}\label{lefprecoll}
    Let $d,l$ be integers satisfying $0\le 2l-1\leq d\leq n$ and let $\mathcal{L}=(L_1,...,L_{m})$ be a collection of basepoint-free line bundles. Suppose that  $2\leq m\leq d-2l+2$ and that $\mathcal{L}$ is $d$-lef below degree $l$. 
    Then for every $r\in [m]$ with $\nd(L_r)\geq 2$, a general hypersurface $Y\in |L_r|$ is smooth and irreducible, and satisfies that the restricted collection 
    $$\mathcal{L}_{\backslash r|Y}=(L_{1|Y},...,L_{r-1|Y},L_{r+1|Y},...,L_{m|Y})$$ 
    is $(d-1)$-lef below degree $l$;
\end{cor}
\begin{proof} 
    By Lemma \ref{Bertini}, a general hypersurface $Y\in |L_r|$ is smooth and irreducible.
    Lemma \ref{equivlefcolfree} implies that 
    $L_I$ is $(|I|+d-m)$-lef below degree $l$ and 
    $L_I+L_r $ is $(|I|+1+d-m)$-lef below degree $l$ for every nonempty subset $I\subset [m]\backslash\{r\}$. 
    Hence Proposition \ref{lefpres} implies that for a general hypersurface $Y\in |L_r|$, $L_{I|Y}$ is $(|I|+d-m)$-lef below degree $l$. 
    Again by Lemma \ref{equivlefcolfree}, the collection $\mathcal{L}_{\backslash r|Y}$ is $(d-1)$-lef below degree $l$.
    This completes the proof.
\end{proof}

This preservation result allows us to generalize the linear independence result Lemma \ref{linindepfull} for a single line bundle to the mixed setting.

\begin{lem}\label{mixedlinindep}
Let $d,l$ be integers satisfying $l\ge 1$ and $2l+1\leq d\leq n$ and let $\mathcal{L}=(L_1,\ldots,L_{d-2l})$ be a collection of basepoint-free line bundles that is $d$-lef below degree $l$.
Then the classes
\begin{equation*}
    \{[V]:V\in\Irred^l(X),\ \mathcal{L}\cdot[V]=0\}
\end{equation*}
are linearly independent.
\end{lem}
\begin{proof}
When $d-2l=1$, the assertion is exactly
the linear independence statement in Lemma \ref{linindeppartial}.
We proceed by induction on $d-2l$.
Assume that $d-2l\ge 2$ and that the assertion holds for collections of
size $d-2l-1$. 

Suppose that there is a linear relation
\begin{equation*}
    \sum_{j=1}^N a_j[V_j]=0,
\end{equation*}
where $V_j\in\Irred^l(X)$ are distinct and satisfy $\mathcal{L}\cdot [V_j]=0$ for every $j$. 
We divide the $V_j$'s into two groups
\begin{equation*}
    I=\{j:L_{d-2l}\cdot[V_j]=0\} \text{ and } J=\{j:L_{d-2l}\cdot[V_j]\neq 0\}.
\end{equation*}
We then have $Y\cap V_i=\emptyset$ for every $i\in I$ for a general hypersurface $Y\in |L_{d-2l}|$.
On the other hand, for every $j\in J$, a general $Y$ intersects $V_j$ properly and in pure dimension.
Write $$V_j\cap Y = W_{j,1}\cup\cdots\cup W_{j,r_j} $$
as the decomposition into irreducible components. 
By Lemma \ref{Bertini}, 
\begin{equation*} 
    [V_j]_{X|Y} = \sum_{\nu=1}^{r_j}[W_{j,\nu}]_Y. 
\end{equation*}
In particular, $W_{j,\nu}\in \Irred^l(Y)$ for every $j$ and $\nu$.
We can also choose $Y$ generally so that it intersects $V_j\cap V_{j'}$ properly for all $j\neq j'$. This guarantees that $W_{j,\nu}$'s are pairwise distinct.
By Corollary \ref{lefprecoll}, for a general hypersurface $Y\in |L_{d-2l}|$,
the restricted collection $\mathcal{L}_{\backslash d-2l|Y}$ is $(d-1)$-lef below degree $l$.
We choose a general hypersurface $Y\in |L_{d-2l}|$ satisfying all above conditions.

We claim that $\mathcal{L}_{\backslash d-2l|Y}\cdot [W_{j,\nu}]_{|Y}=0 $ for every $j\in J$ and every $1\le \nu\leq r_j$. 
Indeed, since 
$\mathcal{L}\cdot [V_{j}]=0$,  the projection formula gives
\begin{equation*}
    j_{Y*}(\mathcal{L}_{\backslash d-2l|Y}\cdot [V_j]_{X|Y})=\mathcal{L}\cdot [V_{j}]=0.
\end{equation*}
The class $\mathcal{L}_{\backslash d-2l|Y}\cdot [V_j]_{X|Y}$ is pseudo-effective. Hence its vanishing after pushforward implies that it is zero.
Therefore
\begin{equation*}
    \mathcal{L}_{\backslash d-2l|Y}\cdot [V_j]_{X|Y}= \sum_{\nu=1}^{r_j}\mathcal{L}_{\backslash d-2l|Y}\cdot [W_{j,\nu}]_Y=0.
\end{equation*}
Since every class $\mathcal{L}_{\backslash d-2l|Y}\cdot [W_{j,\nu}]_Y$ is pseudo-effective, each of them must vanish. This proves the claim.

Restricting the above linear relation to $Y$, we obtain
\begin{equation*}
    \sum_{j\in J}a_j[V_j]_{X|Y}=\sum_{j\in J}\sum_{\nu=1}^{r_j}a_j[W_{j,\nu}]_{Y}=0.
\end{equation*}
These classes are linearly independent by the claim and the induction hypothesis. Hence $a_j=0$ for every $j\in J$. The original linear relation then reduces to
$$\sum_{i\in I}a_i[V_i]=0.$$
Since every $V_i$ is annihilated by $L_{d-2l}$ which is $(2l+1)$-lef below degree $l$, they are linearly independent by Lemma \ref{linindeppartial}.
Therefore $a_i=0$ for every $i\in I$.
Hence all the coefficients in the original relation vanish, and the desired classes are linearly independent.
\end{proof}

\subsection{Mixed Lefschetz theorems}\label{mixedlefsec}
In this subsection, we prove Theorem \ref{mixedlefcharintro} and Theorem \ref{mixedkerintro}. We begin with the following kernel characterization theorem.

\begin{thrm}\label{mixedker(p,q)}
    Let $d,p,q\geq 0$ be integers satisfying $ p+q <  d\leq n$ and let $\mathcal{L}=(L_1,...,L_{d-p-q})$ be a collection of semiample line bundles. 
    Suppose that 
    $\mathcal{L}$ is $d$-lef below degree $\lceil \frac{p+q}{2}\rceil$. Then 
    \begin{equation*}
       \ker \mathcal{L}\cap H^{p,q}(X)=
       \begin{cases}
           0   & \text{if }p\ne q \\
           \Span_{\mathbb{C}}\{[V]:V\in \Irred^p(X), \mathcal{L}\cdot [V]=0\} & \text{if }p=q
       \end{cases}.
    \end{equation*}
    Moreover, if $p=q$, then the classes $\{[V]:V\in \Irred^p(X), \mathcal{L}\cdot [V]=0\}$
    are linearly independent.
\end{thrm}

The proof strategy is as follows. We take $x$ in the kernel and choose a hypersurface preserving the positivity of the collection, whose existence is guaranteed by Corollary \ref{lefprecoll}. By the induction hypothesis, the restriction of $x$ is either zero or a linear combination of algebraic cycle classes. By the homology lemma \ref{homfiber}, all these classes can be lifted to $X$. We then conclude by applying the weak Lefschetz theorem under our lefness assumptions. More precisely, we use a special case of our single-class Lefschetz theorem.

The following version of the Lefschetz decomposition and the Hodge-Riemann relations will be used to show that the restricted class is annihilated by the restricted collection.

\begin{lem}\label{lefdecomp}
    Let $d,p,q\geq 0$ be integers satisfying $p+q\leq d$. Let $X$ be an $d$-dimensional compact K\"ahler manifold, and let $L,L_1,...,L_{d-p-q}\in H^{1,1}(X,{\mathbb{R}})$ be nef classes. 
    Set $\mathcal{L}=(L_1,...,L_{d-p-q})$ and consider the Hermitian form
    \begin{equation*}
        Q: H^{p,q}(X)\times H^{p,q}(X)\rightarrow \mathbb{C} \ \ \ (\alpha,\beta ) \mapsto (\sqrt{-1})^{p-q} (-1)^{\frac{(p+q)(p+q-1)}{2}} \int_X \alpha \cdot \overline{\beta}\cdot \mathcal{L}. 
    \end{equation*}
    Suppose that the linear map
    \begin{equation*}
        \mathcal{L}\cdot L^2 : H^{p-1,q-1}(X)\rightarrow H^{d-q+1,d-p+1}(X)
    \end{equation*}
    is an isomorphism.  Then there is a $Q$-orthogonal decomposition
    \begin{equation*}
        H^{p,q}(X)=LH^{p-1,q-1}(X)\oplus P^{p,q}(X)
    \end{equation*}
    where 
    \begin{equation*}
        P^{p,q}(X)=\ker (\mathcal{L}\cdot L : H^{p,q}(X)\rightarrow H^{d-q+1,d-p+1}(X))
    \end{equation*}
    is the primitive space.
    Moreover, $Q$  is  positive semi-definite on $P^{p,q}(X)$ and
    \begin{equation*}
        \ker Q_{|P^{p,q}(X)}=\ker \mathcal{L}\cap H^{p,q}(X).
    \end{equation*}
\end{lem}
\begin{proof}
    The isomorphism
    \begin{equation*}
        \mathcal{L}\cdot L^2: H^{p-1,q-1}(X) \to H^{d-q+1,d-p+1}(X)
    \end{equation*}
    immediately gives the decomposition
    \begin{equation*}
        H^{p,q}(X) = LH^{p-1,q-1}(X)\oplus P^{p,q}(X).
    \end{equation*}
    It is moreover $Q$-orthogonal since for
    $\eta\in H^{p-1,q-1}(X)$ and $\gamma\in P^{p,q}(X)$, we have $Q(L\eta,\gamma)=0.$

    Fix a K\"ahler class $\omega$. 
    For $t>0$, let $P_t^{p,q}(X)$ and $Q_t$ be obtained by replacing $L_i$ and $L$ with $L_i+t\omega$ and $L+t\omega$ respectively. 
    By the mixed Hodge-Riemann relations \cite{DN06,cattanimixedHRR}, $Q_t$ is
    positive definite on $P_t^{p,q}(X)$. Moreover,
    \begin{equation*}
        (L_1+t\omega)\cdots(L_{d-p-q}+t\omega)(L+t\omega)^2 :H^{p-1,q-1}(X)\to H^{d-q+1,d-p+1}(X)
    \end{equation*}
    is an isomorphism for $t>0$ by the mixed hard Lefschetz theorem, and for $t=0$ by assumption. Hence $\dim P_t^{p,q}(X)$ is constant for $t\geq0$,
    so $ P_t^{p,q}(X)\to  P^{p,q}(X)$ as $t\to 0$.
    Passing to the limit shows that $Q$ is positive semi-definite on
    $P^{p,q}(X)$.

    Finally, let $\alpha\in\ker Q_{|P^{p,q}(X)}$. By the
    $Q$-orthogonal decomposition above, we have  $Q(\alpha,\beta)=0$ for every 
    $\beta\in H^{p,q}(X).$
    Poincar\'e duality therefore gives $\mathcal{L}\cdot \alpha=0.$
    Conversely, if $\mathcal{L}\cdot \alpha=0$, then
    $\alpha\in P^{p,q}(X)$ and $Q(\alpha,\beta)=0$ for every  $\beta\in H^{p,q}(X).$
    Thus
    \begin{equation*}
        \ker Q_{|P^{p,q}(X)} = \ker\mathcal{L}\cap H^{p,q}(X).
    \end{equation*}
    This completes the proof.
\end{proof}

\begin{proof}[proof of Theorem \ref{mixedker(p,q)}]
    We may assume that each $L_i$ is basepoint-free by replacing each $L_i$ by a sufficiently divisible positive multiple.

    We first show that it suffices to prove the theorem when $n=d$.    
    Let $A$ be an ample line bundle and consider the collection
    $$\overline{\mathcal{L}}=(L_1,...,L_{d-p-q},A,...,A)$$ 
    where $A$ is repeated $n-d$ times. 
    By Lemma \ref{addgeom}, this collection is $n$-lef below degree $\lceil\frac{p+q}{2}\rceil$. 
    Assume that the theorem has been proved when the dimension of the ambient variety equals the lef degree.
    If $p\neq q$, the injectivity of 
    \begin{equation*} 
        \mathcal{L}\cdot A^{n-d}: H^{p,q}(X)\to H^{n-q,n-p}(X) 
    \end{equation*} 
    immediately implies the injectivity of 
    \begin{equation*} 
        \mathcal{L}: H^{p,q}(X)\to H^{d-q,d-p}(X). 
    \end{equation*}
    If $p=q$, then
    \begin{equation*}
        \ker (\mathcal{L} \cdot A^{n-d}:H^{p,p}(X)\to H^{n-p,n-p}(X))=\Span_{\mathbb{C}}\{[V]:V\in \Irred^p(X),\mathcal{L}\cdot A^{n-d}\cdot [V]=0\}.
    \end{equation*}
     Hence
    \begin{align*}
        &\ker (\mathcal{L}:H^{p,p}(X)\to H^{d-p,d-p}(X))\\
        \subset &\ker (\mathcal{L} \cdot A^{n-d}:H^{p,p}(X)\to H^{n-p,n-p}(X))\\
        =&\Span_{\mathbb{C}}\{[V]:V\in \Irred^p(X),\mathcal{L} \cdot A^{n-d}\cdot [V]=0\}\\
        =&\Span_{\mathbb{C}}\{[V]:V\in \Irred^p(X),\mathcal{L}\cdot [V]=0\}\\
        \subset &\ker (\mathcal{L} :H^{p,p}(X)\to H^{d-p,d-p}(X)).
    \end{align*}
    Thus all the inclusions are equalities.

    We now assume $n=d$ and proceed by induction on the dimension $d$. 
    If $p=q=0$, the assertion is immediate from $H^{0,0}(X)=\mathbb{C}[X]$. Moreover, if $d-p-q=1$ the assertion is a consequence of Corollary \ref{pureker(p,q)}. 
    We may therefore assume $(p,q)\neq (0,0)$ and $d-p-q\geq 2.$ 
    The case $d=1$ is already covered. 
    Assume that the theorem holds in dimension $d-1$ for all possible bidegrees.

    We split the discussion into two cases.

    \medskip\noindent\textbf{Case 1: $p\neq q$.}
 
    By Lemma \ref{addgeom}, the collection $(\mathcal{L},L_{d-p-q},L_{d-p-q})$ is $d$-lef below degree $\lceil\frac{p+q}{2}\rceil-1$. Here we cannot conclude that this collection is $d$-lef below degree $\lceil\frac{p+q}{2}\rceil$ since the size of this collection is $d-p-q+2$ which may be strictly greater than $d-2\lceil\frac{p+q}{2}\rceil+2$.
    By Bertini's theorem and Corollary \ref{lefprecoll}, a general hypersurface $Y\in |L_{d-p-q}|$ is smooth and irreducible, and satisfies the following conditions
    \begin{itemize}
        \item $\mathcal{L}_{\backslash d-p-q|Y}$ is $(d-1)$-lef below degree $\lceil \frac{p+q}{2}\rceil$, and
        \item $(\mathcal{L}_{\backslash d-p-q|Y},L_{d-p-q|Y},L_{d-p-q|Y})$ is $(d-1)$-lef below degree $\lceil \frac{p+q}{2}\rceil-1$.
    \end{itemize}
    Then by the induction hypothesis, the linear maps
    \begin{equation*}
        \mathcal{L}_{\backslash d-p-q|Y}:H^{p,q}(Y)\to H^{d-q-1,d-p-1}(Y)
    \end{equation*}
    and 
    \begin{equation*} 
        \mathcal{L}_{\backslash d-p-q|Y}\cdot L_{d-p-q|Y}^2: H^{p-1,q-1}(Y)\to H^{d-q,d-p}(Y) 
    \end{equation*} are both injective. They are further isomorphisms for dimensional reasons.

    Hence Lemma \ref{lefdecomp} implies that the Hermitian form
    \begin{equation*}
        Q(\alpha,\beta)=(\sqrt{-1})^{p-q}(-1)^{\frac{(p+q)(p+q-1)}{2}}\int_Y \alpha \cdot \overline{\beta}\cdot \mathcal{L}_{\backslash d-p-q|Y}.
    \end{equation*}
    defined on $H^{p,q}(Y)$ is positive definite on the primitive space
    \begin{equation*}
        P^{p,q}(Y)=\ker(\mathcal{L}_{\backslash d-p-q|Y}\cdot L_{d-p-q|Y}:H^{p,q}(Y)\to H^{d-q,d-p}(Y)).
    \end{equation*}
    Let $x\in \ker (\mathcal{L}:H^{p,q}(X)\to H^{d-q,d-p}(X)).$ 
    Restricting to $Y$, we obtain
    \begin{equation*}
        \mathcal{L}_{\backslash d-p-q|Y}\cdot L_{d-p-q|Y}\cdot x_{|Y}=0,
    \end{equation*}
    and hence $x_{|Y}\in P^{p,q}(Y)$. 
    On the other hand, 
    \begin{align*}
        Q(x_{|Y},x_{|Y})&=(\sqrt{-1})^{p-q}(-1)^{\frac{(p+q)(p+q-1)}{2}}\int_Y x_{|Y} \cdot \overline{x_{|Y}} \cdot \mathcal{L}_{\backslash d-p-q|Y}\\
        &=(\sqrt{-1})^{p-q}(-1)^{\frac{(p+q)(p+q-1)}{2}}\int_X x \cdot \overline{x}\cdot 
        \mathcal{L}_{\backslash d-p-q} \cdot L_{d-p-q}\\
        &=0.
    \end{align*}
    Since $Q$ is positive definite on $P^{p,q}(Y)$, we conclude that $x_{|Y}=0$. Pushing forward to $X$ gives $L_{d-p-q}\cdot x=0$. 
    By Lemma \ref{equivlefcolfree}, the line bundle $L_{d-p-q}$ is $(p+q+1)$-lef below degree $\left\lceil\frac{p+q}{2}\right\rceil$.
    Hence Corollary \ref{pureker(p,q)} implies that 
    \begin{equation*}
        L_{d-p-q}:H^{p,q}(X)\to H^{p+1,q+1}(X)
    \end{equation*}
    is injective. Therefore $x=0$, proving the case $p\neq q$.

    \medskip\noindent\textbf{Case 2: $p=q$.} 
    
    We may assume $p\geq 1$.
    By Lemma \ref{equivlefcolfree}, $L_I$ is $(|I|+2p)$-lef below degree $p$ for every nonempty subset $I\subset [d-2p]$. Let $f_I:X\to Y_I$ be the Kodaira map associated with $L_I$.
    Set 
    \begin{equation*}
        Z_I:=\{x\in X:\dim (f_I^{-1}f_I(x),x)\geq d-p-|I|\}.
    \end{equation*}
    By Proposition \ref{equiv}, $\dim Z_I\leq d-p.$
    By Lemma \ref{addgeom}, the collection
    $(\mathcal{L},L_{d-2p},L_{d-2p})$ is $d$-lef below degree $p$.    
    We choose a general hypersurface $Y\in |L_{d-2p}|$ that is smooth and irreducible, and satisfies the following conditions.
    \begin{itemize}
        \item $\mathcal{L}_{\backslash d-2p|Y}$ is $(d-1)$-lef below degree $p$.
        \item $(\mathcal{L}_{\backslash d-2p|Y},L_{d-2p|Y},L_{d-2p|Y})$ is $(d-1)$-lef below degree $p$.
        \item For every nonempty subset $I\subset [d-2p]$ and every irreducible component $W$ of $Z_I$, the intersection $Y\cap W$ is proper and pure-dimensional.
        Moreover, for every irreducible component $V$ of $W\cap Y$, there exists a constant $c>0$ such that 
        \begin{equation*}
            [V]_{W,hom}=cL_{d-2p|W}\cap [W]_{W,hom} \in H_{*}(W).
        \end{equation*}
    \end{itemize}
    The existence of such a hypersurface is guaranteed by Bertini's theorem, Lemma \ref{lefprecoll} and Lemma \ref{homfiber}.
    
    Applying the induction hypothesis to the bidegree $(p-1,p-1)$, we have
    \begin{equation*}
        \ker(\mathcal{L}_{\backslash d-2p|Y}\cdot L_{d-2p|Y}^2) \cap H^{p-1,p-1}(Y)=\Span_{\mathbb{C}}\{[V]_Y:V\in \Irred^{p-1}(Y),\mathcal{L}_{\backslash d-2p|Y}\cdot L_{d-2p|Y}^2\cdot [V]_Y=0\}
    \end{equation*}
     Since the above collection is $(d-1)$-lef below degree $p$, the right-hand side is zero.
     Hence the map is injective, and therefore an isomorphism for dimensional reasons.
    Moreover, we have
    \begin{equation}\label{d2odqe}
        \ker (\mathcal{L}_{\backslash d-2p|Y}:H^{p,p}(Y)\to H^{d-p-1,d-p-1}(Y))=\Span_{\mathbb{C}}\{[V]_Y:V\in \Irred^p(Y), \mathcal{L}_{\backslash d-2p|Y}\cdot [V]_Y=0\}.
    \end{equation}
    Lemma \ref{lefdecomp} therefore implies that the Hermitian form
    \begin{equation*}
        Q(\alpha,\beta)=(-1)^p\int_Y \alpha\cdot \overline{\beta}\cdot \mathcal{L}_{\backslash d-2p|Y}
    \end{equation*}
    on $H^{p,p}(Y)$ is positive semi-definite on the primitive space
    \begin{equation*}
        P^{p,p}(Y)=\ker (\mathcal{L}_{\backslash d-2p|Y}\cdot L_{d-2p|Y}:H^{p,p}(Y)\to H^{d-p,d-p}(Y))
    \end{equation*}
    with kernel 
    $$\ker Q_{|P^{p,p}(Y)}=\ker (\mathcal{L}_{\backslash d-2p|Y}:H^{p,p}(Y)\to H^{d-p-1,d-p-1}(Y)).$$
    
    Let $x \in \ker(\mathcal{L}:H^{p,p}(X)\to H^{d-p,d-p}(X)).$ 
    Restricting to $Y$, we obtain
    \begin{equation*}
        \mathcal{L}_{\backslash d-2p|Y}\cdot L_{d-2p|Y}\cdot x_{|Y} =0.
    \end{equation*}
    Hence $x_{|Y}\in P^{p,p}(Y)$. Furthermore,
    \begin{align*}
        Q(x_{|Y},x_{|Y})&=(-1)^p\int_Y x_{|Y}\cdot \overline{x_{|Y}}\cdot \mathcal{L}_{\backslash d-2p|Y}  \\
        &= (-1)^p\int_Xx\cdot \overline{x}\cdot \mathcal{L} \\
        &=0.
    \end{align*}
    Therefore 
    $$x_{|Y}\in \ker Q_{|P^{p,p}(Y)}=\ker (\mathcal{L}_{\backslash d-2p|Y}:H^{p,p}(Y)\to H^{d-p-1,d-p-1}(Y)).$$
    By \eqref{d2odqe}, there exist $V_1,...,V_N \in \Irred^p(Y)$ satisfying 
    \begin{equation}\label{b4}
        \mathcal{L}_{\backslash d-2p|Y}\cdot [V_i]_Y=0
    \end{equation}
    and constants $a_1,...,a_N\in \mathbb{C}$ such that 
    \begin{equation}\label{b1}
        x_{|Y}=\sum_{i=1}^Na_i [V_i]_{Y}.
    \end{equation}
    
    We claim that for every $i$, there exists a real constant $c_i>0$ and $W_i \in\Irred^p(X)$ such that
    \begin{equation}\label{b2}
        \mathcal{L}\cdot [W_i]_X=0
    \end{equation}
    and 
    \begin{equation}\label{b3}
        L_{d-2p}\cdot [W_i]_X=c_i[V_i]_X.
    \end{equation}

    To prove the claim, fix $i$ and write $V=V_i$. 
    By (\ref{b4}) and Proposition \ref{hallrado}, 
    there exists a nonempty subset $I\subset [d-2p-1]$ such that $(L_{I}^{|I|})_{|Y}\cdot [V]_Y=0$, which implies $L_I^{|I|}\cdot [V]_X=0$. Therefore
    \begin{equation*}
        V\subset Z_I=\{x\in X:\dim (f_I^{-1}f_I(x),x)\geq d-p-|I|\}.
    \end{equation*}
    Since $\dim V=d-p-1$, $\dim Z_I\leq d-p$ and $Y$ intersects every irreducible component of $Z_I$ properly, there exists a $(d-p)$-dimensional irreducible component $W\subset Z_I$ such that $V$ is an irreducible component of $W\cap Y$. In particular, $W\in\Irred^p(X).$
    Moreover, by our choice of $Y$, there exists $c>0$ such that
    \begin{equation*}
         [V]_{W,hom}=c L_{d-2p|W}\cap [W]_{W,hom}.
    \end{equation*}
    Pushing forward to $X$ and using Poincar\'e duality, we obtain 
    \begin{equation*} 
    L_{d-2p}\cdot[W]_X = \frac{1}{c}[V]_X. 
    \end{equation*} 
    Thus \eqref{b3} holds.
    Moreover,  pushing forward \eqref{b4} to $X$ implies
    \begin{equation*} 
    \mathcal{L}_{\backslash d-2p}\cdot[V]_X=0. 
    \end{equation*} 
    Therefore 
    \begin{equation*} 
    \mathcal{L}\cdot[W]_X = \mathcal{L}_{\backslash d-2p} \cdot L_{d-2p}\cdot[W]_X = \frac{1}{c}\, \mathcal{L}_{\backslash(d-2p)}\cdot[V]_X = 0. 
    \end{equation*} 
    This proves \eqref{b2} and hence the claim.

    Pushing forward \eqref{b1} to $X$ gives 
    \begin{equation*} 
        L_{d-2p}\cdot x = \sum_{i=1}^N a_i[V_i]_X. 
    \end{equation*}
   Using (\ref{b3}), we obtain
    \begin{equation*}
       L_{d-2p}\cdot(x-\sum_{i=1}^N \frac{a_i}{c_i}[W_i]_X)=0.
    \end{equation*}
    By Lemma \ref{equivlefcolfree}, 
    $L_{d-2p}$ is $(2p+1)$-lef below degree $p$. 
    Hence Corollary \ref{pureker(p,q)} shows that
    \begin{equation*} 
        \ker ( L_{d-2p}:H^{p,p}(X)\to H^{p+1,p+1}(X) ) =\Span_{\mathbb{C}} \{ [V]: V\in\Irred^p(X),\ L_{d-2p}\cdot[V]=0 \}.
    \end{equation*}
    Together with \eqref{b2}, this shows 
    \begin{equation*}
        x\in \Span_{\mathbb{C}} \{ [V]: V\in\Irred^p(X),\mathcal{L} \cdot[V]=0 \}.
    \end{equation*}
    Since $x\in \ker(\mathcal{L}:H^{p,p}(X)\to H^{d-p,d-p}(X))$ is arbitrary, we obtain 
    \begin{equation*}
        \ker(\mathcal{L}:H^{p,p}(X)\to H^{d-p,d-p}(X))\subset \Span_{\mathbb{C}} \{ [V]: V\in\Irred^p(X),\mathcal{L} \cdot[V]=0 \}.
    \end{equation*}
    The reverse inclusion is immediate, so equality holds. 
    The linear independence of these classes follows from Lemma \ref{mixedlinindep}.
\end{proof}

We now prove Theorem \ref{mixedlefcharintro} and \ref{mixedkerintro}.

\begin{cor}\label{mixedlefchar}(=Theorem \ref{mixedlefcharintro})
    Let $d,l$ be integers satisfying $0\leq 2l-1\leq d\leq n$, 
    and let $\mathcal{L}=(L_1,...,L_m)$ be a collection of semiample line bundles with $1\leq m\leq d-2l+2$. Then the following statements are equivalent.
    \begin{enumerate}
        \item $\mathcal{L}$ is $d$-lef below degree $l$. Equivalently,
        \begin{equation*}
            \mathcal{L}^{\alpha} \cdot [V]\neq 0
        \end{equation*}
        for every subvariety $V\subset X$ with $\codimension_XV \leq l-1$ and every $\alpha\in \mathbb{Z}_{>0}^{m}$ satisfying $|\alpha|=d-2\codimension_XV.$
        
        \item For every nonempty subset $I\subset [m]$, $L_I$ is $(|I|+d-m)$-lef below degree $l$. Equivalently,
        \begin{equation*}
            \nd_V(L_I)\geq |I|+d-m-2\codimension_XV
        \end{equation*}
        for every subvariety $V\subset X$ with $\codimension_XV \leq l-1$ and every nonempty subset $I\subset [m]$.

        \item For every integer $0\leq k\leq l-1$ and every $\alpha\in \mathbb{Z}_{>0}^{m}$ satisfying $|\alpha|=d-2k$, the linear map 
        \begin{equation*}
            \mathcal{L}^{\alpha}:H^{k,k}(X)\to H^{d-k,d-k}(X)
        \end{equation*}
        is injective.
        
        \item For all integers $p,q\ge 0$ with $ p+q\leq 2l-2$ and every $\alpha\in \mathbb{Z}_{>0}^{m}$ satisfying $|\alpha|=d-p-q$, the linear map 
        \begin{equation*}
            \mathcal{L}^{\alpha}:H^{p,q}(X)\to H^{d-q,d-p}(X)
        \end{equation*}
        is injective.

    \end{enumerate}
\end{cor}
\begin{proof}
    The equivalence $(1)\iff (2)$ follows from Lemma \ref{equivlefcolfree}. The implications $(4)\implies (3) \implies (1)$ are immediate. 
    It remains to prove $(1)\implies (4)$.
    If $(p,q)=(0,0)$, injectivity follows directly from (1).
    We may therefore assume that $p+q\geq1$.
    By Lemma \ref{addgeom}, the collection 
    $(L_1[\alpha_1],\ldots,L_m[\alpha_m])$ is $d$-lef below degree $l$, and hence also below degree $\lceil\frac{p+q}{2}\rceil.$
    If $p\neq q$, Theorem \ref{mixedker(p,q)} immediately gives
    \begin{equation*}
        \ker (\mathcal{L}^{\alpha}:H^{p,q}(X)\to H^{d-q,d-p}(X))=0.
    \end{equation*}
    If $p=q$, Theorem \ref{mixedker(p,q)} gives
    \begin{equation*}
        \ker(\mathcal{L}^{\alpha}:H^{p,p}(X)\to H^{d-p,d-p}(X))=\Span_{\mathbb{C}}\{[V]:V\in \Irred^p(X), \mathcal{L}^{\alpha}\cdot[V]=0\}.
    \end{equation*}
    The right-hand side is zero by $(1)$. 
    Thus $\mathcal{L}^{\alpha}$ is injective in both cases, proving $(4)$.
\end{proof}

\begin{cor}\label{mixedker}(=Theorem \ref{mixedkerintro})
    Let $d,l$ be integers satisfying $l\geq 1$ and $ 2l+1\leq d\leq n$, 
    and let $\mathcal{L}=(L_1,...,L_m)$ be a collection of semiample line bundles. Suppose that $\mathcal{L}$ is $d$-lef below degree $l$. 
    Then for all integers $p,q\geq 0$ with $p+q\leq 2l$ and every $\alpha\in \mathbb{Z}^m_{>0}$ with $|\alpha|=d-p-q$, the following statements hold.
    \begin{enumerate}
        \item If $(p,q)\neq (l,l)$, then the map
        \begin{equation*}
            \mathcal{L}^{\alpha}:H^{p,q}(X)\to H^{d-q,d-p}(X)
        \end{equation*}
        is injective.
        \item If $(p,q)= (l,l)$, then
        \begin{equation*}
            \ker(\mathcal{L}^{\alpha}:H^{l,l}(X)\to H^{d-l,d-l}(X))=\Span_{\mathbb{C}}\{[V]:V\in \Irred^l(X), \mathcal{L}^{\alpha }\cdot [V]=0\}.
        \end{equation*}
    \end{enumerate}
\end{cor}
\begin{proof}
    By Lemma \ref{addgeom}, the collection $(L_1[\alpha_1],...,L_{m}[\alpha_m])$ is $d$-lef below degree $\lceil\frac{p+q}{2}\rceil$.
    The statements then follows from Theorem \ref{mixedker(p,q)} immediately.
\end{proof}

\section{Decoupling for Lefschetz modules}\label{decoupLefMod}

In this section, we introduce the notion of Lefschetz modules, study the kernels of Lefschetz operators in this setting and then provide a new proof of the main results of Section \ref{geometricdecoupling}.
We focus on complex Lefschetz modules, although all the results remain valid for real Lefschetz modules as well.

\subsection{Lefschetz modules}\label{Lefschetzmodules}
We begin by recalling the notion of a Lefschetz module from \cite{aminihuhlarson25lefschetzmodule}.
Let $A=\bigoplus_{i\geq 0}A^i$ be a finite-dimensional graded commutative $\mathbb{R}$-algebra and let $\mathcal{K}_A\subset A^1$ be an open convex cone. We call $\mathcal{K}_A$ the K\"ahler cone and its closure $\overline{\mathcal{K}_A}$ the nef cone. 
Set $A_{\mathbb{C}}=A\otimes\mathbb{C}$.
Let $M=\bigoplus_{i\geq 0} M^i$ be a finite-dimensional graded $A_{\mathbb{C}}$-module endowed with an $A_{\mathbb{C}}$-invariant Hermitian form
$$Q_M: M\times M\to \mathbb{C}.$$
Here $A_{\mathbb{C}}$-invariance means that $Q_M(ax,y)=Q_M(x,\overline a\,y)$ for all $x,y\in M$ and $a\in A_{\mathbb{C}}$.

Every element $L\in A$ acts on $M$ by multiplication. We will not distinguish between $L$ and its associated endomorphism. For example, we will denote by $\ker L$ the kernel of the map 
$$L:M\to M \ \ \ x\mapsto L  x$$
if there is no confusion on the module $M$ in the context.

\begin{defn}
    The pair $(M,Q_M)$ is called a complex Lefschetz module of degree $n$ over $(A,\mathcal{K}_A)$ if the following conditions hold.
    \begin{enumerate}
        \item For every $i\in \mathbb{Z}$, the Hermitian pairing 
        $$Q_M: M^i\times M^{n-i}\rightarrow \mathbb{C}$$ is nondegenerate, and $Q_M(M^i,M^j)=0$ whenever $i+j\neq n.$
        \item For every $L\in \mathcal{K}_A$ and every integer $i\leq \lfloor\frac{n}{2}\rfloor$, the linear map $$L^{n-2i}:M^i\to M^{n-i}$$ is an isomorphism.
        \item For every $L\in \mathcal{K}_A$ and every integer $i\leq \lfloor\frac{n}{2}\rfloor$, the Hermitian form $$M^i\times M^i\to \mathbb{C} \ \ \ (x,y)\mapsto (-1)^iQ_M(x,L^{n-2i}\cdot y)$$
        is positive definite on the primitive subspace 
        $$\ker(L^{n-2i+1}:M^i\to M^{n-i+1}).$$
    \end{enumerate}
\end{defn}

A basic example is provided by the cohomology of a compact K\"ahler manifold.

\begin{example}\label{geomlefmod}
    Let $X$ be a compact K\"ahler manifold of dimension $n$. Let
    $A= \bigoplus_{i=0}^n H^{i,i}(X,\mathbb R)$, and let
    $\mathcal K_A\subset A^1$ be the K\"ahler cone. 
    For every $k\geq 0$, set 
    $$M_k=\bigoplus_{i\geq 0}H^{k+i,i}(X) \text{ and } M_{-k}=\bigoplus_{i\geq 0}H^{i,k+i}(X)$$
    with natural grading $M_k^i=H^{k+i,i}(X)$ and $M_{-k}^i=H^{i,k+i}(X)$.
    We endow $M_k$ with the Hermitian form
    \begin{equation*}
        Q_{M_k}(x,y)  = (\sqrt{-1})^k(-1)^{\frac{|k|(|k|-1)}{2}}
        \int_X x\cdot \overline{y}.
    \end{equation*}
    Then $(M_k,Q_{M_k})$ is a complex Lefschetz module of degree
    $n-|k|$ over $(A,\mathcal K_A)$.
\end{example}

\begin{rmk}\label{mixHRKah}
    As a consequence of the descent lemma (see, e.g., \cite{CKS87} or Lemma \ref{descentlem}),  Cattani \cite{cattanimixedHRR} showed that the mixed hard Lefschetz theorem and the mixed Hodge-Riemann relations hold for any Lefschetz module. More precisely, we have the following results.
    \begin{enumerate}
        \item For every integer $i\leq \lfloor\frac{n}{2}\rfloor$ and every $L_1,...,L_{n-2i}\in \mathcal{K}_A$, the linear map $$L_1\cdots L_{n-2i}:M^i\to M^{n-i}$$ is an isomorphism.
        \item For every integer $i\leq \lfloor\frac{n}{2}\rfloor$ and every $L_1,...,L_{n-2i},L_{n-2i+1}\in \mathcal{K}_A$, the Hermitian pairing $$M^i\times M^i\to \mathbb{C} \ \ \ (x,y)\mapsto (-1)^iQ_M(x,L_1\cdots L_{n-2i}\cdot y)$$
        is positive definite on the primitive space $\ker(L_1\cdots L_{n-2i+1})\cap M^i.$
    \end{enumerate}
\end{rmk}

Although \cite{aminihuhlarson25lefschetzmodule} is formulated primarily
for real Lefschetz modules, the authors also establish the corresponding
results for complex Lefschetz modules. We recall below the facts that
will be used in this paper.

Throughout this subsection, let $(M,Q_M)$ be a complex Lefschetz module of degree $n$ over $(A,\mathcal K_A)$. 
For a graded module $N$ and an integer $k$,
we denote by $N[k]$ the shifted graded module defined by
$N[k]^i=N^{i+k}.$
Let $B\subset A$ be a graded subalgebra generated by $B^1$, and set
$\mathcal K_B :=\operatorname{int}_{B_1} ( B^1\cap\overline{\mathcal K_A} ).$
Assume that $\mathcal K_B\neq\emptyset$. Viewing $M$ as a graded
$B_{\mathbb C}$-module, choose a decomposition
\begin{equation*}
    M \cong \bigoplus_{\alpha}\bigoplus_k N_\alpha[-k]^{\oplus m(\alpha,k)}
\end{equation*}
into indecomposable graded $B_{\mathbb C}$-modules, where the
$N_\alpha$ are chosen such that $\min\{i:N_{\alpha}^i\neq 0\}=0.$
This decomposition is not canonical. However, by the Krull-Schmidt
theorem, the isomorphism classes of the indecomposable modules
$N_\alpha$ and the multiplicities $m(\alpha,k)$ are independent of the
chosen decomposition.

By \cite[Theorem 3.10]{aminihuhlarson25lefschetzmodule}, for each
$\alpha$ there exists, up to a nonzero complex scalar, a unique
nonzero $B_{\mathbb C}$-invariant form
\begin{equation*}
    Q_\alpha:N_\alpha\times N_\alpha\longrightarrow\mathbb C
\end{equation*}
satisfying the degree orthogonality condition. Moreover, there is a
unique scalar $\epsilon_\alpha\in\mathbb C$ with $|\epsilon_\alpha|=1$
such that $(N_\alpha,\epsilon_\alpha Q_\alpha)$
is a complex Lefschetz module over $(B,\mathcal K_B)$.


Given a nef class $L\in \overline{\mathcal{K}_A}$, its action on $M$ is nilpotent and determines a unique increasing filtration
$W_\bullet=W_\bullet(L,M),$
called the weight filtration. It is characterized by
$LW_j\subset W_{j-2}$
and by the requirement that the induced map
\begin{equation*}
    L^k:\Gr_k^W\to \Gr_{-k}^W
\end{equation*}
is an isomorphism for every $k\geq 0$ where $\Gr_k^W=W_k/W_{k-1}$
denotes the $k$-th associated graded piece.
Equivalently, the weight filtration admits the following explicit description. For all $k\geq 0$,
\begin{equation*}
    W_k= \sum_{c\geq 0} L^c\ker L^{k+2c+1}  \text{ and } W_{-k}=L^kW_k.
\end{equation*}
For example, this follows from \cite[\S4.5, (23)]{deCMdecomp}.

In the remaining part of this subsection,
we explain the relationship between complex Lefschetz modules and polarized
Hodge-Lefschetz modules in the sense of
\cite[Definition 2.3]{cattanimixedHRR}.
A related argument appears  in \cite[Section 3.4.1]{aminihuhlarson25lefschetzmodule}.

Let $(M,Q_M)$ be a complex Lefschetz module of degree $n$ over
$(A,\mathcal K_A)$. 
Let $V=M_{\mathbb R}$ be the underlying real vector space of $M$ with grading $V_{n-2i}=M^i_{\mathbb R}$.
For every $i$, there is a canonical decomposition $ (V_{n-2i})_{\mathbb C} = M^i\oplus\overline{M^i}.$
We define a Hodge decomposition $V_{\mathbb{C}}=\bigoplus_{0\leq p,q\leq n+1}V^{p,q}$ by setting $V^{n+1-i,n-i}=M^i$ and $V^{n-i,n+1-i}=\overline{M^i}.$ 
Thus every $L\in A^1$ has Hodge type $(-1,-1)$.
For $x\in M^i$ and $y\in M^{n-i}$, regarded as elements in $V$, define
\begin{equation*}
    S(x,y)=2(-1)^i\operatorname{Im}Q_M(x,y).
\end{equation*}
Set $S(M^i,M^j)=0$ if $i+j\neq n$. Extending by linearity then gives a nondegenerate bilinear form $S:V\times V\to \mathbb{R}$.
The Hermitian condition of $Q_M$ gives
\begin{equation*}
     S(x,y)=2(-1)^i\operatorname{Im}Q_M(x,y)= 2(-1)^i \operatorname{Im}\overline{Q_M(y,x)}=(-1)^{n+1}S(y,x).
\end{equation*}
Hence $S$ has parity $(-1)^{n+1}$.
The $A$-invariance of $Q_M$ implies that 
$$S(Lx,y)+S(x,Ly)=0$$ 
for all $L\in A^1$ and $x,y\in V$.
Therefore, if $\mathfrak a$ denotes the image of $A^1$ in
$\operatorname{End}_{\mathbb R}(V)$, then $\mathfrak a\subset\mathfrak{o}_{-2}(V,S).$
For every $\eta \in \mathcal{K}_A$, the hard Lefschetz property and Hodge-Riemann relations for $(M,Q_M)$ directly imply the corresponding properties for $(V_*,S,\mathfrak a,\eta)$. We therefore obtain the following.


\begin{lem}\label{PHLmodules}
For every $\eta\in \mathcal{K}_A$, $(V_*,S,\mathfrak a,\eta)$ defined above is a polarized Hodge-Lefschetz module of weight $n+1$. Moreover, 
the image of $\mathcal K_A$ in $\mathfrak{a}$ is contained in the polarizing cone 
$\mathcal{C}=\mathcal{C}(V_*,S,\mathfrak a,\eta)$ of this polarized Hodge-Lefschetz module.
\end{lem}

Conversely, from polarized Hodge-Lefschetz modules we obtain complex Lefschetz modules. 
Let $(V_*,S,\mathfrak a,\eta)$ be a polarized Hodge-Lefschetz module of weight $n$ with Hodge decomposition $V_{\mathbb{C}}=\bigoplus_{0\leq p,q\leq n}V^{p,q}$. For $0\leq k\leq n$, set 
$$M_k= \bigoplus_{0\leq q\leq n-k} V^{k+q,q}\text{ and } M_{-k}=\bigoplus_{0\leq p\leq n-k} V^{p,k+p}$$
with grading $M_k^i=V^{n-i,n-k-i}$ and $M_{-k}^i=V^{n-k-i,n-i}$. 
Thus $M_{-k}=\overline{M_k}$.
For $x\in M_k^i$ and $y\in M_k^{n-|k|-i}$, define 
$$Q_{M_{k}}(x,y)=(\sqrt{-1})^k(-1)^iS(x,\overline{y})$$ 
and $Q_{M_k}(M_k^i,M_k^j)=0$ whenever $i+j\neq n-|k|$.
Let $A\subset \End_{\mathbb{R}}(V)$ be the unital subalgebra generated by $\mathfrak{a}$.
Since $\mathfrak{a}$ is abelian, $A$ is commutative.
Moreover, it is naturally graded by $A^i=A\cap \End_{-2i}(V)$.
Since every element of $\mathfrak a$ has Hodge type $(-1,-1)$, the action of $A$ preserves each $M_k$ with $A^jM_{ k}^i\subset M_{k}^{i+j}. $
Let $\mathcal{C}=\mathcal{C}(V_*,S,\mathfrak a,\eta)$ be the polarizing cone. 
The following observation is immediate from the definition.

\begin{lem}\label{PHLmodulesconverse}
For every $k\in \mathbb{Z}$ with $|k|\leq n$, $(M_k,Q_{M_k})$ defined above is  a complex Lefschetz module of degree $n-|k|$ over $(A,\mathcal{C})$.
\end{lem}

\subsection{Lefness on Lefschetz modules}
In this subsection, we introduce the notion of lefness for a nef class
and more generally for a collection of nef classes acting on a Lefschetz
module. We also establish some basic properties of this notion.
Throughout this subsection, we fix a complex Lefschetz module
$(M,Q_M)$ of degree $n$ over $(A,\mathcal K_A)$.

In this abstract setting, there is no intrinsic notion of algebraic cycle classes.
Nevertheless, Corollary \ref{purelefchar} characterizes the lefness of a semiample line bundle by the injectivity of the induced Lefschetz operators. This motivates the following definiton.

\begin{defn}
    Let $d,l$ be integers satisfying $ 2l-1 \leq d\leq n$. A nef class $L\in\overline{\mathcal K_A}$ is said to be
    $d$-lef below degree $l$ on $M$ if the map
    \begin{equation*}
        L^{d-2k}:M^k\longrightarrow M^{d-k}
    \end{equation*}
    is injective for every $k\leq l-1$.
\end{defn}

The lefness condition admits a convenient characterization in terms of
the weight filtration.

\begin{lem}\label{alglefequiv}
Let $l,d$ be integers satisfying $2l-1 \leq d\leq n$ and let $L\in \overline{\mathcal{K}_A}$ be a nef class with weight filtration $W_\bullet$.
Then the following statements are equivalent:
\begin{enumerate}
    \item $L$ is $d$-lef below degree $l$ on $M$.
    \item $W_{d-2k-1}\cap M^k=0$ for every $k\leq  l-1$.
\end{enumerate}
Moreover, if $2l+1\leq d$ and $L$ is $d$-lef below degree $l$ on $M$, then $$\ker L^{d-2l}\cap M^l=W_{d-2l-1}\cap M^l.$$
\end{lem}
\begin{proof}
    For every
    $k\leq l-1$, we have $d-2k-1\geq 0$.
    Hence the explicit description of the weight filtration gives
    \begin{align*}
        W_{d-2k-1}\cap M^k
        &=\sum_{c\geq 0} L^c(\ker L^{d-2k+2c}\cap M^{k-c}) \\
        &= \ker L^{d-2k}\cap M^k +L(W_{d-2k+1}\cap M^{k-1}).
    \end{align*}
    Since $d-2k+1=d-2(k-1)-1,$
    the equivalence follows by induction on $k$.
        If $2l+1\leq d$, then the same identity with $k=l$ gives
    \begin{equation*}
        W_{d-2l-1}\cap M^l = \ker L^{d-2l}\cap M^l + L(W_{d-2l+1}\cap M^{l-1} ).
    \end{equation*}
    Since $L$ is $d$-lef below degree $l$ on $M$, the first part gives
    \begin{equation*}
        W_{d-2l+1}\cap M^{l-1}=0,
    \end{equation*}
    and hence
    \begin{equation*}
        W_{d-2l-1}\cap M^l = \ker L^{d-2l}\cap M^l.
    \end{equation*}
\end{proof}

Let $L,M \in\overline{\mathcal{K}_A}$ be nef classes. We write $L\gtrsim M$ if there exists $\lambda>0$ such that $\lambda L-M\in \overline{\mathcal{K}_A}$ and write $L\sim M$ if both $L\gtrsim M$ and $M\gtrsim L$ hold. Equivalently,
$L_1\sim L_2$ if and only if $L_1$ and $L_2$ lie in the relative
interior of the same face of the nef cone $\overline{\mathcal K_A}$.
Next, we show that the lefness of a nef class $L$ depends only on its $\sim$-equivalent class, or equivalently, on the face whose relative interior contains it. This follows from Lemma \ref{alglefequiv} together with the invariance of the weight filtration due to Cattani and Kaplan \cite{CK82invention}. 

\begin{lem}\label{invwtfil}
    Let $L_1,L_2\in \overline{\mathcal{K}_A}$ be nef classes. If $L_1\sim L_2$, then $$W(L_1,M)=W(L_2,M).$$
\end{lem}

\begin{cor}\label{kerinclusion}
    Let $d,l$ be integers satisfying $2l-1\leq d\leq n$, and
    let $L_1,L_2\in \overline{\mathcal{K}_A}$ be nef classes. 
    Suppose that $L_1\gtrsim L_2$ and that $L_2$ is $d$-lef below degree $l$ on $M$. Then $L_1$ is also $d$-lef below degree $l$ on $M$.
    If moreover $2l+1\leq d$, then
    \begin{equation*}
        \ker L_1^{d-2l}\cap M^l \subset \ker L_2^{d-2l} \cap M^l.
    \end{equation*}
    Consequently, if $L_1\sim L_2$ then $L_1$ is $d$-lef below degree $l$ on $M$ if and only if $L_2$ is.
    If in addition $2l+1\leq d$ and these equivalent conditions hold, then
    \begin{equation*}
        \ker L_1^{d-2l}\cap M^l = \ker L_2^{d-2l} \cap M^l.
    \end{equation*}
\end{cor}
\begin{proof}
    We first prove $L_1$ is $d$-lef below degree $l$ on $M$. 
    Set $e=\min\{i\in \mathbb{Z}: M^i\neq 0\}.$
    When $l=e$, there is nothing to prove, since the condition of being
    $d$-lef below degree $e$ on $M$ is vacuous.
    We therefore proceed by induction on $l$. 
    Suppose that the assertion has already been proved below degree $l-1$.
    For $0\leq t\leq 1$, set
    $$N_t=(1-t)L_1+tL_2.$$
    For every $0\leq t<1$, we have $$N_t\sim L_1\gtrsim L_2.$$
    Since $L_2$ is $d$-lef below degree $l$ on $M$, it is in particular
    $d$-lef below degree $l-1$ on $M$. 
    Hence, by the induction hypothesis, $N_t$ is $d$-lef below degree $l-1$ on $M$ for every $0\leq t<1$.
    Lemma \ref{alglefequiv} therefore gives
    \begin{equation*}
        \ker N_t^{d-2l+2}\cap M^{l-1} = W_{d-2l+1}(N_t)\cap M^{l-1}.
    \end{equation*}
    By Lemma \ref{invwtfil}, $W_\bullet(N_t)=W_\bullet(L_1)$ for $0\leq t<1$.
    Hence for $0\leq t<1$, we have
    \begin{equation*}
        \ker N_t^{d-2l+2}\cap M^{l-1} = \ker L_1^{d-2l+2}\cap M^{l-1},
    \end{equation*}
    Letting $t\to1$, we obtain 
    \begin{equation*} 
        \ker L_1^{d-2l+2}\cap M^{l-1} \subset \ker L_2^{d-2l+2}\cap M^{l-1}. 
    \end{equation*} 
    Since $L_2$ is $d$-lef below degree $l$ on $M$, the space on the right-hand side vanishes. Therefore 
    \begin{equation*} 
        \ker L_1^{d-2l+2}\cap M^{l-1}=0, 
    \end{equation*} 
    and hence $L_1$ is $d$-lef below degree $l$ on $M$.
    
    Suppose now that $2l+1\leq d$. 
    We have already shown that both $N_t$ and $L_1$ are $d$-lef below degree $l$ on $M$ for every $0\leq t<1$.
    Since $N_t\sim L_1$, Lemma \ref{alglefequiv} and Lemma \ref{invwtfil} give
    \begin{equation*} 
        \ker N_t^{d-2l}\cap M^l = \ker L_1^{d-2l}\cap M^l. 
    \end{equation*}
    Letting $t\to1$, we obtain
    \begin{equation*}
        \ker L_1^{d-2l}\cap M^l \subset \ker L_2^{d-2l}\cap M^l.
    \end{equation*}

    Finally, if $L_1\sim L_2$, then both $L_1\gtrsim L_2$ and
    $L_2\gtrsim L_1$. Applying the first assertion in both directions
    proves that $L_1$ is $d$-lef below degree $l$ on $M$ if and only if $L_2$ is. Moreover, when $2l+1\leq d$, applying the kernel inclusion in both directions gives
    \begin{equation*}
        \ker L_1^{d-2l}\cap M^l = \ker L_2^{d-2l} \cap M^l.
    \end{equation*}
\end{proof}

For collections, we define lefness using the equivalent characterizations in Lemma \ref{equivlefcolfree}.

\begin{defn}
    Let $d,l$ be integers satisfying $2l-1\leq d\leq n$.
    A collection of nef classes 
    $\mathcal{L}=(L_1,...,L_{m})\in \overline{\mathcal{K}_A}$ 
    is said to be  $d$-lef below degree $l$ on $M$ if 
    $1\leq m\leq d-2l+2$ and for every nonempty subset $I\subset [m]$, the class $L_I$ is $(|I|+d-m)$-lef below degree $l$ on $M$.
\end{defn}

As in the geometric setting, the lefness of a collection is preserved under
repetitions of its classes and increased by one under adjoining a K\"ahler class.
Denote by 
\begin{equation*}
    (L_1[\alpha_1],\ldots,L_m[\alpha_m])=(L_1,...,L_1,\ldots,L_m,...,L_m)
\end{equation*}
the collection with $L_i$ repeated $\alpha_i$ times.

\begin{lem}\label{addalg}
    Let $d,l$ be integers satisfying $2l-1\leq d\leq n$, and let 
    $\mathcal{L}=(L_1,\ldots,L_m) \in \overline{\mathcal{K}_A}$
    be a collection of nef classes that is $d$-lef below degree $l$ on
    $M$. Then the following statements hold.
    \begin{enumerate}
        \item
        If $m<d-2l+2$, then for every nef class $L$ that is $(d-m)$-lef below degree $l$ on $M$, the collection $(L_1,...,L_m,L)$ is $d$-lef below degree $l$ on $M$.
        \item
        For every
        $\alpha=(\alpha_1,\ldots,\alpha_m)\in\mathbb{Z}_{>0}^m$
        with $|\alpha|\leq d-2l+2$, the collection
        $(L_1[\alpha_1],\ldots,L_m[\alpha_m])$
        is $d$-lef below degree $l$ on $M$.

        \item
        If $d\leq n-1$, then for every $\eta\in\mathcal{K}_A$, the
        collection $(L_1,\ldots,L_m,\eta)$
        is $(d+1)$-lef below degree $l$ on $M$.
    \end{enumerate}
\end{lem}
\begin{proof}
    $(1)$
    Since $\mathcal{L}$ is $d$-lef below degree $l$ on $M$, the class $L_I$ is $(d-m+|I|)$-lef below degree $l$ on $M$. 
    In particular, $L_I$ is $(d-(m+1)+|I|)$-lef below degree $l$ on $M$.
    Moreover, $L$ is $(d-m)$-lef below degree $l$ on $M$ by assumption.
    It remains to consider sums of the form $L_I+L$, where $\emptyset\neq I\subset[m]$. 
    Since $L_I+L\gtrsim L_I$, Corollary \ref{kerinclusion} implies that $L_I+L$ is $(d-m+|I|)$-lef below degree $l$ on $M$.
    Hence $(L_1,\ldots,L_m,L)$ is $d$-lef below degree $l$ on $M$.

    $(2)$ Observe that
    $(L_1[\alpha_1],\ldots,L_m[\alpha_m])$ is obtained from
    $\mathcal{L}$ by successively adding copies of the classes $L_i$.
    Since $\mathcal{L}$ is $d$-lef below degree $l$ on $M$,
    each $L_i$ is $(d-m+1)$-lef below degree $l$ on $M$. 
    As the number of classes increases, the lef degree required in
    $(1)$ only decreases. 
    Therefore $(1)$ can be applied successively, and the repeated collection remains $d$-lef below
    degree $l$ on $M$. 

    $(3)$
    By definition, it suffices to show that
    $\eta$ is $(1+d-m)$-lef below degree $l$ on $M$,
    $L_I$ is $(|I|+d-m)$-lef below degree $l$ on $M$ and $L_I+\eta$ is $(|I|+d-m+1)$-lef below degree $l$ on $M$ for every nonempty subset $I\subset[m]$.
    The first and third assertions are immediate since $\eta$ and $L_I+\eta$ are both K\"ahler, 
    while the second assertion follows from the assumption that
    $\mathcal{L}$ is $d$-lef below degree $l$ on $M$.    
\end{proof}

\subsection{Descent and purity}
In this subsection, we introduce the key inputs for our kernel characterization theorem for Lefschetz modules --  the descent lemma and the purity theorem of Cattani, Kaplan and Schmid \cite{CKS87}. These results replace the use of general hypersurfaces in the geometric proof. 
In loc.~cit., all results are stated for polarized Hodge-Lefschetz modules.
Using the relation between complex Lefschetz modules and polarized Hodge-Lefschetz modules established in Lemma \ref{PHLmodules} and \ref{PHLmodulesconverse}, we first deduce the versions needed here. We then use the decomposition theorem for Lefschetz modules to extend the purity statement to arbitrary nef classes.

As before, we fix a complex Lefschetz module $(M,Q_M)$ of degree $n$ over $(A,\mathcal{K}_A)$.
For any $L\in A^1$, we set $M_L:=M/\ker(L)$ and equip $M_L$ with the induced pairing 
$$Q_{M_L}([x],[y]):=Q_M(x,Ly) ,\ \ \  \forall [x],[y]\in M_L.$$ 
The descent lemma asserts that the resulting module is again Lefschetz whenever $L$ is nef.

\begin{lem}\label{descentlem}
    For every $L\in \overline{\mathcal{K}_A}$, the pair $(M_L,Q_{M_L})$ is a complex Lefschetz module of degree $n-1$ over
    $(A,\mathcal{K}_A)$.
\end{lem}
\begin{proof}
    Let $\eta \in \mathcal{K}_A$ and let $(V_*,S,\mathfrak{a},\eta)$ be the polarized Hodge-Lefschetz module of weight $n+1$ in Lemma \ref{PHLmodules}. 
    Since $L\in\overline{\mathcal K_A}$, we have $L+t\eta \in \mathcal{K}_A$ for all $t>0$.
    Hence the descent lemma for polarized Hodge-Lefschetz modules \cite[Lemma 1.16]{CKS87} applies to $L$ (see also \cite[Theorem 3.2 and Remark 3.3]{cattanimixedHRR}).
    Set $\overline{V}=V/\ker L$ with grading $\overline{V}_i= V_{i+1}/\ker L \cap V_{i+1}$ and Hodge bigrading $\overline{V}^{p,q}=V^{p+1,q+1}/\ker L \cap V^{p+1,q+1}$,
    and define $\overline{S}([x],[y])=S(x,Ly)$. 
    Then $(\overline{V}_*,\overline{S},\overline{\mathfrak{a}},\eta)$ is a polarized Hodge-Lefschetz module of weight $n$ where $\overline{\mathfrak a}$ denotes the induced action of $\mathfrak a$ on $\overline V$.
    By Lemma \ref{PHLmodulesconverse}, $M_L=\bigoplus_{0\leq i\leq  n-1} \overline{V}^{n-i,n-1-i}=\bigoplus_{0\leq i\leq n-1} M^i_L$ endowed with Hermitian form $Q_{M_L}$ is a complex Lefschetz module of degree $n-1$ over $(A,\mathcal{K}_A)$.
\end{proof}

\begin{lem}\label{purity}
    Let $L_1,L_2\in \overline{\mathcal{K}_A}$ be nef classes such that
    $L_1+L_2\in\mathcal{K}_A$. Then for all integers $s_1,s_2\geq 0$,
    \begin{equation*}
        \ker(L_1^{s_1}L_2^{s_2}) \subset \ker(L_1^{s_1}) + \ker(L_2^{s_2}) +    W_{s_1+s_2-1}(L_1+L_2,M).
    \end{equation*}
\end{lem}
\begin{proof}
    Let $\eta =L_1+L_2 \in \mathcal{K}_A$ and let $(V_*,S,\mathfrak{a},\eta)$ be the polarized Hodge-Lefschetz module of weight $n+1$ as in Lemma \ref{PHLmodules}. 
    Viewing $L_1,L_2\in \End_{\mathbb{C}}(V_{\mathbb{C}})$,
    \cite[Lemma~3.16]{CKS87} gives
    \begin{equation*}
        \ker_{V_{\mathbb{C}}}(L_1^{s_1}L_2^{s_2}) \subset
        \ker_{V_{\mathbb{C}}}(L_1^{s_1})+\ker_{V_{\mathbb{C}}}(L_2^{s_2})+W_{s_1+s_2-1}(L_1+L_2,V_{\mathbb{C}}).
    \end{equation*}
    Since $V_{\mathbb{C}}=M\oplus \overline{M}$ and $L_1,L_2$ preserve both summands, we have 
    \begin{equation*}
        \ker_{V_{\mathbb{C}}}(L_1^{s_1}L_2^{s_2})=
        \ker_M(L_1^{s_1}L_2^{s_2})\oplus  \ker_{\overline{M}}(L_1^{s_1}L_2^{s_2})
    \end{equation*}
    and similarly for $\ker_{V_{\mathbb{C}}}(L_1^{s_1})$, $\ker_{V_{\mathbb{C}}}(L_2^{s_2})$. Moreover, the weight filtration is compatible with this decomposition
    \begin{equation*}
        W(L_1+L_2,V_{\mathbb C}) = W(L_1+L_2,M) \oplus W(L_1+L_2,\overline M). 
    \end{equation*}
    Taking the $M$-summand of the first inclusion then yields the desired assertion.
\end{proof}

The decomposition theorem for Lefschetz modules allows us to drop the
assumption that $L_1+L_2\in \mathcal{K}_A$.

\begin{prop}\label{algpurity}
    Let $L_1,L_2\in \overline{\mathcal{K}_A}$ be nef classes. Then for
    all integers $s_1,s_2\geq 0$, we have
    \begin{equation*}
        \ker(L_1^{s_1}L_2^{s_2})\subset
        \ker(L_1^{s_1})+\ker(L_2^{s_2})+W_{s_1+s_2-1}(L_1+L_2,M).
    \end{equation*}
\end{prop}
\begin{proof}
    Set $S=L_1+L_2,$
    and let $B\subset A$ be the graded subalgebra generated by $L_1$
    and $L_2$. Define $\mathcal{K}_B=\operatorname{int}_{B_1}(\overline{B^1\cap \mathcal{K}_A}).$
    By construction, $L_1,L_2\in \overline{ \mathcal{K}_B}$ and $S\in\mathcal{K}_B$.
    Choose a decomposition
    \begin{equation*}
         M\cong \bigoplus_{\alpha} \bigoplus_{k}
        N_\alpha[-k]^{\oplus m(\alpha,k)}
    \end{equation*}
    into indecomposable graded $B$-modules. 
    Since the desired statement is
    invariant under $B_{\mathbb{C}}$-module isomorphisms, we may regard this
    decomposition as an internal direct sum.

    By the decomposition theorem of Amini-Huh-Larson
    \cite{aminihuhlarson25lefschetzmodule} recalled in Section \ref{Lefschetzmodules}, each $N_\alpha$ admits a
    nondegenerate $B$-invariant Hermitian form $Q_\alpha$ such that the pair
    $(N_\alpha,Q_\alpha)$ is a complex Lefschetz module over
    $(B,\mathcal{K}_B)$.
    Since $S=L_1+L_2\in\mathcal{K}_B$, Lemma \ref{purity} applied to
    each $N_\alpha$ gives
    \begin{equation}\label{dnfqo}
        \ker((L_{1|N_\alpha})^{s_1}(L_{2|N_\alpha})^{s_2})
    \subset\ker((L_{1|N_\alpha})^{s_1})+\ker((L_{2|N_\alpha})^{s_2})+W_{s_1+s_2-1}(S_{|N_\alpha},N_\alpha).
    \end{equation}

    Kernels are compatible with direct sums and grading shifts.
    Therefore,
    \begin{equation*}
        \ker(L_1^{s_1}L_2^{s_2})=
        \bigoplus_{\alpha,k}
        \ker((L_{1|N_\alpha})^{s_1}(L_{2|N_\alpha})^{s_2})[-k]^{\oplus m(\alpha,k)}.
    \end{equation*}
    The analogous decompositions hold for
    $\ker(L_1^{s_1})$ and $\ker(L_2^{s_2})$.
    The weight filtration of a nilpotent operator is also compatible
    with direct sums and grading shifts. Hence
    \begin{equation*}
        W_{s_1+s_2-1}(S,M)=
        \bigoplus_{\alpha,k}W_{s_1+s_2-1}(S_{|N_\alpha},N_\alpha)[-k]^{\oplus m(\alpha,k)}.
    \end{equation*}
    Taking the direct sum of the inclusions
    \eqref{dnfqo} then yields
    \begin{equation*}
        \ker(L_1^{s_1}L_2^{s_2})\subset
        \ker(L_1^{s_1})+\ker(L_2^{s_2})+W_{s_1+s_2-1}(L_1+L_2,M),
    \end{equation*}
    as desired.
\end{proof}

We are now ready to prove the preservation of lefness under descent. Its geometric analogue is Proposition \ref{lefpres}.

\begin{lem}\label{alglefpres}
    Let $d,l$ be integers satisfying $2l-1\leq d\leq n-1,$
    and let $L_1,L_2\in\overline{\mathcal{K}_A}$ be nef classes.
    Suppose that $L_1$ is $d$-lef below degree $l$ on $M$ and that
    $L_1+L_2$ is $(d+1)$-lef below degree $l$ on $M$.
    Then $L_1$ is $d$-lef below degree $l$ on $M_{L_2}$.
\end{lem}

\begin{proof}
    We need to show that the induced map
    \begin{equation*}
        L_1^{d-2k}\colon M_{L_2}^k\to M_{L_2}^{d-k}
    \end{equation*}
    is injective for every integer $ k\leq l-1$.
    Let $[x]\in \ker (L_1^{d-2k}\colon M_{L_2}^k\longrightarrow M_{L_2}^{d-k})$ and choose a representative $x\in M^k$.
    We then have $L_1^{d-2k}L_2x=0.$
    By Proposition \ref{algpurity}, we have
    \begin{equation*}
        \ker(L_1^{d-2k}L_2)\cap M^k \subset
        \ker L_1^{d-2k}\cap M^k+\ker L_2\cap M^k+W_{d-2k}(L_1+L_2,M)\cap M^k.
    \end{equation*}
    Since $L_1$ is $d$-lef below degree $l$ on $M$,
    \begin{equation*}
        \ker L_1^{d-2k}\cap M^k=0.
    \end{equation*}
    Moreover, since $L_1+L_2$ is $(d+1)$-lef below degree $l$ on $M$,
    Lemma \ref{alglefequiv} gives
    \begin{equation*}
        W_{d-2k}(L_1+L_2,M)\cap M^k=0.
    \end{equation*}
    Since $x\in\ker(L_1^{d-2k}L_2)\cap M^k$, we conclude that $x\in\ker L_2.$ Hence $[x]=0$ in $M_{L_2}$. This proves the injectivity of $L_1^{d-2k}$ on $M_{L_2}^k$.
    Thus $L_1$ is $d$-lef below degree $l$ on $M_{L_2}$, completing the proof.
\end{proof}

\begin{cor}\label{collectiondescent}
    Let $d,l$ be integers satisfying $2l-1\leq d\leq n$, and let $\mathcal{L}=(L_1,\ldots,L_m) $
    be a collection of nef classes that is $d$-lef below degree $l$ on $M$ with $2\leq m\leq d-2l+2$.
    Then for any $r\in[m]$, the collection
        \begin{equation*}
            \mathcal{L}_{\backslash r}=(L_1,\ldots,L_{r-1},L_{r+1},\ldots,L_m)
        \end{equation*}
        is $(d-1)$-lef below degree $l$ on $M_{L_r}$.
\end{cor}

\begin{proof}
    Fix $r\in[m]$ and a nonempty subset $I\subset[m]\backslash \{r\}$. 
    Since $\mathcal{L}$ is $d$-lef below degree $l$, the class $L_I$ is
    $(|I|+d-m)$-lef below degree $l$ on $M$ and $L_I+L_r$
    is $(|I|+d-m+1)$-lef below degree $l$ on $M$.
    Hence Lemma \ref{alglefpres} implies that $L_I$ is $(|I|+d-m)$-lef below
    degree $l$ on $M_{L_r}$.
    Since $I$ is arbitrary, this is exactly the condition for $\mathcal{L}_{\backslash r}$ to be $(d-1)$-lef below degree $l$ on $M_{L_r}$. 
    This completes the proof.
\end{proof}

By the descent lemma, we can also show that multiplying by K\"ahler classes does not affect the kernel.

\begin{prop}\label{keraddKah}
    Let $d,l\geq 0$ be integers such that $2l\leq d\leq n$.
    Let $L_1,...,L_{d-2l}\in \overline{\mathcal{K}_A}$ be nef classes and let $\eta_1,...,\eta_{n-d}\in \mathcal{K}_A$ be K\"ahler classes. Then
    \begin{equation*}
        \ker(L_1\cdots L_{d-2l}\cdot\eta_1\cdots \eta_{n-d}) \cap M^l=\ker(L_1\cdots L_{d-2l})\cap M^l.
    \end{equation*}
\end{prop}
\begin{proof}
    It suffices to show
    \begin{equation*}
        \ker(L_1\cdots L_{d-2l}\cdot\eta_1\cdots \eta_{n-d}) \cap M^l\subset \ker(L_1\cdots L_{d-2l})\cap M^l.
    \end{equation*}
    Let $x\in \ker(L_1\cdots L_{d-2l}\cdot\eta_1\cdots \eta_{n-d}) \cap M^l$. Then its image $[x]\in M_{L_1\cdots L_{d-2l}}:=M/\ker(L_1\cdots L_{d-2l})$  satisfies $\eta_1\cdots \eta_{n-d}\cdot [x]=0$.
    On the other hand, Lemma \ref{descentlem} implies that $M_{L_1\cdots L_{d-2l}}$ endowed with the induced pairing is a complex Lefschetz module of degree $n-d+2l$ over $(A,\mathcal{K}_A)$. 
    Therefore the mixed hard Lefschetz property on $M_{L_1\cdots L_{d-2l}}$ gives $[x]=0$. Namely, $L_1\cdots L_{d-2l}\cdot x=0$. This completes the proof.
\end{proof}

\subsection{Kernel decoupling}

In this subsection, we provide the proof of Theorem \ref{alglefcharintro} and \ref{algkerintro}.
We start with the following kernel decoupling result.

\begin{thrm}\label{decoupling}
    Let $d,l$ be integers satisfying $2l+1\leq d\leq n$, and let
    $\mathcal{L}=(L_1,\ldots,L_{d-2l})$
    be a collection of nef classes that is $d$-lef below degree $l$
    on $M$. 
    Then
    \begin{equation*}
        \ker\mathcal{L}\cap M^l = 
        \sum_{\emptyset\neq I\subset[d-2l]} \ker L_I^{|I|}\cap M^l.
    \end{equation*}
\end{thrm}

The proof follows the same inductive strategy as that of Theorem \ref{mixedker(p,q)}. Here the descent lemma plays the role of restriction to a hypersurface, while the purity theorem allows us to lift the resulting information back to the original Lefschetz module.

We begin with the following inclusion which follows from the invariance of the weight filtration on the relative interior of each face of the nef cone and Lemma \ref{alglefequiv}.

\begin{lem}\label{singlekerinmixed}
    Let $L_1,...,L_m\in\overline{\mathcal K_A}$ be nef classes and set $ N=L_1+\cdots+L_m.$
    Suppose that $N$ is $(m+2l)$-lef below degree $l$. Then
    \begin{equation*}
        \ker N^m\cap M^l \subset \ker (L_1\cdots L_m)\cap M^l.
    \end{equation*}
\end{lem}
\begin{proof}
    Let $W_\bullet=W_\bullet(N,M)$ be the weight filtration of $N$. 
    We first observe that
    \begin{equation}\label{cnkwd}
        L_iW_j\subset W_{j-2}
    \end{equation}
    for every $1\leq i\leq m$. 
    Indeed, $N\sim N+L_i$. 
    Hence Lemma \ref{invwtfil} gives
    \begin{equation*}
        W_\bullet(N+ L_i,M)  = W_\bullet(N,M).
    \end{equation*}
    By the defining property of the weight filtration, both $N+L_i$ and $N$ send $W_j$ into $W_{j-2}$. 
    Therefore their difference $L_i$ also sends $W_j$ into $W_{j-2}$. 
    
    Now let $x\in\ker N^m\cap M^l.$
    Since $N$ is $(m+2l)$-lef below degree $l$, Lemma \ref{alglefequiv} gives
    \begin{equation*}
        \ker N^m\cap M^l = W_{m-1}\cap M^l.
    \end{equation*}
    By \eqref{cnkwd}, we have $L_1\cdots L_mx\in W_{-m-1}\cap M^{l+m}.$
    Using $W_{-m-1}=N^{m+1}W_{m+1}$, we obtain
    \begin{equation*}
        W_{-m-1}\cap M^{l+m}=N^{m+1}(W_{m+1}\cap M^{l-1}).
    \end{equation*}
    On the other hand, since $N$ is $(m+2l)$-lef below degree $l$, Lemma \ref{alglefequiv} gives $W_{m+1}\cap M^{l-1}=0$. 
    Hence $W_{-m-1}\cap M^{l+m}=0$. A fortiori, $L_1\cdots L_mx=0$. This proves the inclusion
    \begin{equation*}
        \ker N^m\cap M^l \subset \ker (L_1\cdots L_m)\cap M^l.
    \end{equation*}
\end{proof}

\begin{proof}[Proof of Theorem \ref{decoupling}]
    Let  $I\subset[d-2l]$ be a nonempty subset. 
    Since $(L_1,\ldots,L_{d-2l})$ is $d$-lef below degree $l$, 
    the class $L_I$ is $(|I|+2l)$-lef below degree $l$.
    Lemma \ref{singlekerinmixed} therefore gives
    \begin{equation*}
        \ker L_I^{|I|}\cap M^l \subset 
        \ker (\prod_{i\in I}L_i )\cap M^l \subset \ker \mathcal{L}\cap M^l.
    \end{equation*}
    Summing over all such $I$, we obtain
    \begin{equation*}
        \sum_{\emptyset\neq I\subset[d-2l]}\ker L_I^{|I|}\cap M^l \subset
        \ker \mathcal{L}\cap M^l.
    \end{equation*}
    
    It remains to prove
    \begin{equation*}
        \ker\mathcal{L}\cap M^l \subset \sum_{\emptyset\neq I\subset[d-2l]} \ker L_I^{|I|}\cap M^l.
    \end{equation*}
    When $d-2l=1$, the assertion is immediate. 
    We proceed by induction on $d-2l$. 
    Suppose that $d-2l\geq2$ and that the assertion holds for
    $d-2l-1$.
    Let $x\in\ker\mathcal{L}\cap M^l.$ 
    Its image $[x]\in M_{L_{d-2l}}^l$ satisfies
    \begin{equation*}
        L_1\cdots L_{d-2l-1}[x]=0.
    \end{equation*}
    By Corollary \ref{collectiondescent}, the collection
    $(L_1,\ldots,L_{d-2l-1})$ is $(d-1)$-lef below degree $l$ on
    $M_{L_{d-2l}}$. 
    By the induction hypothesis, we have
    \begin{equation*}
        [x] = \sum_{\emptyset\neq I\subset[d-2l-1]}[x_I],
    \end{equation*}
    where $L_I^{|I|}[x_I]=0$ in $M_{L_{d-2l}}$.
    Choose representatives $x_I\in M^l$.
    The above equality in $M_{L_{d-2l}}$ means that 
    \begin{equation*} 
        x-\sum_{\emptyset\neq I\subset[d-2l-1]}x_I \in\ker L_{d-2l}. 
    \end{equation*} 
    Moreover, $L_I^{|I|}[x_I]=0$ in $M_{L_{d-2l}} $ means that $L_{d-2l}L_I^{|I|}x_I=0.$
    Therefore 
    \begin{equation*} 
        x\in \ker L_{d-2l}\cap M^l + \sum_{\emptyset\neq I\subset[d-2l-1]}  \ker(L_{d-2l}L_I^{|I|})\cap M^l . 
    \end{equation*}

    For every $I$, Proposition \ref{algpurity} gives
    \begin{equation*}
        \ker(L_{d-2l}L_I^{|I|})\cap M^l 
        \subset \ker L_{d-2l}\cap M^l + \ker L_I^{|I|}\cap M^l + W_{|I|}(L_I+L_{d-2l},M)\cap M^l
    \end{equation*}
    Since $\mathcal{L}$ is $d$-lef below degree $l$, the class $L_I+L_{d-2l}$ is $(|I|+2l+1)$-lef below degree $l$. Therefore Lemma \ref{alglefequiv} implies that 
    \begin{equation*} 
        W_{|I|}(L_I+L_{d-2l},M)\cap M^l = \ker(L_I+L_{d-2l})^{|I|+1}\cap M^l. 
    \end{equation*}
    Consequently,
    \begin{equation*}
        \ker(L_{d-2l}L_I^{|I|})\cap M^l \subset \ker L_{d-2l}\cap M^l+ \ker L_I^{|I|}\cap M^l + \ker(L_I+L_{d-2l})^{|I|+1}\cap M^l.
    \end{equation*}
    Substituting the above inclusion into the expression for $x$, we obtain 
    \begin{equation*} 
        x\in \sum_{\emptyset\neq I\subset[d-2l]} \ker L_I^{|I|}\cap M^l. 
    \end{equation*} 
    Since $x\in\ker\mathcal{L}\cap M^l$ is arbitrary, it follows
    \begin{equation*} 
        \ker\mathcal{L}\cap M^l \subset \sum_{\emptyset\neq I\subset[d-2l]}  \ker L_I^{|I|}\cap M^l . 
    \end{equation*}
    This completes the proof.
\end{proof}

We are now ready to prove Theorem \ref{alglefcharintro} and \ref{algkerintro}.

\begin{cor}\label{algker}(=Theorem \ref{algkerintro})
    Let $d,l$ be integers satisfying $2l+1\leq d\leq n$, and let $\mathcal{L}=(L_1,...,L_m)$ be a collection of nef classes that is $d$-lef below degree $l$. 
    Then for every $\alpha \in \mathbb{Z}_{>0}^m$ with $|\alpha|=d-2l$, we have
    \begin{equation*}
        \ker \mathcal{L}^{\alpha}\cap M^l =\sum_{\emptyset\neq I\subset [m]}\ker L_I^{\alpha_I}\cap M^l.
    \end{equation*}
\end{cor}
\begin{proof}
    By Lemma \ref{addalg}, the collection 
    \begin{equation*}
        (K_1,...,K_{d-2l})=(L_1[\alpha_1],...,L_{m}[\alpha_m])
    \end{equation*}
    is $d$-lef below degree $l$. 
    Hence, Theorem \ref{decoupling} applied to this collection yields
    \begin{equation*}
        \ker \mathcal{L}^{\alpha} \cap M^l=\sum _{\emptyset \neq I \subset [d-2l]} \ker K_I^{|I|}\cap M^l.
    \end{equation*}
    It remains to prove
    \begin{equation*}
        \sum _{\emptyset \neq I \subset [d-2l]} \ker K_I^{|I|}\cap M^l =\sum_{\emptyset\neq J\subset [m]}\ker L_J^{\alpha_J}\cap M^l.
    \end{equation*}
    Write $[d-2l]=A_1\sqcup\cdots\sqcup A_m$ satisfying $|A_j|=\alpha_j$ and  $K_i=L_j$ for every $i\in A_j$.
    Let $\emptyset\neq I\subset[d-2l]$ and set $J_I=\{j\in[m]:I\cap A_j\neq\emptyset\}.$
    Then 
    $$K_I=\sum _{j\in J_I}|I\cap A_j|L_j\sim L_{J_I}.$$
    By Corollary \ref{kerinclusion}, we obtain
    \begin{equation*}
        \ker K_I^{|I|}\cap M^l= \ker L_{J_I}^{|I|} \cap M^l \subset \ker L_{J_I}^{\alpha_{J_I}} \cap M^l .
    \end{equation*}
    Summing over $I$ gives
    \begin{equation*}
        \sum _{\emptyset \neq I \subset [d-2l]} \ker K_I^{|I|}\cap M^l \subset \sum_{\emptyset\neq J\subset [m]}\ker L_J^{\alpha_J}\cap M^l.
    \end{equation*}
    
    Conversely, let $\emptyset\neq J\subset[m]$ and set $I_J=\bigsqcup_{j\in J}A_j.$
    Then $|I_J|=\alpha_J$ and $$K_{I_J}=\sum_{j\in J}\alpha_j L_j\sim L_J.$$
    Hence Corollary \ref{kerinclusion} implies
    \begin{equation*}
        \ker L_J^{\alpha_J}\cap M^l = \ker K_{I_J}^{|I_J|}\cap M^l.
    \end{equation*}
    Since $I_J$ is a nonempty subset of $[d-2l]$, it follows that
    \begin{equation*}
        \sum_{\emptyset\neq J\subset[m]} \ker L_J^{\alpha_J}\cap M^l \subset \sum_{\emptyset\neq I\subset[d-2l]} \ker K_I^{|I|}\cap M^l.
    \end{equation*}
    Combining the two inclusions completes the proof.
\end{proof}

\begin{cor}\label{alglefchar}(=Theorem \ref{alglefcharintro})
    Let $l,d$ be integers satisfying $ 2l-1\leq d\leq n$ and let $\mathcal{L}=(L_1,...,L_{m})$ be a collection of nef classes
    with $1\leq m\le d-2l+2$. Then the following statements are equivalent.
    \begin{enumerate}
        \item $\mathcal{L}$ is $d$-lef below degree $l$ on $M$. Namely, for every integer $k\leq l-1$ and every nonempty $I\subset [m]$, the linear map
        \begin{equation*}
            L_I^{|I|+d-m-2k}:M^k \to M^{|I|+d-m-k}
        \end{equation*}
        is injective.
        \item For every integer $k\leq l-1$ and every $\alpha\in \mathbb{Z}_{>0}^m$ with $|\alpha|=d-2k$, the linear map
        \begin{equation*}
            \mathcal{L}^{\alpha}:M^{k}\to M^{d-k}
        \end{equation*}
        is injective.
    \end{enumerate}
\end{cor}
\begin{proof}
    $(1)\implies (2)$ For every integer $k\leq l-1$, the collection $\mathcal{L}$ is in particular $d$-lef below degree $k$. 
    Hence Corollary \ref{algker} implies that
    \begin{equation*}
        \ker\mathcal{L}^{\alpha} \cap M^k=\sum_{\emptyset\neq I\subset [m]} \ker L_I^{\alpha_I}\cap M^k
    \end{equation*}
    for every $\alpha\in \mathbb{Z}_{>0}^m$ with $|\alpha|=d-2k$.
    Since $|\alpha|=d-2k$ and $\alpha_j\geq 1$ for every $j\notin I$, we have
    \begin{equation*}
        \alpha_I \leq d-2k- (m-|I|) =|I|+d-m-2k.
    \end{equation*}
    Therefore
    \begin{equation*}
        \ker L_I^{\alpha_I} \cap M^k \subset \ker L_I^{|I|+d-m-2k}\cap M^k.
    \end{equation*}
    On the other hand, since $\mathcal{L}$ is $d$-lef below degree $l$, the class $L_I$ is $(|I|+d-m)$-lef below degree $l$. Hence $\ker L_I^{|I|+d-m-2k}\cap M^k=0.$
    It follows that $\ker\mathcal{L}^{\alpha}\cap M^k=0, $ 
    proving $(2)$.
    
    $(2)\implies (1)$
    Set $e=\min\{i\in \mathbb{Z}:M^i\neq 0\}.$
    We prove by induction on $k$ that $\mathcal{L}$ is $d$-lef below degree $k$ for every $e\leq k\leq l$. The case $k=e$ is trivial since the condition is vacuous. 
    Now suppose that $\mathcal{L}$ is $d$-lef below degree $k$ for some $e\leq k < l$.
    By Corollary \ref{algker}, we have
    \begin{equation*}
        \ker \mathcal{L}^{\alpha}\cap M^k =\sum_{\emptyset\neq I\subset [m]} \ker L_I^{\alpha_I}\cap M^k
    \end{equation*}
    for every $\alpha \in \mathbb{Z}_{>0}^m$ with $|\alpha|=d-2k$.
    The left-hand side vanishes by $(2)$. Hence
    \begin{equation*}
    \ker L_I^{\alpha_I}\cap M^k=0
    \end{equation*}
    for every nonempty subset $I\subset[m]$ and every $\alpha \in \mathbb{Z}_{>0}^m$ with $|\alpha|=d-2k$. 
    For a fixed nonempty subset $I\subset[m]$, the quantity $\alpha_I$
    attains its maximum when $\alpha_j=1$ for every $j\notin I$.
    In this case,
    \begin{equation*}
        \max\alpha_I= d-2k-(m-|I|) = |I|+d-m-2k.
    \end{equation*}
    Therefore $\ker L_I^{|I|+d-m-2k}\cap M^k=0.$
    
    Together with the induction hypothesis, this shows that $L_I$ is
    $(|I|+d-m)$-lef below degree $k+1$ for every nonempty
    $I\subset[m]$. 
    Equivalently, $\mathcal{L}$ is $d$-lef below degree
    $k+1$. 
    The induction proves that $\mathcal{L}$ is $d$-lef
    below degree $l$, completing the proof.
\end{proof}

We can also prove the weak Hodge-Riemann relations in this setting.

\begin{cor}\label{weakHRalg}
    Let $l$ be an integer satisfying $2l\leq n$ and
    let $L,L_1,...,L_{n-2l}\in \overline{\mathcal{K}_A}$ be nef classes.
    Set $\mathcal{L}=(L_1,...,L_{n-2l})$ and consider the Hermitian form
    \begin{equation*}
        H: M^l\times M^l \to \mathbb{C} \ \ \  (x,y) \mapsto (-1)^lQ_M(x,\mathcal{L}  y).
    \end{equation*}
    Suppose that the collection $(\mathcal{L},L,L)$ is $n$-lef below degree $l$. Then there exists a $H$-orthogonal decomposition $M^l=LM^{l-1}\oplus P^l$
    where $P^l=\ker(\mathcal{L}L)\cap M^l$ is the primitive space.
    Moreover, $H_{|P^l}$ is positive semidefinite with $\ker H_{|P^l}=\ker \mathcal{L}\cap M^l$. 

    Here the empty collection $\mathcal{L}=\emptyset$ is allowed, with
    the convention $\mathcal{L}=1$.
\end{cor}
\begin{proof}
    By Corollary \ref{alglefchar}, the linear map
    \begin{equation*}
        \mathcal{L} L^2: M^{l-1}\to M^{n-l+1}
    \end{equation*}
    is injective. Since $\dim M^{l-1} =\dim M^{n-l+1}$, it is in fact an isomorphism. 
    The remaining proof is the same as that of Lemma \ref{lefdecomp} and we leave the details to the reader.

\end{proof}

The kernel $\ker\mathcal{L}\cap M^l$ appearing in Corollary \ref{weakHRalg} is characterized by Theorem \ref{decoupling}.
Consequently, if $\mathcal{L}$ is moreover $n$-lef below degree $l+1$, then the Hermitian form $H$ is positive definite on $P^l$, recovering the usual Hodge-Riemann relations.

 \subsection{Geometric applications}\label{geomapp}
In this subsection, we apply the preceding results to Lefschetz modules arising from the cohomology of smooth projective varieties over $\mathbb{C}$. This gives a new proof of the main results of Section \ref{geometricdecoupling}. We also prove Theorem \ref{nefker(1,1)intro} and Corollary \ref{extremalKT}.

Let $X$ be a smooth projective variety of dimension $n$ over $\mathbb{C}$. 
We use the notation introduced in Example \ref{geomlefmod}.
Applying Corollaries \ref{alglefchar} and \ref{algker} to these Lefschetz modules, we get the following decoupling result in the geometric setting.



\begin{thrm}\label{geomdecoupLefmod}
    Let $d,p,q\geq 0$ be integers satisfying $p+q\leq d\leq n$ and let $\mathcal{L}=(L_1,...,L_m)$ be a collection of nef classes with $1\leq m\leq d-p-q+2$. Then the following statements are equivalent.
    \begin{enumerate}
        \item For every $i\geq 1$ and $\alpha\in \mathbb{Z}_{>0}^m$ with $|\alpha|=d-p-q+2i$, the map
        \begin{equation*}
            \mathcal{L}^{\alpha}:H^{p-i,q-i}(X)\to H^{d-q+i,d-p+i}(X)
        \end{equation*}
        is injective.
        \item 
        For every $i\geq 1$ and any nonempty $I\subset [m]$, the map
        \begin{equation*}
            L_I^{|I|-m+d-p-q+2i}:H^{p-i,q-i}(X) \to H^{|I|-m+d-q+i,|I|-m+d-p+i}(X)
        \end{equation*}
        is injective.
    \end{enumerate}
    Moreover, if these equivalent conditions hold and $m\leq d-p-q$, then for every $\alpha\in \mathbb{Z}_{>0}^m$ with $|\alpha|=d-p-q$, we have
    \begin{equation*}
        \ker \mathcal{L}^{\alpha}\cap H^{p,q}(X)=\sum_{\emptyset \neq I \subset [m]} \ker L_I^{\alpha_I}\cap H^{p,q}(X).
    \end{equation*}
\end{thrm}
\begin{proof}
    By complex conjugation, we assume that $p\geq q$. 
    By Example \ref{geomlefmod}, $M_{p-q}$ is a complex Lefschetz module of degree $n-p+q$ and we have $H^{p-i,q-i}(X)=M_{p-q}^{q-i}$.
    The statements then follow immediately from Corollaries \ref{alglefchar} and \ref{algker}.
\end{proof}


When $p=q=0$, the injectivity of $\mathcal{L}^{\alpha}:H^{0,0}(X)\to H^{d,d}(X)$ is equivalent to the nonvanishing of the class $\mathcal{L}^{\alpha}$. 
As an immediate consequence, we recover the following Hall-Rado criterion for nef classes.

\begin{cor}\label{hallradonef}
    Let $m,d$ be integers satisfying $1\leq m\leq d\leq n$ and
    $\mathcal{L}=(L_1,...,L_{m})$ be a collection of nef classes.
    Then for any $\alpha\in \mathbb{Z}_{>0}^m$ with $|\alpha|=d$, we have
    \begin{equation*}
        \mathcal{L}^{\alpha}\neq 0\iff  \nd(L_I)\geq \alpha_I, \forall \emptyset\neq I\subset [m]
    \end{equation*}
\end{cor}

We now apply the preceding results to give an alternative proof of Theorems \ref{mixedlefcharintro} and \ref{mixedkerintro}.

\begin{proof}[Proof of Theorem \ref{mixedlefcharintro}]
    The implications $(3)\implies (2) \implies (1)$ are clear. It remains to prove $(1)\implies (3)$.
    By Lemma \ref{equivlefcolfree}, $L_I$ is $(|I|+d-m)$-lef below degree $l$ for every nonempty $I\subset [m]$. 
    Hence Theorem \ref{purelefcharintro} implies that the linear map
    \begin{equation*}
        L_I^{|I|+d-m-p-q}:H^{p,q}(X) \to H^{|I|+d-m-q,|I|+d-m-p}(X)
    \end{equation*}
    is injective for all $p+q\leq 2l-2$. Theorem \ref{geomdecoupLefmod} then implies that 
    \begin{equation*}
        \mathcal{L}^{\alpha}:H^{p,q}(X) \to H^{d-q,d-p}(X)
    \end{equation*}
    is injective for all $p+q\leq 2l-2$ and $\alpha\in \mathbb{Z}_{>0}^m$ with $|\alpha|=d-p-q$. This finishes the proof.

\end{proof}

\begin{proof}[Proof of Theorem \ref{mixedkerintro}]
    Fix integers $p,q\geq 0$ satisfying $2l-1\leq p+q\leq 2l$ and $\alpha\in \mathbb{Z}_{>0}^m$ with $|\alpha|=d-p-q$. 
    Since $\mathcal{L}$ is $d$-lef below degree $l$, Theorem \ref{mixedlefcharintro} implies that the linear map
    \begin{equation*}
        \mathcal{L}^{\beta} : H^{p-i,q-i}(X)\to H^{d-q+i,d-p+i}(X)
    \end{equation*}
    is injective for all $i\geq 1$ and $\beta\in\mathbb{Z}_{>0}^m$  with $|\beta|=d-p-q+2i$. Theorem \ref{geomdecoupLefmod} then gives
    \begin{equation*}
        \ker \mathcal{L}^{\alpha}\cap H^{p,q}(X)=\sum_{\emptyset \neq I \subset [m]} \ker L_I^{\alpha_I}\cap H^{p,q}(X).
    \end{equation*}
    By Lemma \ref{equivlefcolfree}, each $L_I$ is $(|I|+d-m)$-lef below degree $l$. In particular, it is $(\alpha_I+p+q)$-lef below degree $l$ since $\alpha_I +m-|I|\leq d-p-q$. 
    It follows from Theorem \ref{purekerintro} that
    \begin{equation*}
        \ker L_I^{\alpha_I}\cap H^{p,q}(X)=
        \begin{cases}
            0 & \text{if }p\neq q \\
            \Span_{\mathbb{C}}\{[V]:V\in \Irred^l(X), L_I^{\alpha_I}\cdot [V]=0\} & \text{if } p=q=l
        \end{cases}.
    \end{equation*}
    On the other hand,  Proposition \ref{hallrado} gives
    \begin{equation*}
        \Span_{\mathbb{C}}\{[V]:V\in \Irred^l(X), \mathcal{L}^{\alpha}\cdot [V]=0\}=\sum_{\emptyset \neq I\subset [m]}
        \Span_{\mathbb{C}}\{[V]:V\in \Irred^l(X), L_I^{\alpha_I}\cdot [V]=0\}.
    \end{equation*}
    Combining the preceding equalities then implies
    \begin{equation*}
        \ker \mathcal{L}^{\alpha}\cap H^{p,q}(X) =
        \begin{cases}
            0 & \text{if }p\neq q \\
            \Span_{\mathbb{C}}\{[V]:V\in \Irred^l(X), \mathcal{L}^{\alpha}\cdot [V]=0\} & \text{if } p=q=l
        \end{cases}.
    \end{equation*}
    The linear independence of these classes follows from Lemma \ref{mixedlinindep}. 
\end{proof}

To prove Theorem \ref{nefker(1,1)intro}, we need a kernel characterization for the powers of a single nef class. 
The arguments of Section \ref{singlelefchar} do not apply directly,
since an arbitrary nef class need not be the pullback of a
K\"ahler class under an algebraic morphism.
Instead, we use the following result of Hu-Xiao \cite[Theorem~A]{HuXiaoLefChar}. Here we state it after tensoring with $\mathbb{C}$

\begin{thrm}\label{hxlefchar}
    Let $X$ be a smooth projective variety over $\mathbb{C}$ of dimension $n$, and let $\mathcal{L}=(L_1,...,L_{n-2})$ be a collection of nef classes. If $\nd(L_i)\ge i+2$ for all $i$, then
    \begin{equation*}
        \ker \mathcal{L}\cap H^{1,1}(X)=\Span_{\mathbb{C}}\{[D]:D\in \prim (X), \mathcal{L}\cdot [D]=0\}.
    \end{equation*}
\end{thrm}

As a corollary, we obtain the single-class kernel characterization as follows.

\begin{cor}\label{nefsingleker}
    Let $X$ be a smooth projective variety over $\mathbb{C}$ of dimension $n$. Let $d$ be an integer with $3\leq d\le n$ and let $L$ be a nef class. If $\nd(L)\ge d$, then
    \begin{equation*}
        \ker L^{d-2}\cap H^{1,1}(X)=\Span_{\mathbb{C}}\{[D]:D\in \prim (X), L^{d-2}\cdot [D]=0\}.
    \end{equation*}
\end{cor}
\begin{proof}
    Fix a K\"ahler class $\omega $. 
    By Proposition \ref{keraddKah}, we have
    \begin{equation*}
        \ker L^{d-2}\cap H^{1,1}(X)=\ker (L^{d-2}\cdot \omega^{n-d})\cap H^{1,1}(X).
    \end{equation*}
    Applying Theorem \ref{hxlefchar} to the collection $(L[d-2],\omega[n-d])$ yields
    \begin{equation*}
        \ker (L^{d-2}\cdot \omega^{n-d})\cap H^{1,1}(X)=
        \Span_{\mathbb{C}}\{[D]:D\in \prim (X), L^{d-2}\cdot \omega^{n-d}\cdot [D]=0\}.
    \end{equation*}
    The result then follows immediately from
    \begin{equation*}
        L^{d-2}\cdot \omega^{n-d}\cdot [D] =0\iff L^{d-2}\cdot [D]=0.
    \end{equation*}
\end{proof}

We now prove Theorem \ref{nefker(1,1)intro}.

\begin{proof}[Proof of Theorem \ref{nefker(1,1)intro}]
    Since the operator $\mathcal{L}$ and divisor classes are defined over
    $\mathbb R$, it suffices to prove the corresponding statement after
    complexification.
    $\mathcal{L}$ being supercritical is equivalent to $\mathcal{L}$ being $n$-lef below degree $1$ on $M_0=\bigoplus_{i=0}^n H^{i,i}(X)$. Hence Theorem \ref{decoupling} gives
    \begin{equation*}
        \ker  \mathcal{L}\cap H^{1,1}(X)=\sum_{\emptyset \neq I \subset [n-2]}\ker L_I^{|I|}\cap H^{1,1}(X).
    \end{equation*}
    Since $\nd(L_I)\geq |I|+2$, Corollary \ref{nefsingleker} implies
    \begin{equation*}
        \ker L_I^{|I|}\cap H^{1,1}(X)=\Span_{\mathbb{C}}\{[D]:D\in \prim (X), L_{I}^{|I|}\cdot [D]=0\}.
    \end{equation*}
    By Proposition \ref{hallrado},
    \begin{equation*}
        \mathcal{L}\cdot [D]=0 \iff L_{I}^{|I|}\cdot [D]=0 \text{ for some } \emptyset \neq I\subset [n-2].
    \end{equation*}
    The desired statement then follows from the preceding equalities immediately.
\end{proof}

To prove Corollary \ref{extremalKT}, we use the following standard fact about hyperbolic quadratic forms; see, for example, \cite{luotieHodgeindex} or \cite[Lemma 3.3]{HuXiaoLefChar}.

\begin{lem}\label{hypextremal}
    Let $V$ be a finite-dimensional real vector space endowed with a quadratic form $q$ with at most one positive eigenvalue. Then for any $\alpha,\beta \in V$ with $q(\beta,\beta)>0$, we have
    \begin{equation*}
        q(\alpha,\beta)^2\geq q(\alpha,\alpha)q(\beta,\beta).
    \end{equation*}
    Moreover, equality holds if and only if $\alpha-\frac{q(\alpha,\beta)}{q(\beta,\beta)}\beta \in \ker q$.
\end{lem}

\begin{proof}[Proof of Corollary \ref{extremalKT}]
    If $A-cB=\sum a_D[D]$ for some prime divisors $D$ annihilated by $A^{k-1}\cdot B^{n-k-1}$, then 
    $$\int_XA^k\cdot B^{n-k}=c\int_XA^{k-1}\cdot B^{n-k+1}=\frac{1}{c}\int_XA^{k+1}\cdot B^{n-k-1}$$
    in which case equality holds clearly.
    
    We now assume equality holds. 
    Together with assumption $\int_XA^k\cdot B^{n-k}>0$, each term must be strictly positive. 
    Hence $\nd_X(A)\geq k+1,$ $\nd_X(B)\geq n-k+1$ and $\nd_X(A+B)\geq n$. Equivalently, the collection $(A[k-1],B[n-k-1])$ is supercritical. 
    Theorem \ref{nefker(1,1)intro} then gives
    \begin{equation*}
        \ker(A^{k-1}\cdot B^{n-k-1})\cap H^{1,1}(X,\mathbb{R})=\Span_{\mathbb{R}}\{[D]:D\in \prim(X), A^{k-1}\cdot B^{n-k-1}\cdot [D]=0\}.
    \end{equation*}
    Consider the quadratic form
    $$q(x,y)=\int_X x\cdot y\cdot A^{k-1}\cdot B^{n-k-1}$$
    on $H^{1,1}(X,\mathbb{R})$. It has at most one positive eigenvalue by the Hodge index theorem. Moreover, the assumption gives
    \begin{equation*}
        q(A,B)^2=q(A,A)q(B,B)>0.
    \end{equation*}
    Hence Lemma \ref{hypextremal} implies that $A-cB \in \ker q$
    where $c=\frac{\int_XA^k\cdot B^{n-k}} {\int_XA^{k-1}\cdot B^{n-k+1}}>0$.
     By Poincar\'e duality, we obtain
    \begin{equation*}
        A^{k-1}\cdot B^{n-k-1}\cdot (A-cB)=0.
    \end{equation*}
    Hence 
    $$A-cB= \sum a_D [D]$$ 
    where the sum ranges over all $D\in \prim(X)$ annihilated by $A^{k-1}\cdot B^{n-k-1}$. 
    This completes the proof.
\end{proof}

\section*{AI disclosure}
All mathematical content is due to the author. 
ChatGPT was used for language polishing and proofreading.

\section*{Acknowledgements}
The author would like to thank his advisor Jian Xiao for his constant encouragement, many helpful discussions and valuable suggestions.
The author also thanks Weizhe Zheng and Enhan Li for helpful discussions.

\bibliography{reference}
\bibliographystyle{amsalpha}

\bigskip

\bigskip

\noindent
\textsc{Tsinghua University, Beijing 100084, China}\\
\noindent
\verb"Email: hujj22@mails.tsinghua.edu.cn"
\end{document}